\documentclass[10pt, a4paper, reqno, oneside]{amsart}
\usepackage[T1]{fontenc}
\usepackage[utf8]{inputenc}
\usepackage[italian, english]{babel}
\usepackage{geometry}
\usepackage{color}
\usepackage[hidelinks]{hyperref}
\usepackage[timezonesep={\ UTC}, showseconds=false]{datetime2}

\usepackage[autostyle]{csquotes}
\usepackage[backend=biber, sorting=nyt, style=numeric, giveninits=true]{biblatex}
\DefineBibliographyExtras{english}{}

\DeclareFieldFormat{postnote}{#1}
\DeclareFieldFormat{multipostnote}{#1}

\usepackage{xurl}

\usepackage{hyphenat}
\usepackage{csquotes}
\usepackage{comment}
\usepackage{enumitem}

\usepackage{amsfonts}
\usepackage{amsmath}
\usepackage{amssymb}
\usepackage{amsthm}
\usepackage{mathrsfs}
\usepackage{stmaryrd}
\usepackage{bm}
\usepackage{mathtools}
\usepackage{braket}

\usepackage{graphicx}
\usepackage{caption}
\usepackage{float}
\usepackage{multicol}
\usepackage{booktabs}
\usepackage{array}

\usepackage{tikz-cd}

\newtheorem{theorem}{Theorem}[section]

\newtheorem{lemma}[theorem]{Lemma}
\newtheorem{proposition}[theorem]{Proposition}
\newtheorem{corollary}[theorem]{Corollary}
\newtheorem{theoremx}{Theorem}

\theoremstyle{definition}
\newtheorem{assumption}[theorem]{Assumption}

\newtheorem{definition}[theorem]{Definition}

\theoremstyle{remark}
\newtheorem{remark}[theorem]{Remark}

\newcommand{\numberset}{\mathbb}
\newcommand{\R}{\numberset{R}}
\newcommand{\C}{\numberset{C}}
\newcommand{\Q}{\numberset{Q}}
\newcommand{\Z}{\numberset{Z}}
\newcommand{\N}{\numberset{N}}
\newcommand{\F}{\numberset{F}}

\newcommand{\Zp}{\Z_p}
\newcommand{\Zq}{\Z_q}

\newcommand{\Qp}{\Q_p}
\newcommand{\Ql}{\Q_\ell}
\newcommand{\Qbar}{\bar{\Q}}

\newcommand{\Qlbar}{\Qbar_\ell}
\newcommand{\Zhat}{\hat{\Z}}

\DeclareFontFamily{U}{wncy}{}
\DeclareFontShape{U}{wncy}{m}{n}{<->wncyr10}{}
\DeclareSymbolFont{mcy}{U}{wncy}{m}{n}
\DeclareMathSymbol{\sha}{\mathord}{mcy}{"58}

\newcommand{\inj}{\hookrightarrow}
\newcommand{\surj}{\twoheadrightarrow}

\newcommand{\iso}{\xrightarrow{\sim}}

\newcommand{\m}{\mathfrak{m}}
\newcommand{\p}{\mathfrak{p}}

\newcommand{\calO}{\mathcal{O}}
\newcommand{\Gtilde}{\tilde{G}}

\DeclareMathOperator{\Imm}{Im}
\DeclareMathOperator{\Hom}{Hom}
\DeclareMathOperator{\End}{End}
\DeclareMathOperator{\Aut}{Aut}

\newcommand{\id}{\mathrm{id}}
\DeclareMathOperator{\length}{length}
\DeclareMathOperator{\proj}{proj}

\DeclareMathOperator{\Gal}{Gal}
\newcommand{\decomp}{\mathcal{D}}
\newcommand{\inertia}{\mathcal{I}}
\newcommand{\Fr}{\mathrm{Fr}}

\DeclareMathOperator{\Tr}{Tr}

\DeclareMathOperator{\Cor}{cor}
\DeclareMathOperator{\Res}{res}
\newcommand{\res}{\Res}
\newcommand{\infl}{\mathrm{infl}}
\DeclareMathOperator{\loc}{loc}
\DeclareMathOperator{\Ind}{Ind}
\DeclareMathOperator{\Sh}{Sh}

\newcommand{\cont}{\mathrm{cont}}
\newcommand{\cohomology}{H}
\newcommand{\ho}{\cohomology^0}
\newcommand{\hone}{\cohomology^1}
\newcommand{\xbar}{\bar{x}}
\newcommand{\ybar}{\bar{y}}
\newcommand{\cohoclass}[1]{\mathbf{#1}}
\newcommand{\abar}{\bar{a}}

\newcommand{\calR}{\mathcal{R}}
\newcommand{\A}{\mathcal{A}}
\newcommand{\T}{\mathcal{T}}
\newcommand{\sfrak}{\mathfrak{s}}
\newcommand{\tfrak}{\mathfrak{t}}
\newcommand{\onefrak}{{\boldsymbol{1}}}
\newcommand{\Rs}{\calR_{\sfrak}}
\newcommand{\Ts}{\T_{\sfrak}}
\newcommand{\Rsa}{\calR_{\sfrak,\alpha}}
\newcommand{\Tsa}{\T_{\sfrak,\alpha}}

\newcommand{\ac}{\mathrm{ac}}
\DeclareMathOperator{\Char}{char}

\newcommand{\calG}{\mathcal{G}}
\newcommand{\admissible}{\mathcal{P}}
\newcommand{\squarefreeadmissible}{\mathcal{N}}
\newcommand{\varadmissible}{\mathcal{L}}
\newcommand{\othervaradmissible}{\mathcal{Q}}
\DeclareMathOperator{\ks}{KS}
\newcommand{\KS}{\boldsymbol{\ks}}
\newcommand{\KSuni}{\overline{\KS}}
\DeclareMathOperator{\es}{ES}
\newcommand{\ES}{\boldsymbol{\es}}
\newcommand{\lcond}{\mathcal{F}}
\newcommand{\koly}{\boldsymbol{\kappa}}
\newcommand{\abf}{\mathbf{a}}
\newcommand{\bbf}{\mathbf{b}}
\newcommand{\cbf}{\mathbf{c}}
\newcommand{\ubf}{\mathbf{u}}

\newcommand{\fs}{\mathrm{fs}}
\newcommand{\tr}{\mathrm{tr}}
\newcommand{\Gr}{\mathrm{Gr}}
\newcommand{\f}{\mathrm{f}}
\newcommand{\s}{\mathrm{s}}
\newcommand{\rel}{\mathrm{rel}}
\newcommand{\ur}{\mathrm{ur}}
\newcommand{\ord}{\mathrm{ord}}
\newcommand{\bk}{\mathrm{BK}}
\newcommand{\honef}{\hone_\f}
\newcommand{\hones}{\hone_\s}
\newcommand{\honeur}{\hone_\ur}
\DeclareMathOperator{\Sel}{Sel}

\newcommand{\ga}{\alpha}

\newcommand{\gd}{\delta}
\newcommand{\e}{\varepsilon}

\newcommand{\gk}{\kappa}
\newcommand{\gl}{\lambda}

\newcommand{\gs}{\sigma}

\newcommand{\gL}{\Lambda}

\newcommand{\sigmati}{\tilde{\sigma}}

\newcommand{\kugasato}{\tilde{\mathcal{E}}}
\DeclareMathOperator{\chow}{CH}
\newcommand{\Kum}{\mathrm{Kum}}

\newcommand{\hp}{\mathrm{HP}}
\newcommand{\Hid}{\mathrm{Hid}}
\newcommand{\gen}{\mathrm{gen}}
\newcommand{\tors}{\mathrm{tors}}
\newcommand{\calW}{\mathcal{W}}
\renewcommand{\div}{\mathrm{div}}

\author{Luca Mastella}
\author{Francesco Zerman}
\date{
}
\address{UniDistance Suisse\\ Schinerstrasse 18, 3900, Brig, Switzerland}
\email{luca.mastella@unidistance.ch}
\email{francesco.zerman@unidistance.ch}

\title{On anticyclotomic Euler and Kolyvagin systems}


\begin{document}

	\begin{abstract}
    We introduce an axiomatization of the notion of ($p\hspace{1pt}$-complete) anticyclotomic Euler system for a wide class of Galois representations, including those attached to a cuspidal eigenform and to a Hida family of modular forms. Under a minimal set of assumptions, we show how to build from these data a universal Kolyvagin system for the representation and for its anticyclotomic twist. Eventually, we recover some applications to the structure of Selmer groups and Iwasawa main conjectures and we review a few concrete examples of these abstract notions that can be found in the literature.
	\end{abstract}

\maketitle

\section{Introduction}

The notion of Euler system and its first applications date back to the groundbreaking work of Kolyvagin \cite{KolFin}, \cite{KolES} and \cite{kolyvagin:structure-sha}. In fact, by studying the properties of the Euler system of Heegner points, he managed to prove important results on the structure of the $p\hspace{1pt}$-part of the Shafarevich--Tate group $\sha(E/K)[p^\infty]$ of an elliptic curve $E/\Q$, over an imaginary quadratic field $K$, with applications to the Birch--Swinnerton-Dyer conjecture in analytic rank $1$. We refer the reader to \cite{gross:kolyvagin} and \cite{mccallum:kolyvagin} as nice accounts of some special cases of Kolyvagin's work. 

Since then, his method has been widely generalized to other arithmetic contexts, where Euler systems have been constructed, in order to `bound the size' of Selmer groups (certain cohomology groups that in the elliptic curve case control the size of the Shafarevich--Tate group) and found applications also in Iwasawa theory, in particular to Iwasawa main conjectures.

At the beginning of this century, Mazur and Rubin \cite{rubin:euler-systems}, \cite{mazur-rubin:kolyvagin-systems} proposed an axiomatization of Kolyvagin's method, introducing a general notion of Euler systems and the concept of Kolyvagin systems in the language of $p\,$-adic Galois representations. Unfortunately, most of their results don't cover the anticyclotomic setting (i.e., the case of Euler systems defined over ring class fields of an imaginary quadratic field $K$), as for instance the original Kolyvagin's Heegner point Euler system. In the literature, this issue was solved case by case, by introducing \emph{ad hoc} notions of anticyclotomic Euler and Kolyvagin systems that are specific for the representation attached to the considered arithmetic object (see e.g.~\cite{howard:heegner-points} for elliptic curves, \cite{castella-hsieh:heegner-cycles}, \cite{longo-vigni:generalized} for modular forms, \cite{H06}, \cite{buyukboduk:big-heegner} for Hida families of modular forms).

The main insight of this article is to give a common axiomatization of the anticyclotomic setting, by introducing the notions of \emph{anticyclotomic Euler system} and of \emph{(modified, universal) Kolyvagin system} for a class of Galois representations that is general enough to include the cases listed above, as we show in Section \ref{sec:examples}. Let's describe more precisely the content of the article.

Let $p$ be an odd prime, $N$ be a positive integer and $K$ be an imaginary quadratic field of discriminant $D_K\ne -3,-4$ coprime with $Np$ and with class number not divisible by $p$. The representations $\T$ considered in this paper are free modules of rank $2$ over a local, complete, Noetherian ring $\calR$ with finite residue field of characteristic $p$, endowed with an $\calR$-linear action of $G_\Q$, unramified outside $Np$, such that the couple $(\T, \calR)$ satisfies some assumptions (Assumptions \ref{ass:assumptions-on-R} and \ref{ass:assumptions-on-T}), including the residual irreducibility of $\T$ and some control on the action of the complex conjugation and the Frobenius automorphisms on $\T$. Notice that we don't require any bound for the Krull dimension $d$ of $\calR$ and any splitting behaviour (other than being unramified) in $K$ for the primes dividing $Np$.

For every integer $n\ge1$, let's denote by $K[n]$ the ring class field of $K$ of conductor $n$. An anticyclotomic Euler system for $\T$ (see Definition \ref{def:euler-systems}) is a particular set of cohomology classes $\cbf(n)\in \hone(K[n],\T\hspace{1pt})$, where $n$ varies in the set $\squarefreeadmissible$ of squarefree products of primes that are inert in $K$ and don't divide $Np$. These are classes that are unramified outside $Np$ and that satisfy a sort of \emph{norm compatibility} \ref{E1} and a \emph{local compatibility} with respect to the Frobenius action \ref{E2}. In Section 4 we will see that, in most known examples, anticyclotomic Euler systems arise geometrically from the arithmetic of Heegner points or cycles. In contrast to what is done in \cite[Definition 7.1]{castella-hsieh:heegner-cycles} (and more in the spirit of \cite{nekovar:euler-system-method}), we don't require any compatibility of the classes $\cbf(n)$ with respect to the action of the complex conjugation, as it can be avoided in the proof of all our main theorems. See Remark \ref{rk:E3} for how our results can be improved by assuming such a condition.

On the other hand, in Section \ref{subsec:modified-Kolyvagin-systems} we introduce the notion of modified universal Kolyvagin system. In fact, the notion of universal Kolyvagin system already exists in the literature (see \cite{mazur-rubin:kolyvagin-systems} and \cite{buyukboduk:deformation-kolyvagin-systems}) and serves as a key ingredient to study the rank and the structure of Selmer groups. A universal Kolyvagin system consists of a compatible set of classes $\koly(n)_\sfrak \in \hone(K, \Ts)$ that are unramified outside $Np$ and have values in a family of quotients $\{\Ts\}_{\sfrak\in \Z_{>0}^2}$ of $\T$, defined in Section \ref{sec:hypotheses}. These classes are indexed on a subset of $\squarefreeadmissible$ and, for every index $n\ell$ with $\ell$ prime, they must satisfy the compatibility
\begin{equation}
	[\loc_\lambda(\koly(n\ell)_\sfrak)]_\s = \phi^\fs_\ell\bigl(\loc_\lambda(\koly(n)_\sfrak)\bigr),\label{eq:K2-intro-2}
\end{equation}
where $\loc_\lambda$ denotes the localization map at the unique prime $\lambda$ of $K$ above $\ell$, $\phi_\ell^\fs$ is the finite-singular isomorphism (see Definition \ref{def:finite-singular}) and $[\, \ast\, ]_\s$ denotes the projection of $\ast$ to the singular quotient (see Section \ref{sec:local-conditions}). However, in our generality, we are not able to show that the existence of an anticyclotomic Euler system implies the existence of a universal Kolyvagin system. What we obtain is a \emph{modified} version of this notion, where the relation \eqref{eq:K2-intro-2} must be twisted by an automorphism $\chi_{n,\ell}$ of $\T$ (possibly depending on the couple $(n,\ell)$). See Section \ref{subsec:modified-Kolyvagin-systems} for more details on the definition of a modified universal Kolyvagin system. Therefore, the main result of Section \ref{sec:euler-systems-and-kolyvagin-systems} is the following.

\begin{theoremx}[Theorem \ref{th:euler-to-kolyvagin-system}]\label{thm:A}
	If $\{\cbf(n)\}_{n\in\squarefreeadmissible}$ is an anticyclotomic Euler system for $\T$, then there is a subset $\squarefreeadmissible'$ of $\squarefreeadmissible$ and a modified universal Kolyvagin system \[
	\{\koly(n)_\sfrak : \sfrak \in \Z_{>0}^2,  n\in\squarefreeadmissible'\}
	\]
	for $\T$ such that $\Cor^{K[1]}_{K}\cbf(1)=\varprojlim_\sfrak \koly(1)_\sfrak$.
\end{theoremx}

In fact, Theorem \ref{th:euler-to-kolyvagin-system} is even more precise and its proof is constructive, in the sense that we explicitly build the subset $\squarefreeadmissible'$ (see Section \ref{sec:finer-choice-primes}) and the derivative classes $\koly(n)_\sfrak$ that form the modified universal Kolyvagin system. These are constructed by applying a careful variation of Kolyvagin's descent argument (see Sections \ref{sec:Derivative-classes-and-local-properties} and \ref{sec:from-Euler-to-Kolyvagin-systems}). In Section \ref{sec:applications} we take care that, as far as we are concerned with the classical applications of the Kolyvagin system machinery (i.e., studying the structure of Selmer groups), a modified universal Kolyvagin system has exactly the same usage as a classical (universal) Kolyvagin system.

Let now $K_\infty$ be the anticyclotomic $\Zp$-extension of $K$ and let $K_\ga$ be its $\ga$-th layer. One of the main topics of Iwasawa theory is the study of Iwasawa Selmer groups, that are submodules of $\varprojlim_\ga \hone(K_\ga,\T\hspace{1pt})$. By Shapiro's lemma (see Lemma \ref{lem:Shapiro application}), 
one can shift the above problem to the study of the arithmetic of the anticyclotomic twist $\T^{\ac}=\T \hat{\otimes}_{\Zp} \Zp\llbracket\Gal(K_\infty/K)\rrbracket$ of $\T$. 

This is what we do in Section \ref{sec:Iwasawa-theory}, where we follow the same pattern of Section \ref{sec:euler-systems-and-kolyvagin-systems} for the $G_K$-representation $\T^\ac$ in place of $\T$. In particular, in Section \ref{sec:p-complete-euler-systems} we give the definition of a \emph{$p\hspace{1pt}$-complete anticyclotomic Euler system} for $\T$ (Definition \ref{def:p-complete-euler-systems}). The main difference with respect to the anticyclotomic Euler systems of Definition \ref{def:euler-systems} is that now the cohomology classes are allowed to vary among the set $\squarefreeadmissible^{(p)}$ of products of elements of $\squarefreeadmissible$ and powers of $p$, and that we require the `vertical' compatibility \ref{pE0} when varying the exponent of $p$. In the same section, we also define the notion of modified universal Kolyvagin system for $\T^\ac$, that coincides with the same notion given on $\T$ with the difference that we have to work on a suitable class of quotients $\{\T_{\sfrak,\ga}^\ac\}_{\sfrak \in \Z_{>0}^2, \alpha \in \N}$ of $\T^\ac$, defined in Section \ref{sec:the-anticyclotomic-twist}. Therefore, the classes forming such a modified universal Kolyvagin system will be elements $\koly(n)^\ac_{\sfrak,\ga}\in \hone(K,\T_{\sfrak,\ga}^\ac)$ satisfying suitable compatibility properties. As a parallel to Theorem \ref{thm:A}, the main result of Section \ref{sec:Iwasawa-theory} will be the following.

\begin{theoremx}[Theorem \ref{th:Iwasawa-euler-to-kolyvagin-system}]\label{thm:B}
	If $\{\bbf(np^\ga)\}_{np^\ga\in\squarefreeadmissible^{(p)}}$ is a $p\hspace{1pt}$-complete anticyclotomic Euler system for $\T$, then there exists a modified universal Kolyvagin system 
	\[
	\{\koly(n)^{\ac}_{\sfrak,\ga}: \sfrak \in \Z_{>0}^2, \alpha \in \N,n\in\squarefreeadmissible'\}
	\]
	for $\T^\ac$, such that $\Cor^{K[1]}_{K}\varprojlim_\ga\bbf(p^\ga)=\varprojlim_{\sfrak,\ga} \koly(1)^\ac_{\sfrak,\ga}$.
\end{theoremx}

The arguments used for the proof of Theorem \ref{thm:B} follow closely those used for Theorem \ref{thm:A}. In Section \ref{sec:Iwasawa-applications} we recover some classical applications of the existence of a non-trivial (modified, universal) Kolyvagin system for $\T^\ac$ to the Iwasawa theory of the representation $\T$. 

In Section \ref{sec:examples} we recover some explicit examples of our theory in the cases of elliptic curves, higher weight modular forms and Hida families. There, we show how our main theorems imply some known results on the structure of some relevant Selmer groups and on the Iwasawa main conjecture attached to these arithmetic objects. We also hope that our general theory can be applied to other arithmetic objects where these results are not completely known, e.g.~the case of Coleman families families of modular forms (see \cite[Remark 6.4.4]{jetchev-loeffler-zerber:heegner-coleman}). We will not address this problem here for space reasons, but it will be an interesting direction for future work.

Finally, in the appendices we recall some technical details: in Appendix \ref{sec:formula-abstract}  we review and generalize a key formula of Nekov\' a\v r that we use in the proof of Theorems \ref{thm:A} and \ref{thm:B}, and in Appendix \ref{sec:semi-local-cohomology} we list some useful results on semi-local cohomology groups.

\subsection{Notation}\label{subsec:notation}

If $F$ is a perfect field, we write $\bar{F}$ for a fixed algebraic closure of $F$ and $G_F:=\Gal(\bar{F}/F)$ for its absolute Galois group. In particular, fix an algebraic closure $\Qbar$ of $\Q$ and, for any rational prime $\ell$, an algebraic closure $\Qbar_\ell$ of $\Q_\ell$. In addition, fix embeddings $\Qbar\inj\C$ and $\Qbar\inj\Qlbar$. Denote by $\tau_c\in\Gal(\C/\R)$ the complex conjugation and use the same symbol for the elements of $G_\Q$ and $G_{\Q_\ell}$ associated with $\tau_c$ via the chosen embeddings. 

For any non-archimedean prime $v$ of a number field $F$, denote by $\Fr_v$ the set of elements of $G_F$ that lie in some decomposition group at $v$ and that reduce to the Frobenius automorphism modulo the corresponding inertia at $v$. 
If $M$ is a $G_F$-module unramified at $v$, the image of $\Fr_v$ in $\Aut(M)$ is well defined as a conjugacy class of elements of $\Aut(M)$, therefore the trace, the determinant and the characteristic polynomial of $\Fr_v$ are independent of the actual choice of the Frobenius element. 

For a topological group $G$, a closed subgroup $H$ of $G$ and a $G$-module $M$, denote by $\res^H_G$, $\Cor^H_G$ and, when $H$ is normal in $G$, $\infl^H_G$ respectively the restriction, corestriction and inflation morphisms on first cohomology groups with values in $M$. If $F'/F$ is an extension of global or local fields, $G = G_F$ and $H = G_{F'}$, write also $\res^{F'}_{F}$, $\Cor^{F'}_F$ and $\infl^{F'}_F$. Finally, for a number field $F$ and a place $v$ of $F$, denote by $\loc_{v}$ (called \textit{localization} at $v$) the restriction $\res^{G_F}_{G_{F_v}}$, obtained by identifying $G_{F_v}$ with a (fixed) decomposition group of $F$ at $v$. Moreover, this identification allows us to see $\Fr_v$ as a well defined coset of the inertia in $G_{F_v}$.

\subsection{Acknowledgments}
We would like to thank Matteo Longo and Stefano Vigni for enlightening discussions on the topics of this paper  and the anonymous referee for helpful comments and suggestions on an earlier version of the article, which led to many improvements. The first named author was partially supported, when writing this article, by PRIN 2022, \emph{The arithmetic of motives and L-functions}. The second named author was supported by the ERC Consolidator Grant \emph{ShimBSD: Shimura varieties and the BSD conjecture}.

\section{Euler systems and Kolyvagin systems}\label{sec:euler-systems-and-kolyvagin-systems}

In this section we define the notion of anticyclotomic Euler system, inspired by \cite{castella-hsieh:heegner-cycles}, for a general family of Galois representations. We then show how one can use the arguments explained in Appendix \ref{sec:formula-abstract}, that rely on the work of \cite{nekovar:chow-groups}, in order to build a universal modified Kolyvagin system, starting from an anticyclotomic Euler system.

\subsection{The arithmetic picture}\label{subsec:the-arithmetic-picture} 

Let $(\calR,\m)$ be a complete local Noetherian ring with finite residue field of characteristic $p>2$ and let $K$ be an imaginary quadratic field of discriminant $D_K$. Let $\T$ be a finite and free $\calR$-module equipped with a continuous linear $G_K$-action. Let $N$ be a natural number coprime with $p$ such that $\T$ is unramified outside $Np$. 

\begin{assumption}\label{ass:class-number-K}
	For the rest of the article we assume that
	\begin{enumerate}[label=(\alph*)]
		\item $D_K\notin\{-3,-4\}$;
		\item $Np$ is coprime with $D_K$;
		\item the class number of $K$ is not divisible by $p$.\label{ass-cond:class-number-K}
	\end{enumerate}
\end{assumption}

For every $n\ge 1$, denote by $K[n]$ the ring class field of $K$ of conductor $n$, by $K(n)$ the maximal $p\,$-extension of $K$ contained in $K[n]$ and set $\calG_n:=\Gal(K[n]/K[1])$. Thanks to Assumption \ref{ass:class-number-K} \ref{ass-cond:class-number-K}, the index $[K(n):K]$ coincides with the $p\,$-part of $[K[n]:K[1]]$.

\begin{definition}\label{def:set-P}
	Let $\admissible$ be the set of all prime numbers $\ell\nmid Np$ that are inert in $K$ and denote by $\squarefreeadmissible$ the set of all squarefree products of elements of $\admissible$, with the convention that $1\in\squarefreeadmissible$.
\end{definition}

It is well known that, for every $n\in\squarefreeadmissible$, there is an isomorphism $\calG_n\cong \prod_{\ell\mid n} \calG_\ell$, where the product runs over all primes $\ell$ dividing $n$. Also, $\calG_\ell$ is cyclic of order $\ell+1$, hence we fix a generator $\sigma_\ell$. 

By class field theory, the prime $\lambda=(\ell)$ of $K$ splits completely in $K[m]$ whenever $\ell\nmid m$. Moreover, every prime $\gl_1$ of $K[1]$ above $\lambda$ is totally ramified in $K[\ell]$ and therefore we have that $\calG_\ell\cong \Gal(K[\ell]_{\lambda_\ell}/K[1]_{\lambda_1})\cong\Gal(K[\ell]_{\gl_\ell}/K_\gl)$, where $\gl_\ell$ is the prime of $K[\ell]$ above $\gl_1$. This implies that $\lambda$ is totally ramified in $K(\ell)$ and that $\Gal(K(\ell)_{\lambda_\ell'}/K_\gl)$ is isomorphic to the $p\,$-Sylow subgroup of $\calG_\ell$, where $\lambda_\ell'=\lambda_\ell\cap K(\ell)$.

\begin{lemma}\label{lem:maximal-tamely-ram-ext}
	Let $\ell\in\admissible$ such that $p\mid \ell+1$ and let $\lambda_\ell'$ be the prime of $K(\ell)$ above the prime  $\lambda=(\ell)$ of $K$. The extension $K(\ell)_{\gl_\ell'}/K_{\lambda}$ is a maximal totally tamely ramified abelian $p\,$-extension of $K_\gl$.
\end{lemma}

\begin{proof}
	Local class field theory (see the proof of \cite[Theorem 7.9]{neukirch:cft} and \cite[Corollary 7.18]{neukirch:cft}) implies that every abelian totally ramified extension of $K_\gl$ has degree that divides $\ell^{2n}(\ell^2-1)$, for some $n$. Therefore, maximal totally tamely ramified abelian extensions of $K_\gl$ have at most degree $\ell^2-1$. Since $p\mid \ell+1$ and $p>2$, the $p\,$-part of $\ell^2-1$ is equal to the $p\,$-part of $\ell+1$, which is equal to the cardinality of the $p\,$-Sylow of $\calG_\ell$, hence it coincides with the cardinality of $\Gal(K(\ell)_{\gl_\ell'}/K_{\lambda})$.
\end{proof}

\begin{remark}\label{rk:maximal-tamely-ramified-extension}
	For a prime $\ell \in \admissible$, we will be interested in the structure of the wild and tame inertia groups at $\ell$. We refer the reader to \cite[Section VII.5]{neukirch:cnf} for a proof of the following facts. The wild inertia group $\inertia_\ell^w := \Gal(\Qbar_\ell/\Q_\ell^t)$ is a free pro\hspace{1pt}-$\ell$ group and denote by $\inertia_\ell^t:=\Gal(\Q_\ell^t/\Q_\ell^{\ur})$ the tame inertia group at $\ell$. One has the decomposition
	\[
	\Gal(\Q_\ell^t/\Q_\ell) = \inertia_\ell^t \rtimes \Gal(\Q_\ell^{\ur}/\Q_\ell),
	\]
	where $\Fr_\ell$, seen as the profinite generator of $\Gal(\Q_\ell^{\ur}/\Q_\ell)$, satisfies the relation $\Fr_\ell \hspace{1pt} \tau \hspace{1pt} \Fr_\ell^{-1} = \tau^{\ell}$ for every element $\tau\in\inertia_\ell^t$.
	
	Let now $\gl=(\ell)$ be the prime of $K$ above $\ell$. Since $\ell$ is inert in $K$, it follows that $K_\lambda^\ur = \Q_\ell^\ur$, $K_\lambda^t = \Q_\ell^t$ and $\Fr_\lambda = \Fr_\ell^2$. Therefore $\Gal(\Q_\ell^t/K_\lambda) = \inertia_\ell^t \rtimes \Gal(\Q_\ell^{\ur}/K_\lambda)$,
	where $\Fr_\gl \tau \hspace{1pt}\Fr_\gl^{-1} = \tau^{\ell^2}$ for every $\tau\in \inertia_\ell^t$. The identifications
	$\Gal(\Q_\ell^{\ur}/\Q_\ell)\cong\Zhat$ and $\inertia_\ell^t\cong\prod_{q \ne \ell} \Zq$ yield isomorphisms
	\begin{equation*}
		\Gal(\Q_\ell^t/\Q_\ell)\cong \prod_{q \ne \ell} \Zq\rtimes \Zhat\quad\text{and}\quad \Gal(\Q_\ell^t/K_\gl)\cong \prod_{q \ne \ell} \Zq\rtimes 2\Zhat.
	\end{equation*}
\end{remark}

\subsection{Selmer structures}\label{sec:local-conditions}

For a finite extension $L$ of $K$ and a finite place $v$ of $L$, a \textit{local condition} $\lcond$ on $\T$ at $v$
is the choice of an $\calR$-submodule $H_{\lcond}^1(L_v,\T\hspace{1pt})\subseteq \hone(L_v,\T\hspace{1pt})$. We recall some relevant local conditions.
\begin{enumerate}[label=(\alph*)]
	\item The \emph{relaxed} local condition on $\T$ at $v$ corresponds to the choice of the whole 
	$\hone(L_v,\T\hspace{1pt})$, whereas the \emph{zero} local condition to the choice of $\{0\}\subseteq \hone(L_v,\T\hspace{1pt})$.
	
	\item If $\T$ is unramified at $v$, the \emph{finite} local condition on $\T$ at $v$ is
	\begin{equation*}
		\honef(L_v,\T\hspace{1pt}):=\ker\bigl(\hone(L_v,\T\hspace{1pt})
		\xrightarrow{\res}\hone(L_v^{\ur},\T\hspace{1pt})\bigr).
	\end{equation*}
	We also denote by $\hones(L_v,\T\hspace{1pt})$ the \emph{singular quotient} $\hone(L_v,\T\hspace{1pt})/\honef(L_v,\T\hspace{1pt})$ and, for any class $\alpha \in \hone(L_v, \T\hspace{1pt})$, by $[\alpha]_\s$ its projection to $\hones(L_v, \T\hspace{1pt})$.
	\item When $v\nmid p$, the \emph{$F$-transverse} local condition on $\T$ at $v$ is 
	\begin{equation*}
		\hone_{F-\tr}(L_v,\T\hspace{1pt}):=\ker\bigl(\hone(L_v,\T\hspace{1pt})\xrightarrow{\res} \hone(F,\T\hspace{1pt})\bigr),
	\end{equation*}
	where $F$ is a maximal totally tamely ramified abelian $p\,$-extension of $L_v$. In this paper, when $L=K$ and $v=(\ell)$ for $\ell\in\admissible$ such that $p\mid\ell+1$, we fix $F=K(\ell)_{\gl_\ell'}$ as in Lemma \ref{lem:maximal-tamely-ram-ext}. We then set $\hone_{\tr}(K_v,\T\hspace{1pt}):=\hone_{F-\tr}(K_v,\T\hspace{1pt})$ and refer to it just as the \emph{transverse} local condition.
\end{enumerate}

\begin{definition}\label{def:selmer-structures}
	A \textit{Selmer structure} $\lcond$ on $\T$ over $L$ is a choice of local conditions at every finite place $v$ of $L$ such that $H_{\lcond}^1(L_v,\T\hspace{1pt})=\honef(L_v,\T\hspace{1pt})$ for all the primes $v$ of $L$ such that $v \nmid Np$.
\end{definition}

\begin{definition}
	Given a Selmer structure $\lcond$ on $\T$ over $L$, the \emph{Selmer module} of $\T$ over $L$ with respect to $\lcond$ is the $\calR$-module
	\[
	\hone_{\lcond}(L, \T\hspace{1pt}) := \ker\biggl(\hone(L, \T\hspace{1pt}) \to \prod_{v} \frac{\hone(L_v, \T\hspace{1pt})}{\hone_{\lcond}(L_v, \T\hspace{1pt})} \biggr), 
	\]
	where $v$ varies over all finite places of $L$ and the map is the product of the localizations $\loc_v$.
\end{definition}

\begin{remark}
	If $v$ is an archimedean prime of $L$, then $L_v=\C$ and so we have that $\hone(L_v,\T\hspace{1pt})=\{0\}$. Therefore, the localization at archimedean primes will never play a role in our arguments.
\end{remark}

\subsection{The finite-singular isomorphism}
For this section, fix $\ell\in\admissible$, call $\lambda=(\ell)$ the prime of $K$ above $\ell$ and suppose that $(\ell+1)\T=\{0\}$.

\begin{lemma}\label{lem:isomorfismi-finito-singolare}
	There are canonical functorial isomorphisms 
	\begin{equation*}
		\alpha_\ell\colon \honef(K_\gl,\T\hspace{1pt}) 
		\xrightarrow{\cong}
		\T/(\Fr_\gl-1)\T \quad\text{and}\quad \beta_\ell\colon \hones(K_\gl,\T\hspace{1pt})\otimes \calG_\ell\xrightarrow{\cong} \T^{\Fr_\gl=1},
	\end{equation*}
	where $\ga_\ell$ is induced by evaluating cocycles at $\Fr_\lambda$ and $\beta_\ell$ is induced by $\xi\otimes \gs\mapsto \xi(\sigmati)$ for any lift $\sigmati\in \Gal(\Qlbar/\Q_\ell^{\ur})$ of $\gs$.
\end{lemma}
\begin{proof}
	The isomorphism $\ga_\ell$ comes from point (i) of \cite[Lemma 1.2.1]{mazur-rubin:kolyvagin-systems}. Although the isomorphism $\beta_\ell$ is a direct consequence of point (ii) of \cite[Lemma 1.2.1]{mazur-rubin:kolyvagin-systems}, we will give here a more explicit construction, that better fits our applications.
	
	By the inflation-restriction exact sequence, together with the fact that $\T$ is unramified at $\gl$, it follows that
	\[
	\hones(K_\lambda, \T\hspace{1pt}) \cong \hone(\Q_\ell^\ur, \T\hspace{1pt})^{\Fr_\lambda = 1} = \Hom(\inertia_\ell, \T\hspace{1pt})^{\Fr_\lambda=1},
	\]
	where $\inertia_\ell$ is the absolute inertia group at $\ell$, that coincides with the absolute inertia group at $\gl$ as $K_\gl/\Q_\ell$ is unramified. Since $\inertia_\ell^w$ is pro\hspace{1pt}-$\ell$ and $\T$ is pro\hspace{1pt}-$p$,  
	we get 
	\begin{equation*}
		\Hom(\inertia_\ell, \T\hspace{1pt})^{\Fr_\lambda=1}=\Hom(\inertia_\ell^t, \T\hspace{1pt})^{\Fr_\lambda=1}=\Hom\big(\inertia_\ell^t/(\inertia_\ell^t)^{\ell+1}, \T\big)^{\Fr_\lambda=1},
	\end{equation*}
	using the fact that $(\ell+1)\T=0$ in the last equality. Moreover,
	\begin{equation*}
		\Hom\big(\inertia_\ell^{t}/(\inertia_{\ell}^{t})^{\ell+1}, \T\big)^{\Fr_\lambda=1}=\Hom\big(\inertia_\ell^{t}/(\inertia_\ell^{t})^{\ell+1}, \T^{\Fr_\lambda=1}\big),
	\end{equation*}
	since for any $\phi\in \Hom(\inertia_\ell^{t}/(\inertia_{\ell}^{t})^{\ell+1}, \T\hspace{1pt})$ and $[\tau] \in \inertia_\ell^{t}/(\inertia_\ell^{t})^{\ell+1}$ one has
	\[
	(\Fr_\lambda \, \phi)\bigl([\tau]\bigr) = 
	\Fr_\lambda \cdot \, \phi\bigl(\Fr_\lambda [\tau] \Fr_\lambda^{-1}\bigr) = 
	\Fr_\lambda \cdot \, \phi\bigl([\tau^{\ell^2}]\bigr) = \Fr_\lambda \cdot \, \ell^2\phi\bigl([\tau]\bigr) = \Fr_\lambda \cdot \, \phi\bigl([\tau]\bigr).
	\] 
	
	Let now $\gl_\ell$ be a prime of $K[\ell]$ above $\gl$. Since the extension $K[\ell]_{\gl_\ell}/K_\gl$ is totally ramified, there is a canonical isomorphism $\Gal(\Q_\ell^{\ur} K[\ell]_{\gl_\ell}/\Q_\ell^{\ur})\cong \calG_\ell$, obtained by restricting Galois automorphisms. Since $(\inertia_{\ell}^t)^{\ell+1}$ is the unique closed subgroup of index $\ell+1$ inside $\inertia_{\ell}^t$,
	
	one has that $\inertia_\ell^{t}/(\inertia_{\ell}^{t})^{\ell+1}=\Gal(\Q_\ell^{\ur} K[\ell]_{\gl_\ell}/\Q_\ell^{\ur})$. Therefore,
	\begin{equation*}
		\Hom\big(\inertia_\ell^{t}/(\inertia_{\ell}^{t})^{\ell+1}, \T^{\Fr_\lambda=1}\big)\cong \Hom(\calG_\ell, \T^{\Fr_\lambda=1}).
	\end{equation*}
	We conclude by noticing that there is a canonical isomorphism   \[
	\Hom(\calG_\ell, \T^{\Fr_\lambda=1})\otimes\calG_\ell\iso \T^{\Fr_\lambda=1}
	\]
	obtained by sending $\xi\otimes\sigma$ to $\xi(\sigma)$.
\end{proof}

\begin{remark}
	Since $(\ell+1)\T=\{0\}$, there is an isomorphism 
	\[
	\hones(K_\gl,\T\hspace{1pt})\otimes \calG_\ell\cong
	\hones(K_\gl,\T\hspace{1pt}).\] 
	However, this isomorphism is not canonical, depending on the choice of a generator for $\calG_\ell$. 
\end{remark}

\begin{definition}\label{def:finite-singular}
	Whenever $\Fr_\gl=\Fr_\ell^2$ acts trivially on $\T$, the isomorphism
	\begin{equation*}
		\phi_\ell^{\fs}:=\beta_\ell^{-1}\circ \alpha_\ell\colon \honef(K_\gl,\T\hspace{1pt}) \xrightarrow{\cong} \T \xrightarrow{\cong} \hones(K_\gl,\T\hspace{1pt})\otimes \calG_\ell
	\end{equation*}
	is called the \emph{finite-singular isomorphism}. 
\end{definition}

Since $(\ell+1)\T=0$, by \cite[Lemma 1.2.4]{mazur-rubin:kolyvagin-systems} there is a functorial splitting
\begin{equation*}
	\hone(K_\gl,\T\hspace{1pt})=\honef(K_\gl,\T\hspace{1pt})\oplus H_{\tr}^1(K_\gl,\T\hspace{1pt})
\end{equation*}
and $H_{\tr}^1(K_\gl,\T\hspace{1pt})$ projects isomorphically to $\hones(K_\gl,\T\hspace{1pt})$. By Lemma \ref{lem:maximal-tamely-ram-ext}, there is the explicit description $\hone_{\tr}(K_\gl,\T\hspace{1pt})=\ker(\hone(K_\gl,\T\hspace{1pt})\to \hone(K(\ell)_{\lambda_\ell'},\T\hspace{1pt}))$. Using inflation-restriction together with the fact that $\Gal(K[\ell]/K(\ell))$ has order coprime with $p$, one can also show that
\begin{equation}\label{equ:easy-description-of-transverse-condition}
	\hone_{\tr}(K_\gl,\T\hspace{1pt})=\ker\bigl(\hone(K_\gl,\T\hspace{1pt})\xrightarrow{\res}
	\hone(K[\ell]_{\lambda_\ell},\T\hspace{1pt})\bigr).
\end{equation}

\subsection{Action of the Frobenius}

Assume from now on that the action of $G_K$ on $\T$ extends to a continuous linear action of $G_\Q$. We study here in detail the interaction between the action of the Frobenius $\Fr_\ell$ and the morphisms $\alpha_\ell$, $\beta_\ell$ and $\phi_\ell^{\fs}$ in local Galois cohomology. 

Fix $\ell\in\admissible$, call $\lambda=(\ell)$ the prime of $K$ above $\ell$ and suppose that $(\ell+1)\T=0$. In this setting, there is an action of the group $\Gal(K_\lambda/\Q_\ell)$, cyclic of order $2$ generated by the image of $\Fr_\ell$, on $\honef(K_\gl,\T\hspace{1pt})$ and $\hones(K_\gl,\T\hspace{1pt})$, induced by its natural action on $\hone(K_\gl,\T\hspace{1pt})$. Also, there is a natural action of $\Gal(K_\lambda/\Q_\ell)$ on $\T/(\Fr_\gl-1)\T$ and $\T^{\Fr_\gl=1}$. 

\begin{lemma}\label{lem:commutativity-complex-conj-fs-isomorphism}
	In the setting above, we have
	\begin{enumerate}[label=\emph{(\roman*)}]
		\item $\alpha_\ell\circ \Fr_\ell=\Fr_\ell\circ\,\alpha_\ell\hspace{1pt};$\label{item:alpha-complex-conj}
		\item $\beta_\ell\circ \Fr_\ell=-\Fr_\ell\circ \, \beta_\ell\hspace{1pt};$\label{item:beta-complex-conj}
		\item $\phi_\ell^{\fs}\circ \Fr_\ell=-\Fr_\ell\circ \, \phi_\ell^{\fs}$, whenever $\Fr_\gl=\Fr_\ell^2$ acts trivially on $\T$.\label{item:finite-sing-complex-conj}
	\end{enumerate}
\end{lemma}
\begin{proof}
	\emph{\ref{item:alpha-complex-conj}} By the inflation-restriction exact sequence and the fact that $\T$ is unramified at $\ell$, we have $\honef(K_\gl, \T\hspace{1pt})\cong \hone(\Q_\ell^{\ur}/K_{\gl},\T\hspace{1pt})$. Moreover, for every class $[\xi]\in \hone(\Q_\ell^{\ur}/K_{\gl},\T\hspace{1pt})$ one has
	\begin{equation*}
		(\Fr_\ell\xi)(\Fr_\gl)=  \Fr_\ell\cdot\,\xi(\Fr_\ell\Fr_\gl\Fr_\ell^{-1})=\Fr_\ell\cdot\,\xi(\Fr_\gl),
	\end{equation*}
	i.e., $(\alpha_\ell\circ \Fr_\ell)[\xi]$ coincides with $(\Fr_\ell\circ\,\alpha_\ell)([\xi])$ in $\T/(\Fr_\gl-1)\T$.

	\emph{\ref{item:beta-complex-conj}} 
	As seen in the proof of Lemma \ref{lem:isomorfismi-finito-singolare}, $\hones(K_\lambda, \T\hspace{1pt}) \cong \Hom(\inertia_\ell^t, \T\hspace{1pt})^{\Fr_\gl=1}$, where $\inertia_\ell^t$ is the tame inertia group at $\ell$.
	Let $\gs\in\calG_\ell$ and choose a lift $\tilde{\gs}$ of $\gs$ to $\inertia_\ell^t$, which satisfies the relation $\Fr_\ell\tilde{\gs}\Fr_\ell^{-1}=\tilde{\gs}^\ell$ as noted in Remark \ref{rk:maximal-tamely-ramified-extension}. Hence, for every $\xi\in \Hom(\inertia_\ell^t,\T\hspace{1pt})^{\Fr_\gl=1}$, we obtain that
	\begin{equation*}
		(\beta_\ell\circ \Fr_\ell)(\xi\otimes\sigma)=
		(\Fr_\ell\xi)(\tilde{\gs})= \Fr_\ell\cdot\,\xi(\Fr_\ell\tilde{\gs}\Fr_\ell^{-1})= \Fr_\ell\cdot\,\xi(\tilde{\gs}^\ell)= \Fr_\ell\cdot\,\ell\xi(\tilde{\gs})= 
		-\Fr_\ell\cdot\,\xi(\tilde{\gs}),
	\end{equation*}
	where in the last equality we used the fact that $(\ell+1)\T=0$. Since $\Fr_\ell\cdot\,\xi(\tilde{\gs})=(\Fr_\ell\circ \beta_\ell)(\xi \otimes \sigma)$, we obtain the claim.
	
	Point \emph{\ref{item:finite-sing-complex-conj}} is an immediate consequence of \emph{\ref{item:alpha-complex-conj}} and \emph{\ref{item:beta-complex-conj}}.
\end{proof}

\subsection{Hypotheses}\label{sec:hypotheses} 

We now make a more precise choice for the ring $\calR$ and the representation $\T$ by imposing some assumptions on them, that will be in force until the end of the paper. We tried to keep them as minimal as possible in order to comprehend all known examples (see Section \ref{sec:examples}) and leave the door open to future applications of this method.

\begin{assumption}\label{ass:assumptions-on-R}
	There is a descending chain of ideals $J_1 \supseteq J_2 \supseteq \dots$ of $\calR$ such that, denoting, for any $\sfrak = (s_1, s_2) \in \Z_{>0}^2$, by $I_\sfrak$ the $\calR$-ideal generated by $p^{s_1}$ and $J_{s_2}$,
	\begin{enumerate}[label=(\roman*)]
		\item $\calR/J_{s_2}$ is a finite and free $\Zp$-module;\label{ass-cond:J-ideals}
		\item $\calR\cong\varprojlim_{\sfrak}\calR/I_\sfrak$, where the limit is taken with respect to the partial order on $\Z_{>0}^2$ given by the rule $(s_1,s_2)\preceq (t_1,t_2)$ if $s_i\le t_i$ for every $i=1,2$.\label{ass-cond:limit-R}
	\end{enumerate}
\end{assumption}

\begin{assumption}\label{ass:assumptions-on-T}
	The $\calR\llbracket G_\Q\rrbracket $-module $\T$ has the following properties:
	\begin{enumerate}[label=(\roman*)]
		\item it is free of rank $2$ as an $\calR$-module;\label{ass-cond:T-free-rank-2}
		\item the residual $G_\Q$-representation $\T/\m_{\calR}\T$ is irreducible;\label{condition:assumption-irreducible-residual-representation}
		\item the action of $\Fr_\ell$ on $\T$ has determinant $\ell$ for every prime $\ell\nmid Np$;\label{condition:assumption-characteristic-polynomial-Frobenius}
		\item the eigenvalues $1$ and $-1$ of the action of the complex conjugation $\tau_c$ on $\T$ both have $1$-dimensional eigenspaces;\label{condition:assumption-eigenvalues-complex-conjugation}
		\item the image of $G_\Q$ in $\Aut_{\calR}(\T\hspace{1pt})$ contains the scalars $1+p\Zp$.\label{condition:big-image-assumption}
	\end{enumerate}
\end{assumption}

\begin{remark}
	As noted in \cite[(1.5.3)(3)]{NP}, point \ref{condition:assumption-irreducible-residual-representation} and point \ref{condition:assumption-eigenvalues-complex-conjugation} together imply that the residual representation $\T/\m_{\calR}\T$ is absolutely irreducible. Point \ref{condition:assumption-eigenvalues-complex-conjugation} often descends from the existence of a perfect alternating bilinear pairing $\T\times\T\to\calR(1)$ and it implies that the characteristic polynomial of $\tau_c$ on $\T$ is $X^2-1$. Point \ref{condition:big-image-assumption} is a big image assumption.
\end{remark}

For every $\sfrak=(s_1,s_2)\in \Z_{>0}^2$, define 
\begin{equation*}
	\Rs:=\calR/I_\sfrak \quad\text{and}\quad  \Ts:=\T\otimes_{\calR}\Rs.
\end{equation*}
Notice that if $\sfrak\preceq \tfrak$, then $\Rs$ and $\Ts$ are quotients of $\calR_{\tfrak}$ and $\T_{\tfrak}$, respectively.
As a consequence of Assumption \ref{ass:assumptions-on-R} \ref{ass-cond:J-ideals}, we have that $\Rs$ is finite of $p\,$-power order and by Assumption \ref{ass:assumptions-on-T} \ref{ass-cond:T-free-rank-2}, it follows that also $\Ts$ is finite of $p\,$-power order and that $\T\cong\varprojlim_\sfrak \Ts$. 
Moreover, we will also use the notation
\begin{equation*}
	\Rs':=\calR/J_{s_2} \quad\text{and}\quad  \Ts':=\T\otimes_{\calR}\Rs'.
\end{equation*}

In the following, if $F$ is a number field and $\rho\colon G_F\to\Aut(T)$ is a Galois representation, we will denote by $F(T)$ the extension of $F$ cut out by $\ker\rho$.

\begin{definition}\label{dfn:s-admissible-primes}
	Let $\sfrak=(s_1,s_2)\in \Z_{>0}^2$.
	\begin{enumerate}[label=(\alph*)]
		\item Define $\admissible_{\sfrak}$ to be the set of all the primes $\ell\in\admissible$ such that the conjugacy class $\Fr_\ell$ coincides with the conjugacy class of the complex conjugation $\tau_c$ in $\Gal(K(\Ts)/\Q)$.
		\item Define $\squarefreeadmissible_{\sfrak}$ to be the set of all squarefree products of elements of $\admissible_{\sfrak}$, with the convention that $1\in\squarefreeadmissible_\sfrak$.
	\end{enumerate}
\end{definition}

By Chebotarev's density theorem, each set $\admissible_{\sfrak}$ consists of infinitely many primes and $\admissible_{\sfrak}\supseteq \admissible_{\tfrak}$ whenever $\sfrak\preceq \tfrak$. For every $\ell\in\admissible_{\sfrak}$, comparing $\Fr_\ell$ and $\tau_c$ through their characteristic polynomials, we obtain that
\begin{itemize}
	\item $\Fr_{\ell}^2$ acts as the identity on $\Ts$;
	\item $p^{s_1}$ divides $ \ell+1$ and $\Tr(\hspace{1pt}\Fr_\ell \hspace{1pt} \vert \hspace{1pt} \T \hspace{1pt})$ in $\Rs'$.
\end{itemize}
In particular, calling $\gl=(\ell)$ the prime of $K$ above $\ell$, we have that the finite-singular isomorphism $\phi_\ell^{\fs} \colon \honef(K_\lambda,\Ts) \to \hones(K_\lambda, \Ts)$ of Definition \ref{def:finite-singular} is well defined.

We end this section by showing a consequence of point \ref{condition:assumption-irreducible-residual-representation} of Assumption \ref{ass:assumptions-on-T} that will be used later on. We begin with a general result, that is an application of Nakayama's lemma.

\begin{lemma}\label{lem:nakayama-for-irred-repr}
	Let $A$ be a local Noetherian ring with maximal ideal $\m_A$ and $T$ be a finitely generated $A$-torsion-free module. Let also $G$ be a topological group that acts continuously and $A$-linearly on $T$. If $\ho(G,T/\m_A T)=\{0\}$, then $\ho(G,T)=\{0\}$.
\end{lemma}
\begin{proof}
	If $(T/\m_A T)^G=\{0\}$, then $T^G\subseteq \m_A T$. Therefore, one can write any nonzero element $x\in T^G$ as $x=ay$ with $a\in\m_A\setminus\{0\}$ and $y\in T$. For every $\gs\in G$, one has the relation
	\begin{equation*}
		0=\gs(ay)-ay=a(\gs(y)-y).
	\end{equation*}
	Since $T$ is $A$-torsion-free, we obtain that $\gs(y)-y=0$, i.e., $y\in T^G$. Hence, $T^G\subseteq \m_A T^G$. As $T$ is finitely generated over the Noetherian ring $A$, then also $T^G$ is finitely generated over $A$, and we can conclude by applying Nakayama's lemma.
\end{proof}

We now apply this criterion to show that the quotients of $\T$ have no invariants over the absolute Galois group of some ring class fields of $K$. 

\begin{lemma}\label{lem:no-invariants-T}
	Let $n$ be a positive integer coprime with $NpD_K$ and let $T$ be any quotient of $\T$ by an ideal of $\calR$. Then $\ho(K[n],T)=0$.
\end{lemma}
\begin{proof}
	Denote by $\bar{\T}$ the residual $G_\Q$-representation of $\T$, which is a free module of rank $2$ over a finite field, unramified outside $Np$. Let $F$ be any field different from $\Q$ that is contained in $\Q(\bar{\T}\hspace{1pt})\cap K[n]$ and take a prime $q$ that ramifies in $F$. Since the extension $\Q(\bar{\T}\hspace{1pt})/\Q$ is unramified outside $Np$, we must have that $q\mid Np$ and from the fact that the extension $K[n]/\Q$ is unramified outside $nD_K$, it follows that $q\mid nD_K$. But this is impossible, as $Np$ is coprime with $n D_K$. Therefore, $\Q(\bar{\T}\hspace{1pt})\cap K[n]=\Q$.
	Hence, restricting automorphisms from $\Qbar$ to $\Q(\bar{\T}\hspace{1pt})$ gives a surjection $G_{K[n]}\twoheadrightarrow\Gal(\Q(\bar{\T}\hspace{1pt})/\Q)$. Since, by point \ref{condition:assumption-irreducible-residual-representation} of Assumption \ref{ass:assumptions-on-T}, the representation $\bar{\T}$ has no $\Gal(\Q(\bar{\T}\hspace{1pt})/\Q)$-invariants, we conclude that it does not have any $G_{K[n]}$-invariants. The claim follows by applying Lemma \ref{lem:nakayama-for-irred-repr}.
\end{proof}

\subsection{Anticyclotomic Euler systems}\label{subsec:anticyclotomic-Euler-systems}

Let $\lcond$ ($=\{\lcond_L\}$) be a collection of Selmer structures on $\T$ over every field $L$ which is a subextension of $K[n]/K$ for some $n$ coprime with $N$. We will denote with the same letter the Selmer structure induced by $\lcond$ by \textit{propagation} (see for instance \cite[Section 1.1]{howard:heegner-points}) on every quotient of $\T$. 

Let  $\admissible'$ be an infinite subset of $\admissible$ and $\squarefreeadmissible'$ be the set of squarefree products of primes of $\admissible'$. Fix a set  $\ubf = \{\ubf_\ell\}_{\ell \in \admissible'}$ of units of $\calR$. This choice uniquely (and equivalently) determines a set $\abf = \{\abf_\ell\}_{\ell \in \admissible'}$ of elements of $\calR$ such that  $\Tr(\hspace{1pt}\Fr_\ell \hspace{1pt}\vert \hspace{1pt} \T \hspace{1pt}) = \ubf_\ell\,\abf_\ell$. Inspired by \cite[Definition 7.1]{castella-hsieh:heegner-cycles}, we give the following definition.

\begin{definition}\label{def:euler-systems}
	An \emph{anticyclotomic Euler system} attached to the triple $(\T,\lcond,\admissible')$ and relative to the set $\abf$ is a collection $\{\cbf(n)\}_{n\in\squarefreeadmissible'}$ of classes $\cbf(n)\in \hone_{\lcond}(K[n],\T\hspace{1pt})$ such that for every $n\ell\in\squarefreeadmissible'$ with $\ell$ prime, we have:
	\begin{enumerate}[label=(E\arabic*), series=E]
		\item $\Cor^{K[n\ell]}_{K[n]} \cbf(n\ell)=\abf_\ell\,\cbf(n)$;\label{E1}
		\item $\loc_{\lambda_{n\ell}}\cbf(n\ell) = \res^{K[n\ell]_{\lambda_{n\ell}}}_{K[n]_{\lambda_{n}}}\bigl(\Fr_\ell \loc_{\lambda_n}\cbf(n)\bigr)$ in $\hone(K[n\ell]_{\gl_{n\ell}},\T\hspace{1pt})$, for every prime $\gl_{n\ell}$ of $K[n\ell]$ above $\ell$ and setting $\lambda_n = \lambda_{n\ell} \cap K[n]$.\label{E2}
	\end{enumerate}
	The $\calR$-module of anticyclotomic Euler systems for $(\T,\lcond,\admissible')$ relative to $\abf$ will be denoted by $\ES(\T,\lcond,\admissible', \abf)$. Moreover, for $\{\cbf(n)\}_{n\in\squarefreeadmissible'} \in \ES(\T,\lcond,\admissible', \abf)$, we will call $\cbf_K := \Cor^{K[1]}_K \cbf(1) \in \hone(K, \T\hspace{1pt})$ the \emph{basic} class of the system.
\end{definition}

\subsection{Universal Kolyvagin systems}\label{subsec:modified-Kolyvagin-systems}

We now adapt the theory of universal Kolyvagin systems (introduced in the seminal work \cite{mazur-rubin:kolyvagin-systems}) to the anticyclotomic setting, generalizing the approaches of \cite{howard:heegner-points} and \cite{buyukboduk:deformation-kolyvagin-systems}. In the following, let $\lcond$ be a Selmer structure on $\T$ over $K$.

\begin{definition}
	For every $\sfrak\in\Z_{>0}^2$ and $n\in\squarefreeadmissible_{\sfrak}$, the \emph{modified Selmer structure} $\lcond(n)$ on $\Ts$ is defined as
	\begin{equation*}
		\hone_{\lcond(n)}(K_v,\Ts):=\begin{cases}
			\hone_{\lcond}(K_v,\Ts) &\text{if $v\nmid n$;}\\
			\hone_{\tr}(K_v,\Ts) &\text{if $v\mid n$,}
		\end{cases}
	\end{equation*}
	for every place $v$ of $K$.
\end{definition}

\begin{definition}
	For every $n\in\squarefreeadmissible$ set $\calG(n)=\bigotimes_{\ell\mid n}\calG_\ell$, where the tensor product runs over all primes $\ell$ dividing $n$.
\end{definition}

Fix now $\sfrak\in\Z_{>0}^2$. Let $\admissible_\sfrak'$ be an infinite subset of $\admissible_\sfrak$ and define $\squarefreeadmissible_{\sfrak}'$ to be the set of all squarefree products of primes of $\admissible_{\sfrak}'$. For every couple $(n,\ell)$ such that $n\ell\in\squarefreeadmissible_\sfrak'$ with $\ell$ prime, let $\chi_{n,\ell}\colon \Ts\to \Ts$ be an $\Rs$-linear isomorphism. We can use these maps to modify the finite-singular isomorphism at $\gl=(\ell)$ by defining
\[
\begin{tikzcd}[cramped]
	\phi_\ell^{\fs}(\chi_{n,\ell})\colon \honef(K_\gl,\Ts) \ar[r, "\alpha_\ell"] & \Ts \ar[r, "\chi_{n,\ell}"] & \Ts \ar[r, "\beta_\ell^{-1}"] & \hones(K_\gl,\Ts)\otimes\calG_\ell.
\end{tikzcd}
\]
For any couple $(n,\ell)$ as above, we also define the maps
\begin{itemize}
	\item $\psi_{n\ell}^\ell = ([\,\ast\,]_\s \circ \loc_\lambda) \otimes \id \colon \hone_{\lcond(n\ell)}(K,\Ts)\otimes\calG(n\ell)\to \hones(K_{\gl},\Ts)\otimes\calG(n\ell)$;
	\item $\psi_n^\ell(\chi_{n, \ell}) = (\phi^\fs_\ell(\chi_{n, \ell}) \circ \loc_\lambda) \otimes \id  \colon \hone_{\lcond(n)}(K,\Ts)\otimes\calG(n)\to \hones(K_{\gl},\Ts)\otimes\calG(n\ell)$,
\end{itemize}
where $[\,\ast\,]_\s$ denotes the projection to the singular quotient.
\begin{definition}
	A \emph{$\{\chi_{n,\ell}\}$-Kolyvagin system} for the triple $(\Ts,\lcond,\admissible_{\sfrak}')$ is a collection of cohomology classes $\{\koly(n)_\sfrak\}_{n\in\squarefreeadmissible_{\sfrak}'}$ such that, for any couple $(n,\ell)$ with $n\ell\in\squarefreeadmissible_{\sfrak}'$ and $\ell$ prime, we have:
	\begin{enumerate}[label=(K\arabic*)]
		\item $\koly(n)_\sfrak\in \hone_{\lcond(n)}(K,\Ts)\otimes\calG(n)$;\label{condition:K1}
		\item $\big(\psi_n^\ell(\chi_{n,\ell})\big)\big(\koly(n)_\sfrak\big)=\psi_{n\ell}^\ell\big(\koly(n\ell)_\sfrak\big)$.\label{condition:K2}
	\end{enumerate}
	We denote by $\KS(\Ts,\lcond,\admissible_{\sfrak}',\{\chi_{n,\ell}\})$ the $\Rs$-module of all $\{\chi_{n,\ell}\}$-Kolyvagin systems for the triple $(\Ts,\lcond,\admissible_{\sfrak}')$.
\end{definition}

\begin{remark}
	When we don't intend to specify the automorphisms $\chi_{n, \ell}$, we will refer to this notion simply as \emph{modified} Kolyvagin systems. This terminology is justified by the fact that if we choose $\chi_{n,\ell}=\id$ for every couple $(n,\ell)$, then an $\{\id\}$-Kolyvagin system for $(\Ts,\lcond,\admissible_{\sfrak}')$ is just a Kolyvagin system in the usual sense (see e.g.~\cite[Definition 1.2.3]{howard:heegner-points}).
\end{remark}

Consider now $\admissible' \subseteq \admissible_\onefrak$, where $\onefrak=(1, 1)$, such that $\admissible_\sfrak' := \admissible' \cap \admissible_\sfrak$ is an infinite subset of $\admissible_\sfrak$, for any $\sfrak \in \Z_{>0}^2$. Note in particular that $\admissible_{\sfrak}'\supseteq\admissible_{\tfrak}'$, whenever $\sfrak\preceq \tfrak$, and that $\admissible' = \admissible_\onefrak'$ is the union of all the $\admissible_\sfrak'$. Write again $\squarefreeadmissible_\sfrak'$ for the set of all squarefree products of primes of $\admissible_\sfrak'$. 

For every couple $(n,\ell) \in \Z^2$ with $n\ell\in\squarefreeadmissible'$ and $\ell$ prime, let $\chi_{n,\ell}\colon \T\to \T$ be an $\calR$-linear automorphism, which induces automorphisms $\chi_{n,\ell}\colon \Ts\to \Ts$ for every $\sfrak\in\Z_{>0}^2$. Generalizing \cite[Definition 3.1.6]{mazur-rubin:kolyvagin-systems} and \cite[Section 3.2]{buyukboduk:deformation-kolyvagin-systems}, we give the following final definition. 

\begin{definition}
	The $\calR$-module of \emph{universal $\{\chi_{n,\ell}\}$-Kolyvagin systems} for $(\T,\lcond,\admissible')$ is
	\begin{equation*}
		\KSuni(\T,\lcond,\admissible',\{\chi_{n,\ell}\}):=\varprojlim_{\sfrak}\left(\varinjlim_{\tfrak\succeq \sfrak} \KS(\Ts,\lcond,\admissible_{\tfrak}',\{\chi_{n,\ell}\})\right).
	\end{equation*}
\end{definition}

By our conventions, $1\in\squarefreeadmissible_{\sfrak}'$ for every $\sfrak\in\Z_{>0}^2$. Therefore, for a universal Kolyvagin system $\koly\in  \KSuni(\T,\lcond,\admissible',\{\chi_{n,\ell}\})$ induced by sets of classes $\{\koly(n)_\sfrak\}_{n\in\squarefreeadmissible_{\sfrak}'}\in \KS(\Ts,\lcond,\admissible_{\sfrak}',\{\chi_{n,\ell}\})$, one can define
\begin{equation*}
	\koly(1):=\varprojlim_{\sfrak}\koly(1)_\sfrak\in\varprojlim_{\sfrak}\hone(K,\Ts)=\hone(K,\T\hspace{1pt}),
\end{equation*}
where for the last equality we used the fact that $\T=\varprojlim_\sfrak \Ts$.

\begin{remark}\label{rk:equivalence-for-applications}
	\begin{enumerate}
		\item Again, when we don't intend to specify the automorphisms $\chi_{n, \ell}$, we use the terminology \emph{modified} universal Kolyvagin system. Also, a universal $\{\id\}$-Kolyvagin system for $(\T,\lcond,\admissible')$ is just a universal Kolyvagin system in the usual sense. In this case, in literature (see \cite[Definition 3.13]{mazur-rubin:kolyvagin-systems}) there is also the notion of the $\calR$-module $\KS(\T,\lcond,\admissible',\{\id\})$ of Kolyvagin systems for the quadruple $(\T,\lcond,\admissible',\{\id\})$, that comes endowed with with a natural morphism
		$\KS(\T,\lcond,\admissible',\{\id\})\to \KSuni(\T,\lcond,\admissible',\{\id\})$. This map need not be either injective nor surjective, in general. However, as far as the applications of the Kolyvagin system machinery is concerned (i.e., bounding Selmer groups), any ``classical'' Kolyvagin system has the exact same use as an element of $\KSuni(\T,\lcond,\admissible',\{\id\})$.  Moreover, the notion of universal Kolyvagin system better fits the setting when $\calR$ has Krull dimension greater than $1$.
		
		\item The decision to modify the classical notion of universal Kolyvagin systems using the automorphisms $\{\chi_{n,\ell}\}$ is motivated by the fact that, in general, we will not be able to build a universal $\{\id\}$-Kolyvagin system out of an Euler system of cohomology classes. Moreover, up to our best effort, it seems not always possible to modify a universal \{$\chi_{n,\ell}\}$-Kolyvagin system into an $\{\id\}$-one. However, the classical results that link the existence of a nontrivial (universal) Kolyvagin system to the structure of some relevant Selmer groups rely only on the existence of an isomorphism between the finite and the singular parts that links $\loc_{\gl}\koly(n)_{\sfrak}$ with $\loc_{\gl}\koly(n\ell)_{\sfrak}$, in order to deduce the non-triviality of one from the non-triviality of the other. Therefore, our generalization is still perfectly suitable for such applications, as we will see in Section \ref{sec:applications}.
	\end{enumerate}
\end{remark}

\subsection{Kolyvagin primes and main result}\label{sec:finer-choice-primes}

The aim of the rest of this section is to show how one can build a modified universal Kolyvagin system starting from an anticyclotomic Euler system. In order to do this, we need to shrink the set of admissible primes.

Let $\tfrak=(t_1,t_2)\in\Z_{>0}^2$. Whenever $\ell\in\admissible_\tfrak$, we know that $p^{t_1}\mid\ell+1\pm\abf_\ell$ as elements of $\calR_\tfrak'$, as observed after Definition \ref{dfn:s-admissible-primes}.  
Since we will need to control precisely the $p\,$-adic valuation of these two elements, we give the following definition.

\begin{definition}\label{def:Kolyvagin-primes}
	Let $\varadmissible_\tfrak$ be the set of primes $\ell$ of $\admissible_\tfrak$ such that the two quantities $\frac{\ell+1\pm \abf_\ell}{p^{t_1}}$ are units of $\calR_\tfrak'$. If $\Omega$ is a directed  subset of $\Z_{>0}^2$, for every $\sfrak\in\Z_{>0}^2$ we set
	\begin{equation*}
		\admissible'_\sfrak(\Omega):=\bigcup_{\substack{\tfrak\succeq\sfrak \\ \tfrak\in\Omega}}\varadmissible_\tfrak
	\end{equation*}
	and define $\admissible'(\Omega):= \admissible'_\onefrak(\Omega)$, where $\onefrak=(1, 1)$.
\end{definition}

From now on, we assume that $\Omega$ is fixed and we will drop it from the notation, writing $\admissible'_\sfrak:=\admissible'_\sfrak(\Omega)$ and $\admissible':=\admissible'(\Omega)$. As usual, we will also write $\squarefreeadmissible_\sfrak '$ and $\squarefreeadmissible'$ for the sets of squarefree products of the admissible primes in $\admissible_\sfrak '$ and $\admissible'$, respectively. From now on, we require also the following technical condition.

\begin{assumption}\label{ass:u-ell}
	There are $\epsilon \in \{\pm1\}$ and $a \in \Z$ such that, for every $\ell\in\admissible_\sfrak$ with $\sfrak\in\Omega$, the reduction of $\ubf_\ell$ to $\Rs'$ coincides with $\epsilon \ell^a$.
\end{assumption}

For the applications, we are concerned with the size of these sets of admissible primes. A first consequence of Assumption \ref{ass:u-ell} is the following result.

\begin{lemma}\label{lem:infinite-admissible-primes}
	The set $\varadmissible_{\tfrak}$ is infinite for any $\tfrak \in \Omega$.
\end{lemma}

\begin{proof}
	First, notice that proving that an element $r$ of $\calR_\tfrak'$ is a unit is equivalent to showing that the projection of $r$ to any arbitrary quotient of $\calR_\tfrak'$ is a unit. 
	
	Let $\tfrak\in \Omega$ and fix a prime $\ell_0\in\admissible_{\tfrak}$. By Assumptions \ref{ass:assumptions-on-T} \ref{condition:assumption-characteristic-polynomial-Frobenius} and \ref{ass:u-ell}, we have that
	\[
	\Tr(\hspace{1pt}\Fr_{\ell_0} \hspace{1pt}\vert\hspace{1pt} \T_\tfrak'\hspace{1pt})=\ubf_{\ell_0}\abf_{\ell_0} \qquad \text{and}\qquad \det(\hspace{1pt}\Fr_{\ell_0}\hspace{1pt}\vert\hspace{1pt} \T_\tfrak'\hspace{1pt})=\ell_0,
	\]
	where $\ubf_{\ell_0} = \epsilon \ell_0^a \in (\calR_\tfrak')^\times$ by Assumption \ref{ass:u-ell}. Let now $\ga\in 1+p^{t_1}\Zp$. By Assumption \ref{ass:assumptions-on-T} \ref{condition:big-image-assumption}, there is an element $\gs_\ga\in G_\Q$ such that the image of $\gs_\ga$ in $\Aut(\T_\tfrak')$ is the scalar $\ga$. Then, on $\T_\tfrak'$, one has
	\begin{equation*}
		\Tr(\hspace{1pt}\Fr_{\ell_0}\gs_\ga \hspace{1pt}\vert \hspace{1pt} \T_\tfrak'\hspace{1pt})=\alpha \, \epsilon \ell_0^a \, \abf_{\ell_0} \qquad \text{and}\qquad \det(\hspace{1pt} \Fr_{\ell_0}\gs_\ga \hspace{1pt}\vert\hspace{1pt} \T_\tfrak'\hspace{1pt})=\ga^2\ell_0.
	\end{equation*}
	Let now $\varadmissible_\tfrak(\ell_0, \alpha)$ be the set consisting of those primes $\ell$ such that $\Fr_\ell$ is conjugated to $\Fr_{\ell_0}\gs_\ga$ in $\Gal(K(\T_{(t_1+1,t_2)})/\Q)$.
	Note that, by Chebotarev's density theorem, $\varadmissible_\tfrak(\ell_0, \alpha)$ has infinite cardinality. Let's show that $\varadmissible_\tfrak(\ell_0, \alpha) \subseteq \admissible_\tfrak$. Indeed, for any $\ell \in \varadmissible_\tfrak(\ell_0, \alpha)$, one has
	\begin{equation}\label{eq:point-b}
		\epsilon\ell^a \, \abf_\ell\equiv \alpha \, \epsilon\ell_0^a \,  \abf_{\ell_0} \mod{I_{(t_1+1,t_2)}} \qquad \text{and}\qquad \ell\equiv\ga^2\ell_0 \mod{I_{(t_1+1,t_2)}}.
	\end{equation}
	Reducing modulo $I_\tfrak$, we obtain that $\abf_\ell\equiv\abf_{\ell_0}\bmod I_\tfrak$ and $\ell\equiv\ell_0 \bmod I_\tfrak$, the first congruence following from the second by taking into account that $\alpha \equiv 1 \bmod p^{t_1}$. Since $\ell_0\in \admissible_{\tfrak}$, this implies that $\ell\in\admissible_{\tfrak}$.
	
	Notice that \eqref{eq:point-b} yields also the congruence
	\begin{equation}\label{eq:proof-kolyvagin-congruences-2}
		\abf_\ell \equiv \alpha^{1-2a} \abf_{\ell_0}\mod{I_{(t_1+1,t_2)}}
	\end{equation}
	for any $\ell \in \varadmissible_\tfrak(\ell_0, \alpha)$. We then look for an $\alpha \in 1 + p^{t_1}\Zp$ such that, by dividing  the two elements
	\begin{equation}\label{eq:proof-alpha-elements}
		\alpha^2\ell_0 + 1 \pm \alpha^{1-2a}\abf_{\ell_0} \mod{J_{t_2}}
	\end{equation}
	by $p^{t_1}$, we obtain a unit of $\calR_\tfrak'$. Since for any $\ell \in \varadmissible_\tfrak(\ell_0, \alpha)$ we have, by \eqref{eq:point-b} and \eqref{eq:proof-kolyvagin-congruences-2}, the congruences 
	\[
	\ell + 1 \pm \abf_\ell \equiv \alpha^2\ell_0 +1 \pm \alpha^{1-2a}\abf_{\ell_0} \mod{I_{(t_1+1,t_2,\dots,t_d)}},
	\]
	the existence of such an $\alpha$ would imply that $\varadmissible_\tfrak(\ell_0, \alpha) \subseteq \varadmissible_\tfrak$, concluding the proof.
	
	Writing $\alpha=1+p^{t_1}x$ for $x\in\Zp$, dividing the elements \eqref{eq:proof-alpha-elements} by $p^{t_1}$ and reducing them modulo the maximal ideal, one finds that in order to get such an $\alpha$ it is sufficient to require that the reduction of $x$ avoids two explicit values. Since $p>2$, this is always possible.
\end{proof}

From the above lemma it follows that also $\admissible_\sfrak'$ is infinite for any $\sfrak \in \Z_{>0}^2$. 
Moreover, if $\ell\in \admissible_\sfrak'$, we set $s_1(\ell)$ to be the unique positive integer such that $\ell\in\varadmissible_{\sfrak(\ell)}$, with $\sfrak(\ell):=(s_1(\ell),s_2)$. In this case, thanks to Lemma \ref{lem:infinite-admissible-primes}, the values $ \frac{\ell+1\pm \abf_\ell}{p^{\,s_1(\ell)}}$ are units of $\calR_{\sfrak(\ell)}'$ and project to units of $\Rs$.

\begin{corollary}\label{cor:Frob-cpmbination-invertible}
	If $\ell\in\admissible_\sfrak'$, then the image of ${\displaystyle\frac{(\ell+1)\Fr_\ell\pm\abf_\ell}{p^{s_1(\ell)}}}$ in $\End(\Ts)$ is invertible.
\end{corollary}
\begin{proof}
	Since $\Fr_\ell$ is conjugated with $\tau_c$ in $\Gal(K(\Ts)/\Q)$ and by point \ref{condition:assumption-eigenvalues-complex-conjugation} of Assumption \ref{ass:assumptions-on-T}, we can find a basis of $\Ts$ made of eigenvectors (with eigenvalues $1$ and $-1$) for $\Fr_\ell$. With respect to this basis, one easily computes the determinant of the morphisms $\frac{(\ell+1)\Fr_\ell\pm\abf_\ell}{p^{s_1(\ell)}}$ and finds that it is a unit of $\Rs$, thanks to Lemma \ref{lem:infinite-admissible-primes}.
\end{proof}

Constructing a modified universal Kolyvagin system amounts to building a compatible system of classes $\koly(n)_\sfrak\in \hone(K,\Ts)\otimes\calG(n)$ that satisfy conditions \ref{condition:K1} and \ref{condition:K2}, for a suitable choice of morphisms $\{\chi_{n,\ell}\}$. Notice that, in order to check \ref{condition:K2}, it is necessary to know that the localization at $\gl$ of the classes $\koly(n)_\sfrak$ lie in $\hone_{\lcond(n)}(K_\gl,\Ts)\otimes\calG(n)$ for all primes $\gl$ that lie above primes of $\admissible_\sfrak'$. On the other hand, the choice of the local conditions at primes dividing $Np$ don't play any role in the proof of \ref{condition:K2}. This is why, for the sake of simplicity and generality, we will work with the following `maximal' Selmer structure.

\begin{definition}\label{dfn:relaxed-selmer-structure}
	The \emph{relaxed Selmer structure $\lcond_{\rel}$} on $\T$ over any finite extension $L$ of $K$ is defined as
	\begin{equation*}
		\hone_{\lcond_{\rel}}(L_v,\T\hspace{1pt})=\begin{cases}
			\honef(L_v,\T\hspace{1pt})&\text{if $v\nmid Np$};\\
			\hone(L_v,\T\hspace{1pt})&\text{if $v\mid Np$},
		\end{cases}
	\end{equation*}
	for every prime $v$ of $L$.
\end{definition}

\begin{theorem}\label{th:euler-to-kolyvagin-system}
	If $\{\cbf(n)\}_{n\in\squarefreeadmissible'}\in\ES(\T,\lcond_{\rel},\admissible', \abf)$, then there exist a set of automorphisms $\{\chi_{n,\ell}\}$ of $\T$ and a universal $\{\chi_{n,\ell}\}$-Kolyvagin system 
	\[
	\koly\in\KSuni(\T,\lcond_{\rel},\admissible',\{\chi_{n,\ell}\})
	\]
	such that $\koly(1)=\cbf_K$.
\end{theorem}

We devote the remainder of this section to the proof of this theorem, building explicitly the set of automorphisms $\{\chi_{n,\ell}\}$ and the modified universal Kolyvagin system.

\begin{remark}\label{remark:stressed-local-conditions}
	The previous theorem can be used also to build modified universal Kolyvagin systems in cases where one has to deal with other Selmer structures. In fact, if  $\lcond$ is a Selmer structure on $\T$ over $K$, $\koly \in \KSuni(\T,\lcond,\admissible',\{\chi_{n,\ell}\})$ if and only if $\koly \in \KSuni(\T,\lcond_{\rel},\admissible',\{\chi_{n,\ell}\})$ and $\loc_v\koly(n)_\sfrak \in H_{\lcond}^1(K_v,\Ts)\otimes\calG(n)$ for every $v\mid Np$. This means that the existence of an element of $\ES(\T,\lcond_\rel,\admissible', \abf)$ implies the existence of an element of $\KSuni(\T,\lcond,\admissible', \{\chi_{n,\ell}\})$ up to checking that the classes $\kappa(n)_\sfrak$ of Definition \ref{def:kolyvagin-derivative} (used to build the Kolyvagin system) have the correct local behavior at $Np$, which is usually dealt in concrete situations with \emph{ad hoc} methods. For some examples of this usage, see Section \ref{sec:examples}.
\end{remark}

\subsection{Dictionary}\label{sec:dictionary}

In this subsection we show how to adapt the abstract machinery of Appendix \ref{sec:formula-abstract} to our setting. Fix $\{\cbf(n)\}_{n\in\squarefreeadmissible'}\in\ES(\T,\lcond_{\rel},\admissible', \abf)$ and let $\sfrak=(s_1,s_2)\in\Z_{>0}^2$. For any $n\in\squarefreeadmissible_{\sfrak}'$, we will write $\cbf(n)_{\sfrak}$ and $\cbf(n)_{\sfrak}'$ for the image of $\cbf(n)$ in $\hone(K[n],\Ts)$ and $\hone(K[n],\Ts')$, respectively. For $\ell' \in \admissible_\sfrak'$ and $n\in \squarefreeadmissible_\sfrak'$, let 
\begin{equation*}
	D_{\ell'}:=\sum_{i=1}^{\ell'} i\sigma_{\ell'}\in\Z[\calG_{\ell'}]\quad\text{and}\quad D_n:=\prod_{\ell'\mid n}D_{\ell'}\in\Z[\calG_n],
\end{equation*}
where $\sigma_{\ell'}$ is the generator of $\calG_{\ell'}$ fixed in Section \ref{subsec:the-arithmetic-picture}. Therefore, we obtain that $D_n \cbf(n) \in \hone(K[n], \T\hspace{1pt})$.

Let now $n\in \squarefreeadmissible_\sfrak'$ and $\ell\in\squarefreeadmissible_\tfrak'$ with $\tfrak=(t_1,t_2)\succeq \sfrak$ such that $\ell\nmid n$. Fix a prime $\lambda_{n\ell}$ of $K[n\ell]$ above $\lambda = (\ell)\subseteq K$; this in turn induces primes $\lambda_n$ (resp.~$\lambda_{\ell}$) of $K[n]$ (resp.~$K[\ell]$).  We apply the formalism of Appendix \ref{sec:formula-abstract} under the following dictionary:
\medskip
\begin{center}
	\def\arraystretch{1.5}
	\begin{tabular}{c|c|c|c|c|c|c|c|c|c|c}
		Appendix \ref{sec:formula-abstract}& $\Gtilde$ & $G$ & $H$  &$\Gtilde_0$ &$G_0$ &$H_0$  &$\gs$ &$c$ & $M$ &$D$ \\
		\hline
		Section \ref{sec:euler-systems-and-kolyvagin-systems} &$G_{K[n]^+}$  &$G_{K[n]}$  &$G_{K[n\ell]}$  &$G_{\Q_\ell}$ &$G_{K_\gl}$ &$G_{K[\ell]_{\gl_\ell}}$  &$\gs_\ell$ &$\tau_c$  &$\ell+1$ &$D_\ell$\\
		\hline
	\end{tabular}
\end{center}

\begin{center}
	\def\arraystretch{1.5}
	\begin{tabular}{c|c|c|c|c|c|c|c|c|c|c}
		Appendix \ref{sec:formula-abstract}&$d$ & $R$ & $T$ & $s$  &$T/p^sT$ &$\mathbf{x}$ &$\mathbf{y}$  &$M_1$ &$\phi$ &$\delta$ \\
		\hline
		Section \ref{sec:euler-systems-and-kolyvagin-systems}  &$\ell$ &$\calR_\tfrak'$  &$\T_\tfrak'$  &$t_1$  &$\T_\tfrak$ &$D_{n} \cbf(n)_{\tfrak}'$ &$D_{n}\cbf(n\ell)_{\tfrak}'$  &$\abf_{\ell}$ &$\Fr_\ell$ &$\ubf_\ell$\\
		\hline
	\end{tabular}
\end{center}
\medskip
where $K[n]^+$ is the subextension of $K[n]$ fixed by the complex conjugation $\tau_c$. Notice that there are equalities $K[n]_{\gl_n}=K_\gl$ and $K[n\ell]_{\gl_{n\ell}}=K[\ell]_{\gl_\ell}$. Also, since $\gl_n$ is totally ramified in $K[n\ell]$, we have canonical isomorphisms 
\[
\Gal(K[n\ell]_{\lambda_{n\ell}}/K[n]_{\lambda_n}) \cong \Gal(K[n\ell]/K[n]) \cong \calG_\ell.
\]

We briefly explain why all the conditions of Section \ref{sec:setting} are satisfied. Indeed, we have that \ref{condition:pro-dihedrality}-\ref{condition:semidirect-product} follow by the pro-dihedrality of the ring class fields and $\pi$ is the projection of $G_{\Q_\ell}$ to $\Gal(\Q_\ell^\ur/\Q_\ell)$. Condition \ref{condition:structure-T} is a consequence of the freeness of $\T$ over $\calR$; \ref{condition:unramifiedness} holds since $\T$ is unramified at $\ell$; \ref{condition:frob-square-identity-mod-p-s} holds since $\ell\in\admissible_\tfrak$; \ref{condition:vanishing-h-0} is Lemma \ref{lem:no-invariants-T}; \ref{condition:unramified-classes} follows from the inflation-restriction exact sequence, as we have that $\loc_{\gl_n} \cbf(n)\in \honef(K[n]_{\gl_n},\T\hspace{1pt})$ and $\loc_{\gl_{n\ell}} \cbf(n\ell)\in \honef(K[n\ell]_{\gl_{n\ell}},\T\hspace{1pt})$ by Definition \ref{def:euler-systems}; \ref{condition:corestriction} follows from \ref{E1}, the commutativity of $D_n$ with the corestriction and the fact that $p^{t_1}\mid \abf_\ell\,$ in $\calR'$, as observed after Definition \ref{dfn:s-admissible-primes}; \ref{condition:char-poly-frob} is a combination of Assumption \ref{ass:assumptions-on-T} \ref{condition:assumption-characteristic-polynomial-Frobenius} and \ref{E1}; \ref{condition:eichler-shimura} follows from \ref{E2}, as explained in the following lemma.

\begin{lemma}
	With notation as above, we have that
	\[
	\res_{K[n]_{\lambda_{n}}}^{K[n\ell]_{\lambda_{n\ell}}}\Big(\Fr_\ell \loc_{\lambda_n}\bigl(D_n \cbf(n)\bigr)\Big) = \loc_{\lambda_{n\ell}}\bigl(D_n \cbf(n\ell)\bigr).
	\]
\end{lemma}

\begin{proof}
	If $n=1$, the relation is \ref{E2} applied to our choice of prime $\lambda_{n\ell}$. If $n>1$, note that $\Fr_\ell \loc_{\lambda_n}\bigl(D_n \cbf(n)\bigr) = \loc_{\lambda_n}\bigl(\Fr(\lambda_n/\ell) D_n \cbf(n)\bigr)$, where $\Fr(\lambda_n/\ell)$ denotes the Frobenius element of $\lambda_n$ inside $\Gal(K[n]/\Q)$ and that $\Fr(\lambda_n/\ell)\sigma = \sigma \,\Fr(\sigma^{-1}\lambda_n/\ell)$, for any $\sigma \in \Gal(K[n\ell]/\Q)$. Moreover, the action of $\sigma$ in cohomology commutes with $\res_{K[n]}^{K[n\ell]}$ and is related with the localization in the following way: $\loc_{\lambda_{n\ell}}\circ \, \sigma = \sigma \circ \loc_{\sigma^{-1}\lambda_{n\ell}}$. 
	
	Using these properties and relation \ref{E2} for $\sigma^{-1}\lambda_{n\ell}$, we obtain that
	\begin{equation*}
		\res_{K[n]_{\lambda_{n}}}^{K[n\ell]_{\lambda_{n\ell}}}\Big(\Fr_\ell \loc_{\lambda_n}\bigl(\sigma \cbf(n)\bigr)\Big) = \loc_{\lambda_{n\ell}} \bigl(\sigma \cbf(n\ell)\bigr).
	\end{equation*}
	Since $D_n$ is a linear combination of elements of $\Gal(K[n\ell]/\Q)$, the lemma follows.
\end{proof}

\subsection{Derivative classes and local properties}\label{sec:Derivative-classes-and-local-properties}

In this subsection, we apply a suitable Kolyvagin's descent to our fixed anticyclotomic Euler system $\{\cbf(n)\}_{n\in\squarefreeadmissible'}\in\ES(\T,\lcond_{\rel},\admissible', \abf)$. For this section, fix also $\sfrak\in\Z_{>0}^2$.

\begin{lemma}\label{lemma:definition-kolyvagin-derivative}
	For every $n\in\squarefreeadmissible_{\sfrak}'$ we have that
	\begin{enumerate}[label=\emph{(\alph*)}]
		\item $D_n \cbf(n)_{\sfrak}\in \hone(K[n],\Ts)^{\calG_n}$;\label{item:invariance}
		\item the restriction map $\hone(K[1],\Ts)\to \hone(K[n],\Ts)^{\calG_n}$ is an isomorphism.\label{item:iso-restriction}
	\end{enumerate}
\end{lemma}
\begin{proof}
	(a) Since $\calG_n=\prod_{\ell'\mid n}\calG_{\ell'}$, it suffices to prove that $(\gs_{\ell'}-1)D_n\cbf(n)_{\sfrak}=0$ in $\hone(K[n],\Ts)$ for every prime $\ell'\mid n$. By the telescopic identity \eqref{eq:telescopic-identity-appendix}  and the equality $\Tr_{\calG_{\ell'}}=\res_{K[n/\ell']}^{K[n]}\circ\Cor_{K[n/\ell']}^{K[n]}$, this reduces to proving that  $\Cor_{K[n/\ell']}^{K[n]}(\cbf(n)_{\sfrak})=0$, since  $\ell'+1\equiv 0\bmod p^{s_1}$. Since $p^{s_1}\mid \abf_{\ell'}$, this fact follows from \ref{E1}.
	
	\emph{\ref{item:iso-restriction}} By Lemma \ref{lem:no-invariants-T}, our assumptions on $\T$ imply that $\ho(K[n],\Ts)=0$. We obtain the claim by applying the inflation-restriction exact sequence.
\end{proof}

\begin{definition}\label{def:kolyvagin-derivative}
	For every $n\in\squarefreeadmissible_{\sfrak}'$, define \[
	\gk(n)_\sfrak:=\Cor^{K[1]}_K (\res_{K[1]}^{K[n]})^{-1}D_n\cbf(n)_{\sfrak}\in \hone(K, \Ts).
	\]
\end{definition}

We will show that a slight modification of the classes $\gk(n)_\sfrak$ forms a modified universal Kolyvagin system. A first step in this direction is the following proposition, which establishes \ref{condition:K1}.

\begin{proposition}\label{prop:derived-classes-in-the-Selmer}
	For every $n\in\squarefreeadmissible_{\sfrak}'$, we have that $\gk(n)_\sfrak\in \hone_{\lcond_{\rel}(n)}(K,\Ts)$.
\end{proposition}

\begin{proof}
	We have to show that the $\loc_v \gk(n)_\sfrak \in \hone_{\lcond_{\rel}(n)}(K_v, \Ts)$, for any prime $v$ of $K$. If $v\mid Np$, this is trivial since $\hone_{\lcond_{\rel}(n)}(K_v,\Ts)=\hone(K_v,\Ts)$. For the remaining primes, we split the proof into two steps, depending on wether $v\nmid n$ or $v\mid n$. Since $\cbf(n) \in \ES(\T, \lcond_{\rel}, \admissible, \abf)$, we have that $\cbf(n)_\sfrak\in H_{\lcond_{\rel}}^1(K[n],\Ts)$. In particular, for every prime $v_n$ of $K[n]$ above $v\nmid Np$, this means that
	\begin{equation}\label{equ:proof-Selmer-local-conditions-1}
		\loc_{v_n} \cbf(n)_\sfrak \in  \honef(K[n]_{v_n}, \Ts).
	\end{equation}
	Moreover, write $v_1=v_n\cap K[1]$.
	
	(a) Suppose that $v\nmid nNp$, so that $\hone_{\lcond_{\rel}(n)}(K_v,\Ts)=\honef(K_v,\Ts)$. Note that in this case  $K[n]_{v_n}^{\ur}=K[1]_{v_1}^{\ur}=K_v^{\ur}$. Therefore, the claim follows from \eqref{equ:proof-Selmer-local-conditions-1} and some diagram chasing, by applying the functoriality of restriction, corestriction and $D_n$ in semi-local Galois cohomology (see Appendix \ref{sec:semi-local-cohomology}).
	
	(b) Suppose that $v\mid n$, so that $\hone_{\lcond_{\rel}(n)}(K_v,\Ts)=\hone_{\tr}(K_v,\Ts)$. Call $\ell$ the rational prime below $v$ and $v_\ell=v_n\cap K[\ell]$. By the explicit description of \eqref{equ:easy-description-of-transverse-condition}, it is enough to prove that
	\begin{equation*}
		\loc_v \gk(n)_\sfrak\in  \ker\big(\hone(K_v,\T\hspace{1pt})\xrightarrow{\res} \hone(K[\ell]_{v_\ell},\T\hspace{1pt})\big).
	\end{equation*}
	
	\textbf{Step 1.} We show that $\loc_{v_n} (D_n\cbf(n)_{\sfrak})=0$.
	
	By \eqref{equ:proof-Selmer-local-conditions-1} and the fact that $D_n$ commutes with $\loc_{v}$ (in a semi-local sense, see  Appendix \ref{subsec:Galois-action}), we have that $\loc_{v_n} (D_n\cbf(n)_{\sfrak})\in \honef(K[n]_{v_n},\Ts)$. Then, noting that $D_n=\prod_{\ell'\mid n}D_{\ell'}$ and that the evaluation at $\Fr_{v_n}$ induces an isomorphism between $\hone_{\f}(K[n]_{v_n}, \Ts)$ and $\Ts$ by  
	\cite[Lemma B.2.8]{rubin:euler-systems}, it suffices to show that the evaluation $(D_\ell \cbf(n)_{\sfrak})(\Fr_{v_n})$ of a cocycle representing $D_\ell \cbf(n)_{\sfrak}$ in $\Fr_{v_n}$ is trivial.
	
	Since $\Ts$ is unramified outside $Np$, one can lift $\sigma_\ell$ (our fixed generator of $\calG_\ell$) to an element of $G_{K[1]}$ that acts trivially on $\Ts$. But then, working exactly as in the proof of Lemma \ref{lemma:D=0}, one finds that
	\begin{equation*}
		(D_\ell \cbf(n)_{\sfrak})(\Fr_{v_n})=\sum_{i=1}^\ell i(\sigma_\ell \cbf(n)_{\sfrak})(\Fr_{v_n})=\sum_{i=1}^\ell i \cbf(n)_{\sfrak}(\Fr_{v_n})=\frac{\ell(\ell+1)}{2} \cbf(n)_{\sfrak}(\Fr_{v_n})=0,
	\end{equation*}
	since $\ell+1$ is zero in $R_{\sfrak}$.
	
	\textbf{Step 2.} By class field theory, the prime $v_\ell$ splits completely in $K[n]$ and therefore $K[n]_{v_n}=K[\ell]_{v_\ell}$. Then, the commutative diagram
	\begin{equation*}
		\begin{tikzcd}
			{\hone(K[n],\Ts)} && {\hone(K[n]_{v_n},\Ts)} \\
			{\hone(K[1],\Ts)} & {\hone(K[1]_{v_1},\Ts)} & {\hone(K[\ell]_{v_\ell},\Ts)}
			\arrow["{\loc_{v_n}}", from=1-1, to=1-3]
			\arrow["\res", from=2-1, to=1-1]
			\arrow["{\loc_{v_1}}", from=2-1, to=2-2]
			\arrow["\res", from=2-2, to=2-3]
			\arrow[Rightarrow, no head, from=2-3, to=1-3]
		\end{tikzcd}
	\end{equation*}
	together with the commutativity of corestriction in semi-local Galois cohomology (see Appendix \ref{app:corestriction}) and Step 1 yield that the restriction of $\gk(n)_\sfrak$ to $\hone(K[\ell]_{v_\ell},\Ts)$ is zero.
\end{proof}

\subsection{From Euler systems to Kolyvagin systems}\label{sec:from-Euler-to-Kolyvagin-systems}

We are now able to apply the machinery of Appendix \ref{sec:formula-abstract} in its full strength in order to prove Theorem \ref{th:euler-to-kolyvagin-system}.
\begin{lemma}\label{lem:formula-capitolo-2}
	Let $\sfrak=(s_1,s_2) \in \Z_{>0}^2$. For any $n\ell\in\squarefreeadmissible_{\sfrak}'$ with $\ell$ prime, the relation
	\begin{equation*}
		\ga_\ell(\loc_{\gl}\gk(n)_\sfrak)=\vartheta_{\ell}\big(\beta_\ell([\loc_{\gl}\gk(n\ell)_\sfrak]_{\s} \otimes \sigma_\ell)\big)
	\end{equation*}
	holds in $\Ts$, where $\gl$ is the prime of $K$ above $\ell$ and
	\begin{equation*}
		\vartheta_{\ell}=\left(\frac{(\ell+1)\Fr_\ell-\abf_\ell}{p^{s_1(\ell)}}\right)^{-1}\left(\frac{(\ell+1)\Fr_\ell-\ubf_\ell\, \abf_\ell}{p^{s_1(\ell)}}\right)\Fr_\ell \in\Aut(\Ts).
	\end{equation*}
\end{lemma}

\begin{proof}
	Let $\gl_{n\ell}$ be a prime of $K[n\ell]$ above $\gl$, write $\gl_n=\gl_{n\ell}\cap K[n]$ and $\gl_1=\gl_{n\ell}\cap K[1]$, and recall that $K[n]_{\gl_n}=K[1]_{\gl_1}=K_{\gl}$. For any $\tfrak\succeq\sfrak$, write 
	\begin{equation*}
		\tilde{\kappa}(n)_{\tfrak} := D_n\cbf(n)_\tfrak \quad\text{and}\quad  \tilde{\kappa}(n\ell)_{\tfrak} := (\res_{K[n]}^{K[n\ell]})^{-1}D_{n\ell}\cbf(n\ell)_\tfrak.
	\end{equation*}
	Notice that $\tilde{\kappa}(n\ell)_{\tfrak}$ is well defined as the proof of Lemma \ref{lemma:definition-kolyvagin-derivative} still holds if we replace $\sfrak$ with $\tfrak$, $n$ with $n\ell$ and $K[1]$ with $K[n]$, just assuming $n \in \squarefreeadmissible'$ and $\ell \in \squarefreeadmissible_\tfrak'$.
	
	Since $\ell\in\admissible_\sfrak'$, then $\sfrak(\ell) \succeq \sfrak$, where $\sfrak(\ell)\in\Z_{>0}^2$ was defined before Corollary \ref{cor:Frob-cpmbination-invertible}. Under the dictionary of Section \ref{sec:dictionary}, with $\tfrak=\sfrak(\ell)$, the key formula of Proposition \ref{prop:key-formula-appendix} translates to
	\begin{equation*}
		\biggl(\frac{\ell+1}{p^{\,s_1(\ell)}}\Fr_\ell-\frac{\abf_\ell}{p^{\,s_1(\ell)}}\biggr)\tilde{\gk}(n)_{\sfrak(\ell)}(\Fr_\ell^2)=\biggl(\frac{\ell+1}{p^{\,s_1(\ell)}}-\frac{\ubf_\ell \, \abf_\ell}{p^{\,s_1(\ell)}}\Fr_\ell\biggr)\tilde{\gk}(n\ell)_{\sfrak(\ell)}(\sigmati_\ell)
	\end{equation*}
	on $\T_{\sfrak(\ell)}$, where $\sigmati_\ell$ is a lift of $\sigma_\ell$ to $\Gal(\Qlbar/\Q_\ell^{\ur})$, since we have that $\abar_x = \tilde{\kappa}(n)_{\sfrak(\ell)}(\Fr_\ell^2)$ and $\abar=-\tilde{\kappa}(n\ell)_{\sfrak(\ell)}(\sigmati_\ell)$. Thanks again to Lemma \ref{lemma:definition-kolyvagin-derivative}, we can also define 
	\begin{equation*}
		\tilde{\kappa}'(n)_{\sfrak} := (\res^{K[n]}_{K[1]})^{-1}\tilde{\gk}(n)_{\sfrak} \quad\text{and}\quad  \tilde{\kappa}'(n\ell)_{\sfrak} := (\res^{K[n]}_{K[1]})^{-1}\tilde{\gk}(n\ell)_{\sfrak},
	\end{equation*}
	so that the above relation, modulo $I_\sfrak$, yields the equality  
	\begin{equation}\label{eq:tilda-prime}
		\biggl(\frac{(\ell+1)\Fr_\ell-\abf_\ell}{p^{\,s_1(\ell)}}\biggr)\tilde{\gk}'(n)_{\sfrak}(\Fr_\ell^2)=\biggl(\frac{(\ell+1)-\ubf_\ell \, \abf_\ell\Fr_\ell}{p^{\,s_1(\ell)}}\biggr)\tilde{\gk}'(n\ell)_{\sfrak}(\sigmati_\ell)
	\end{equation}
	on $\Ts$. By the explicit description of the maps $\alpha_\ell$ and $\beta_\ell$ given by Lemma \ref{lem:isomorfismi-finito-singolare}, we have that
	\[
	\tilde{\gk}'(n)_\sfrak(\Fr_\ell^2) = \alpha_{\ell}(\loc_{\lambda_1} \tilde{\gk}'(n)_\sfrak) \quad \text{and} \quad 
	\tilde{\gk}'(n\ell)_\sfrak(\sigmati_\ell) = \beta_{\ell}([\loc_{\lambda_1} \tilde{\gk}'(n\ell)_\sfrak]_{\s} \otimes \sigma_\ell).
	\] 
	As seen in the proof of Lemma \ref{lem:no-invariants-T}, the fields $K[1]$ and $\Q(\Ts)$ are disjoint over $\Q$. This implies that we can lift every element of $\Gal(K[1]/K)$ to an element of $G_K$ that acts trivially on $\Ts$. Then, the explicit formula of Appendix \ref{app:corestriction} yields the commutative diagram
	\begin{equation*}
		\begin{tikzcd}
			{\hone(K[1],\Ts)} & {\hone(K[1]_{\gl_1},\Ts)} \\
			{\hone(K,\Ts)} & {\hone(K_{\gl},\Ts),}
			\arrow["{\loc_{\gl_1}}", from=1-1, to=1-2]
			\arrow["{\Cor_K^{K[1]}}"', from=1-1, to=2-1]
			\arrow["{[K[1]:K]}", from=1-2, to=2-2]
			\arrow["{\loc_\gl}", from=2-1, to=2-2]
		\end{tikzcd}
	\end{equation*}
	where the right vertical map is multiplication by $[K[1]:K]$. Therefore, by multiplying both sides of equation \eqref{eq:tilda-prime} by $[K[1]:K]$, we obtain the relation
	\begin{equation*}
		\biggl(\frac{(\ell+1)\Fr_\ell-\abf_\ell}{p^{\,s_1(\ell)}}\biggr)\ga_\ell(\loc_\gl\gk(n)_{\sfrak})=\biggl(\frac{(\ell+1)-\ubf_\ell \, \abf_\ell\Fr_\ell}{p^{\,s_1(\ell)}}\biggr)\beta_\ell([\loc_\gl\gk(n\ell)_\sfrak]_{\s}\otimes\sigma_\ell).
	\end{equation*}
	Since $\ell\equiv -1\bmod p^{s_1}$, we have that $\ubf_\ell\equiv\epsilon\ell^a\equiv\pm 1\bmod I_\sfrak$. By Corollary \ref{cor:Frob-cpmbination-invertible}, the coefficients of the left-hand side and the right-hand side are invertible. Therefore, using the fact that $\Fr_\ell^2$ acts as the identity on $\Ts$, we obtain the claimed relation.
\end{proof}

\begin{proof}[Proof of Theorem \ref{th:euler-to-kolyvagin-system}]
	Fix $\sfrak\in\Z_{>0}^2$, set $\chi_{n,\ell}:=\vartheta_\ell^{-1}$, for every $n\ell\in\squarefreeadmissible_\sfrak'$ with $\ell$ prime, and define  
	\begin{equation}
		\koly(n)_\sfrak:=\gk(n)_\sfrak\otimes\bigotimes_{\ell'\mid n}\sigma_{\ell'}\in \hone(K,\Ts)\otimes\calG(n).\label{equ:kappa'}
	\end{equation}
	By combining  Proposition \ref{prop:derived-classes-in-the-Selmer} and Lemma \ref{lem:formula-capitolo-2}, we conclude that the set of classes $\{\koly(n)_\sfrak\}_{n\in\squarefreeadmissible_\sfrak'}$ is a $\{\chi_{n,\ell}\}$-Kolyvagin system for $(\Ts,\lcond_{\rel},\admissible_\sfrak')$.
	
	We need now to prove that, when varying $\sfrak\in\Z_{>0}^2$, these classes interpolate into a modified universal Kolyvagin system. First, notice that, for fixed $n$ and $\ell$, the automorphism $\chi_{n,\ell}$ of $\Ts$ is the image modulo $I_\sfrak$ of an automorphism of $\T$, unique for all $\sfrak\in\Z_{>0}^2$. Therefore, it is enough to check that, whenever $\tfrak\succeq\mathfrak{r}\succeq\sfrak$, the image of 
	$\{\koly(n)_\mathfrak{r}\}_{n\in\squarefreeadmissible_{\tfrak}'}$ under the natural projection map
	\begin{equation*} \KS(\T_{\mathfrak{r}},\lcond_{\rel},\admissible_{\tfrak}',\{\chi_{n,\ell}\})\longrightarrow \KS(\Ts,\lcond_{\rel},\admissible_{\tfrak}',\{\chi_{n,\ell}\})
	\end{equation*}
	coincides with $\{\koly(n)_\sfrak\}_{n\in\squarefreeadmissible_{\tfrak}'}$. This follows easily from the construction of  $\koly(n)_{\mathfrak{r}}$ and $\koly(n)_\sfrak$. Indeed, $\kappa(n)_{\mathfrak{r}}$ and $\kappa(n)_{\sfrak}$ are obtained as the images of the same class $\cbf(n) \in \hone(K[n], \T\,)$ by applying operators that commute  with the natural projection $\T_{\mathfrak{r}} \to \Ts$. 
	
	Lastly, it is clear from the definitions that $\koly(1)_\sfrak=\gk(1)_\sfrak=\Cor_K^{K[1]}\cbf(1)_\sfrak$ for any $\sfrak\in\Z_{>0}^2$, therefore we eventually obtain the relation claimed in Theorem \ref{th:euler-to-kolyvagin-system}.
\end{proof}

\begin{remark}
	As noticed at the end of the proof of Lemma \ref{lem:formula-capitolo-2}, for every $\ell\in\admissible_\sfrak'$ we have that $\ubf_\ell\equiv\pm1\bmod I_\sfrak$ and its value is independent on $\sfrak=(s_1,s_2)\in\Z_{>0}^2$ thanks to Assumption \ref{ass:u-ell}. We remark here that, when $\ubf_\ell\equiv 1\bmod I_\sfrak$, the automorphisms appearing in Lemma \ref{lem:formula-capitolo-2} simplify to $\vartheta_\ell=\Fr_\ell$. This happens for example when $\T$ is the representation attached to a modular form of weight $k\equiv 2\bmod 4$, as we will see in Section \ref{sec:examples}.
\end{remark}

\begin{remark}\label{rk:E3}
	We conclude this section by noticing that some authors require also a condition about the action of the complex conjugation on the classes forming an anticyclotomic Euler system. Using the notation of Definition \ref{def:euler-systems}, this condition can naturally be generalized to our setting to
	\begin{enumerate}[resume*=E]
		\item  $\tau_c(\cbf(n))=w \cdot \sigma(n) (\cbf(n))$ for some $w\in\{\pm 1\}$ and $\sigma(n)\in\Gal(K[n]/K)$.\label{E3}
	\end{enumerate}
	Let $\varepsilon_n := (-1)^{\omega(n)}w$, where $\omega(n)$ is the number of prime factors dividing $n$. Under this further condition, Lemma \ref{lem:formula-capitolo-2} shows that $ \phi_\ell^{\fs}(\loc_{\gl}\gk(n)_\sfrak)=v_{n, \ell}([\loc_{\gl}\gk(n\ell)_\sfrak]_{\s}\otimes \sigma_\ell)$, for every $n\ell\in\squarefreeadmissible_{\sfrak}'$, where
	\begin{equation*}
		v_{n, \ell}=\e_n\frac{\e_n(\ell+1)-\ubf_\ell\,\abf_\ell}{\e_n(\ell+1)-\abf_\ell}\in \calR^\times.
	\end{equation*}
	Indeed, using the commutativity $\tau_c D_n\equiv (-1)^{\omega(n)}D_n \tau_c \bmod{p^{s_1}}$ together with (E3), one can show that $\gk(n)_\sfrak$ lies in the $\e_n$-eigenspace of $\hone(K,\Ts)$ with respect to the action of $\tau_c$. The claimed formula  follows by applying $\beta_\ell^{-1}$ to the last displayed equation in the proof of Lemma \ref{lem:formula-capitolo-2}, since the action of $\Fr_\ell$ coincides with the action of $\tau_c$ on $\hone(K_\gl,\Ts)$ (as they agree as elements of $\Gal(K_\lambda/\Ql)$). 
\end{remark}

\subsection{Applications}\label{sec:applications}

In the spirit of point (2) of the Remark \ref{rk:equivalence-for-applications}, we show here that modified universal Kolyvagin systems have the same usage of classical (universal) Kolyvagin systems when it comes to study the structure of Selmer groups.

For this section, let $\calR$ be a discrete valuation ring, $\Phi$ be its fraction field and $\A:=\T\otimes_{\calR}\Phi/\calR$. By Assumption \ref{ass:assumptions-on-R} \ref{ass-cond:J-ideals}, we have that $J_{s_2}=\{0\}$ for every $s_2\in\Z_{>0}$. This is why, identifying $\sfrak$ with $s_1=:s$, we define
\begin{equation*}
	\calR_s=\calR/p^s\quad\text{and}\quad\T_s=\T/p^s\T.
\end{equation*}
Denote by $\bar{\T}$ the residual $G_\Q$-representation and, if $M$ is any $\calR$-module with an action of the complex conjugation $\tau_c$, we write $M^+$ and $M^-$ for the submodules of $M$ where $\tau_c$ acts as $1$ and $-1$, respectively.

Let $\lcond$ be a Selmer structure on $\T$ over $K$ and assume that the triple $(\T,\lcond,\admissible')$ satisfies hypotheses H.2--H.5 of \cite[Section 1.3]{howard:heegner-points}. In order to recover a parallel of \cite[Theorem 1.6.1]{howard:heegner-points}, we first need a deeper study of the set of primes $\admissible'$ defined in Section \ref{sec:finer-choice-primes}.

\begin{lemma}\label{lem:even-finer-choice-of-primes}
	Let $c^+\in \hone(K,\bar{\T}\hspace{1pt})^+$ and $c^-\in \hone(K,\bar{\T}\hspace{1pt})^-$. Then, for every $s\gg 0$ there are infinitely many primes $\ell\in\varadmissible_s$ such that if $c^\pm \ne 0$, then $\loc_\gl(c^{\pm})\ne 0$, where $\gl=(\ell)$ is the prime of $K$ above $\ell$.
\end{lemma}
\begin{proof}
	We assume that both $c^\pm$ are nonzero, the proof of the other cases being analogous. Let $F/\Q$ be the extension of \cite[Hypothesis H.2]{howard:heegner-points} and let $L_s = K(\T_s)$, for every $s\ge 1$. Since $L_s$ is contained in $F(\mu_{p^{\infty}})$, the restriction
	\begin{equation*}
		\hone(K,\bar{\T}\hspace{1pt})\to \hone(L_s,\bar{\T}\hspace{1pt})^{\Gal(L_s/K)}\cong\Hom(G_{L_s},\bar{\T}\hspace{1pt})^{\Gal(L_s/K)}
	\end{equation*}
	is an injection for any $s\ge1$. Denote by $\psi_s^+$ and $\psi_s^-$ the non-zero homomorphisms of $\Hom(G_{L_s},\bar{\T}\hspace{1pt})$ corresponding to $c^+$ and $c^-$. Let $\tilde{L}_s$ be the smallest (necessarily abelian because $\bar{\T}$ is so) extension of $L_s$ that is cut out by $\psi_s^+$ and $\psi_s^-$ and Galois over $\Q$, and let $G_s:=\Gal(\tilde{L}_s/L_s)$. 
	
	Note that $\tau_c$ acts on $G_s$ by conjugation, determining the eigenspaces $G_s^+$ and $G_s^-$ associated with the eigenvalues $1$ and $-1$. Set
	\begin{equation*}
		M_s:=\Hom(G_s,\bar{\T}\hspace{1pt})^{\Gal(L_s/K)}
	\end{equation*}
	and call $\bar{\psi}_s^{\pm}$ the elements of $M_s$ determined by $\psi_s^{\pm}$, respectively. There is also an action of $\tau_c$ on $M_s$ and, by hypothesis, $\bar{\psi}_s^{\pm}\in M_s^{\pm}$.
	
	We now claim that the maps $\bar{\psi}_s^\pm$ are both non-zero on $G_s^+$. This is because each map $\bar{\psi}_s^\pm$ factors through the $p\hspace{1pt}$-primary part of $G_s$, which splits as the sum of the two eigenspaces for the action of $\tau_c$. 
	Hence, if we suppose that $\bar{\psi}_s^\pm(G_s^+)=0$, then $\bar{\psi}_s^\pm(G_s)=\bar{\psi}_s^\pm(G_s^-)$ is contained in $\bar{\T}^\mp$, which is 1-dimensional by Assumption \ref{ass:assumptions-on-T} \ref{condition:assumption-eigenvalues-complex-conjugation}. Since $\bar{\psi}_s^\pm$ is non-zero and fixed by $\Gal(L_s/K)$, it follows that $\bar{\psi}_s^\pm(G_s)$ spans a non-zero proper $(\calR/\m_\calR)$-submodule of $\bar{\T}$ stable under the action of $G_\Q$, which contradicts Assumption \ref{ass:assumptions-on-T} \ref{condition:assumption-irreducible-residual-representation}. Then, it follows that we can find
	\begin{equation*}
		g\in G_s^+\quad\text{such that}\quad \bar{\psi}_s^\pm(g)\ne 0.
	\end{equation*}
	
	Notice now that for any $s\ge1$ the (nontrivial) homomorphism $\psi_{s+1}^{\pm}$ coincides with the restriction of $\psi_s^\pm$ to $G_{L_{s+1}}$, so that $\ker \psi_{s+1}^\pm=\ker \psi_s^\pm\cap G_{L_{s+1}}$. Therefore, a careful analysis of the Galois groups involved yields that $L_{s+1}\tilde{L}_s= \tilde{L}_{s+1}\ne L_{s+1}$, hence $G_{s+1}$ is a nonzero subgroup of $G_s$. Since $G_1$ is finite, there is an $s_0$ such that $G_{s+1}=G_s$ for all $s\ge s_0$, which yields also that $L_s=L_{s+1}\cap \tilde{L}_s$ and in particular that
	\begin{equation}\label{equ:splitting-Galois-groups-applications}
		\Gal(\tilde{L}_{s+1}/L_s)=\Gal(L_{s+1}/L_s)\times G_s.
	\end{equation}
	
	Suppose therefore that $s\ge s_0$ and let $\ga\in 1+p^s\Zp$. By Assumption \ref{ass:assumptions-on-T} \ref{condition:big-image-assumption} and the splitting of \eqref{equ:splitting-Galois-groups-applications}, it follows that there is an element $\gs_\ga\in G_{L_s}$ that fixes $\tilde{L}_s$ and such that the image of $\gs_\ga$ in $\Aut(\T_{s+1})$ is the scalar $\ga$. Define $\othervaradmissible_s(\ga)$ to be the set of primes $\ell$ that are unramified and whose Frobenius $\Fr_\ell$ is conjugated with $\tau_c g \gs_\ga$ in $\Gal(\tilde{L}_{s+1}/\Q)$. Notice that, in particular, $\Fr_\ell|_{\tilde{L}_s}$ is conjugated with $\tau_c g$ and $\Fr_\ell|_{L_{s+1}}$ is conjugated with $\tau_c \sigma_{\alpha}$ (whence $\othervaradmissible_s(\alpha) \subseteq \admissible_s$).
	
	By Chebotarev's density theorem, the set $\othervaradmissible_s(\ga)$ is infinite. Our aim is to find a suitable $\ga$ such that $\othervaradmissible_s(\ga)\subseteq\varadmissible_s$ and such that the primes of $\othervaradmissible_s(\ga)$ satisfy the claim of the lemma.
	
	Choose now a prime $\tilde{\gl}_s$ of $\tilde{L}_s$ above $\ell\in \othervaradmissible_s(\ga)$ and set $\gl_s=\tilde{\gl}_s\cap L_s$ and $\gl=\gl_s\cap K$. Notice that $\gl$ splits completely in $L_s$, because $\Fr_{\lambda}$ acts trivially on $\T_s$. Moreover we have the equality $\Fr_{\tilde{\gl}_s/\gl_s}=(\Fr_{\tilde{\gl}_s/\ell})^2=\tau_c \, g\,\tau_c\, g=g^2\in G_s$, using the fact that $g\in G_s^+$, and therefore
	\begin{equation*}
		\bar{\psi}^{\pm}(\Fr_{\tilde{\gl}_s/\gl_s})
		=\bar{\psi}^{\pm}(g^2)=2\bar{\psi}^{\pm}(g)\ne 0,
	\end{equation*}
	since $p\ne 2$. This shows that the restriction of $\bar{\psi}^{\pm}$ to $\Gal\big((\tilde{L}_s)_{\tilde{\gl}_s}/(L_s)_{\gl_s}\big)$ is non-zero and hence, since $(L_s)_{\gl_s}=K_\gl$, we obtain that $\loc_\gl c^{\pm}$ is non-zero.
	
	Eventually, notice that, for any $\ell_0 \in \admissible_{s+1}$, the set $\varadmissible_s(\ell_0, \alpha)$ defined in the proof of Lemma \ref{lem:infinite-admissible-primes} contains $\othervaradmissible_s(\alpha)$, hence we can argue as in that proof and find a suitable $\ga\in 1+p^s\Zp$ such that $\othervaradmissible_s(\ga)\subseteq \varadmissible_s$, concluding the proof.
\end{proof}

\begin{theorem}\label{thm:howard-abstract}
	Suppose that there is a universal Kolyvagin system \[
	\koly\in  \KSuni(\T,\lcond,\admissible',\{\chi_{n,\ell}\})
	\] 
	with $\koly(1)\ne 0$. Then $\hone_{\lcond}(K,\T\hspace{1pt})$ is a free rank-one $\calR$-module and there is a finite $\calR$-module $M$ such that
	\begin{equation*}
		\hone_{\lcond}(K,\A)\cong \Phi/\calR\oplus M\oplus M.
	\end{equation*}
	Furthermore, we have that $\length_{\calR}(M) \le \length_\calR(\hone_{\lcond}(K, \T\hspace{1pt})/\calR\cdot \koly(1))$.
\end{theorem}
\begin{proof}
	Notice that the hypotheses H.0 and H.1 of \cite[Section 1.3]{howard:heegner-points} are a consequence of the first two points of Assumption \ref{ass:assumptions-on-T}, therefore all the hypotheses H.0-H.5 of \textit{loc.~cit.}~are in force. Then, the proof goes along the lines of the proof of \cite[Theorem 1.6.1]{howard:heegner-points}, where we replace Lemma 1.6.2 in \textit{loc.~cit.}~with Lemma \ref{lem:even-finer-choice-of-primes}. Notice also that, as observed in Remark \ref{rk:equivalence-for-applications}, we can use the existence of a  universal Kolyvagin system instead of a classical one and replace the usual Kolyvagin system relations, used by Howard at the end of the proof of \cite[Lemma 1.6.4]{howard:heegner-points}, with our modified Kolyvagin system relations, where we twisted the finite-singular automorphism with the $\chi_{n,\ell}$'s. These changes are harmless and lead to the claim exactly as in \cite{howard:heegner-points}. 
\end{proof}

\begin{remark}
	In the literature one can find also stronger results making use of a different variation of Kolyvagin's method. However, these results are specific to the chosen arithmetic object and often make also use of \ref{E3}, which we didn't assume. See Section \ref{sec:examples} for more on this topic.
\end{remark}

\section{Iwasawa theory}\label{sec:Iwasawa-theory}

The main goal of this section is to define a general notion of \emph{$p\hspace{1pt}$-complete anticyclotomic Euler system} and to show how one can derive a modified universal Kolyvagin system for the anticyclotomic twist of our fixed Galois representation $\T$ from it. We keep using the same notation and assumptions of Sections \ref{subsec:the-arithmetic-picture}--\ref{sec:hypotheses}.

\subsection{The anticyclotomic twist of $\T$}\label{sec:the-anticyclotomic-twist}

For every integer $\ga\ge 1$, denote by $K_\ga$ the maximal $p\hspace{1pt}$-subextension of $K[p^{\ga+1}]/K$ and set $K_0:=K$. By \cite[Theorem 7.24]{cox:primes-of-the-form} and
Assumption \ref{ass:class-number-K}, we have that $K_\ga/K$ is cyclic of order $p^\ga$.

\begin{definition}
	The field $K_{\infty}:=\bigcup_{\alpha\ge 0} K_\alpha$ is called the \emph{anticyclotomic $\Zp$-extension} of $K$.
\end{definition}

If $n\in\Z_{>0}$ is coprime with $p$, the fields $K_\ga$ and $K[n]$ are disjoint over $K$ by ramification issues and Assumption \ref{ass:class-number-K}. When $n=1$, this implies that any prime of $K$ that lies above $p$ is totally ramified in $K_\ga$. Moreover, by class field theory, every prime of $K$ lying above a prime of $\admissible$ is split in $K_\ga$.

\begin{definition}
	We define $K_\ga[n]$ to be the compositum of $K_\ga$ and $K[n]$. Moreover, we set $\Gamma^{\ac}:=\Gal(K_\infty/K)\cong\Zp$ and $\gL^{\ac}:=\Zp\llbracket \Gamma^{\ac}\rrbracket $. We also fix a profinite generator $\gamma_{\ac}$ of $\Gamma^{\ac}$.
\end{definition}

In order to study the arithmetic of $\T$ along the anticyclotomic tower, we introduce the following twists of $\calR$ and $\T$.

\begin{definition}
	Call $\calR^{\ac}=\calR\llbracket \Gamma^{\ac}\rrbracket \cong\calR \ \hat{\otimes}_{\Zp} \gL^{\ac}$ and set $\T^{\ac}=\T\otimes_{\calR}\calR^{\ac}\cong\T \ \hat{\otimes}_{\Zp}\gL^{\ac}$. Here the symbol $\hat{\otimes}_{\Zp}$ denotes the completed tensor product between topological $\Zp$-modules.
\end{definition}

We make $\calR^{\ac}$ a left $G_K$-module via the natural projection $G_K\twoheadrightarrow \Gamma^{\ac}$. Due to this, $\T^{\ac}$ is naturally a $G_K$-module via the action coming from both sides of the tensor product. Since, as $G_K$-module, $\Gamma^{\ac}$ is unramified outside $p$, the $G_K$-representation $\T^{\ac}$ is unramified outside $Np$. Moreover, it is free of rank $2$ over $\calR^{\ac}$ and its residual representation is isomorphic to the residual representation of $\T$.

For every $\sfrak=(s_1,s_2)\in\Z_{>0}^{2}$ and $\ga\ge 0$, set $I_{\sfrak,\ga}:=(I_\sfrak,\gamma_{\ac}^{p^{\ga}}-1)\subseteq \calR^{\ac}$ and
\begin{equation*}
	\Rsa^{\ac}:=\calR^{\ac}/I_{\sfrak,\ga}\cong\Rs\otimes_{\Zp}\gL^{\ac}/(\gamma_{\ac}^{p^{\ga}}-1)\quad\text{and}\quad \Tsa^{\ac}:=\T^{\ac}\otimes_{\calR^{\ac}}\Rsa^{\ac}.
\end{equation*}
Both $\Rsa^{\ac}$ and $\Tsa^{\ac}$ are finite of $p\,$-power order and
\begin{equation*}
	\calR^{\ac}=\varprojlim_{\sfrak,\ga}\Rsa^{\ac}\quad\text{and}\quad \T^{\ac}=\varprojlim_{\sfrak,\ga}\Tsa^{\ac}.
\end{equation*}

\begin{lemma}\label{lem:iwasawa-primes}
	Let $\ell\in\admissible_\sfrak$ and let $\gl=(\ell)$ be the prime of $K$ above $\ell$. Then
	\begin{enumerate}[label=\emph{(\roman*)}]
		\item $\Fr_\gl$ acts as the identity on $\Tsa^{\ac}$ for every $\ga\ge 0$.\label{condition:iwasawa-Frobenius-trivial}
		\item Let $n$ be a positive integer coprime with $NpD_K$ and let $T^{\ac}$ be any quotient of $\T^{\ac}$ by an ideal of $\calR^{\ac}$. Then $\ho(K[n],T^{\ac})=0$.\label{condition:iwasawa-no-galois-invariants}
	\end{enumerate}
\end{lemma}
\begin{proof}
	\emph{\ref{condition:iwasawa-Frobenius-trivial}} Since $\ell\in \admissible_\sfrak$, we have that $\Fr_\gl$ acts as the identity on $\Ts$. Moreover, $\Fr_\gl$ acts as the identity also on $\gL^{\ac}$, since $\gl$ is split in $K_\infty/K$.
	
	\emph{\ref{condition:iwasawa-no-galois-invariants}} Since the residual $G_K$-representation $\bar{\T}^{\ac}$ of $\T^{\ac}$ coincides with the residual representation of $\T$, by Lemma \ref{lem:no-invariants-T} (that is a consequence of Assumption \ref{ass:assumptions-on-T}  \ref{condition:assumption-irreducible-residual-representation}) we know that $\ho(K[n],\bar{\T}^{\ac})=0$. We obtain the claim by applying Lemma \ref{lem:nakayama-for-irred-repr}.
\end{proof}

Point \emph{\ref{condition:iwasawa-Frobenius-trivial}} of Lemma \ref{lem:iwasawa-primes} implies that the maps $\alpha_\ell$, $\beta_\ell$ and the finite-singular isomorphism $\phi_\ell^{\fs} \colon \honef(K_\gl,\Tsa^{\ac})\to \hones(K_\gl,\Tsa^{\ac})\otimes\calG_\ell$ of Definition \ref{def:finite-singular} are well defined for every $\ell\in\admissible_\sfrak$.

\subsection{Shapiro's lemma}\label{subsec:Shapiro's-lemma}

Just for this subsection, let $F$ be any perfect field. Let $F_\infty/F$ be a $\Zp$-extension and, for every $\ga\ge 0$, denote by $F_\ga$ the $\ga$-th layer of the extension. Let moreover $T$ be a $\Zp\llbracket G_F\rrbracket $-module. We allow $G_F$ to act on both factors of $T\otimes_{\Zp}\Zp[\Gal(F_\ga/F)]$ and $T \ \hat{\otimes}_{\Zp}\Zp\llbracket \Gal(F_\infty/F)\rrbracket $ via the natural projection.

\begin{lemma}\label{lem:Shapiro application}
	Let $\ga\ge 0$. Shapiro's lemma induces isomorphisms
	\begin{enumerate}[label=\emph{(\roman*)}]
		\item $\Sh_\ga\colon \hone(F_\ga,T)\cong \hone\big(F, T\otimes_{\Zp}\Zp[\Gal(F_\ga/F)]\big)$;\label{item:shapiro-alpha}
		\item $\Sh_\infty\colon \varprojlim_{\ga} \hone(F_\ga,T)\cong \hone\big(F, T \ \hat{\otimes}_{\Zp}\Zp\llbracket \Gal(F_\infty/F)\rrbracket \big)$, where the inverse limit is taken with respect to the corestriction maps.\label{condition:shapiro-anticyclotomic-p-extension}
	\end{enumerate}
\end{lemma}
\begin{proof}
	Shapiro's lemma yields an isomorphism $\hone(F_\ga, T)\cong \hone(F,\Ind_{F_\ga}^F T)$, where $\Ind_{F_\ga}^F T$ is the (co)induction of $T$ from $G_{F_\ga}$ to $G_F$. Since the extension $F_\ga/F$ is finite and normal, the isomorphism between induction and coinduction (see e.g.~\cite[Section I.6]{neukirch:cnf}) yields the isomorphism $\Ind_{F_\ga}^F T\cong T\otimes_{\Zp}\Zp[\Gal(F_\ga/F)]$, that proves point \emph{\ref{item:shapiro-alpha}}. Point \emph{\ref{condition:shapiro-anticyclotomic-p-extension}} is obtained by taking inverse limits. For a more explicit approach, see e.g.~\cite[Proposition II.1.1]{Col}.
\end{proof}

Coming back to our usual notation, for every $\sfrak=(s_1,s_2)\in\Z_{>0}^{2}$ and $\ga\ge 0$,  Lemma \ref{lem:Shapiro application} yields isomorphisms
\begin{equation}\label{eq:shapiro-isomorphisms}
	\Sh_\ga\colon \hone(K_\ga,\Ts)\cong \hone(K,\Tsa^{\ac})\quad\text{and}\quad  \Sh_\infty\colon \varprojlim_\ga \hone(K_\ga,\T\hspace{1pt})\cong \hone(K,\T^{\ac}).
\end{equation}
Therefore, studying the arithmetic of $\T$ over the anticyclotomic tower is equivalent to studying the arithmetic of $\T^\ac$ (and its quotients) over $K$.

\subsection{$p\hspace{1pt}$-complete anticyclotomic Euler systems and universal Kolyvagin systems} \label{sec:p-complete-euler-systems}

We now want to modify the notion of anticyclotomic Euler system given in Definition \ref{def:euler-systems}, requiring also some compatibilities at powers of $p$. 

\begin{definition}
	Denote by $\squarefreeadmissible^{(p)}$ the set of products of elements of $\squarefreeadmissible$ and powers of $p$. For every subset $\admissible'$ of $\admissible$, denote by $\squarefreeadmissible^{(p)'}$ the set of products of elements of $\squarefreeadmissible'$ and powers of $p$.
\end{definition}

We can write every element $c\in \squarefreeadmissible^{(p)}$ as $c=np^\ga$ for unique $n\in\squarefreeadmissible$ and $\ga\ge 0$. As in Section \ref{subsec:anticyclotomic-Euler-systems}, let $\admissible'$ be an infinite subset of $\admissible$ and let $\lcond(=\{\lcond_L\})$ be a collection of Selmer structures on $\T$ over every subfield $L$ that is a subextension of $K[c]/K$ for some $c\in \squarefreeadmissible^{(p)}$. Let $\abf = \{\abf_\ell\}_{\ell \in \admissible'}$ be a set of elements of $\calR$ such that $\Tr(\hspace{1pt} \Fr_\ell \hspace{1pt} \vert \hspace{1pt} \T\hspace{1pt}) = \ubf_\ell\abf_\ell$, where $\ubf_\ell$ is a unit of $\calR$ satisfying Assumption \ref{ass:u-ell}.

\begin{definition}\label{def:p-complete-euler-systems}
	A \emph{$p\hspace{1pt}$-complete anticyclotomic Euler system} attached to the triple $(\T,\lcond,\admissible')$ and relative to the set $\abf$ is a collection $\{\bbf(np^{\ga})\}_{np^{\ga}\in\squarefreeadmissible^{(p)'}}$ of classes $\bbf(np^{\ga})\in \hone_{\lcond}(K_\ga[n],\T\hspace{1pt})$ such that, for every $\ell\in\admissible'$ prime to $n$, we have:
	\begin{enumerate}[label=(pE\arabic*), start=0]
		\item $\Cor^{K_{\ga+1}[n]}_{K_\ga[n]}\bbf(n p^{\ga+1})=\bbf(n p^{\ga})$;\label{pE0}
		\item $\Cor^{K_\ga[n\ell]}_{K_\ga[n]} \bbf(n\ell p^\ga)=\abf_\ell\bbf(np^\ga)$;\label{pE1}
		\item $\loc_{\lambda_{n\ell}}\bbf(n\ell p^\ga) = \res^{K_\ga[n\ell]_{\lambda_{n\ell}}}_{K_\ga[n]_{\lambda_{n}}}\bigl(\Fr_\ell \loc_{\lambda_n}\bbf(np^\ga)\bigr)$ in $\hone(K_\ga[n\ell]_{\gl_{n\ell}},\T\hspace{1pt})$, for every prime $\gl_{n\ell}$ of $K_\ga[n\ell]$ above $\ell$ and setting $\lambda_n = \lambda_{n\ell} \cap K_\ga[n]$.\label{pE2}
	\end{enumerate}
	The $\calR$-module of $p\hspace{1pt}$-complete anticyclotomic Euler systems for $(\T,\lcond,\admissible')$ and relative to $\abf$ will be denoted by $\ES^{(p)}(\T,\lcond,\admissible', \abf)$.
\end{definition}

When $n=1$, condition \ref{pE0} implies that the set $\{\bbf(p^\ga)\}_{\ga\ge 0}$ is compatible with respect to the corestriction maps, therefore we can define
\begin{equation*}
	\bbf(1)^{\ac}:=\{\bbf(p^\ga)\}_{\ga\ge 0}\in \varprojlim_\ga \hone(K_\ga[1],\T\hspace{1pt})\cong \hone(K[1],\T^{\ac}),
\end{equation*}
where the isomorphism is the  one of point \emph{\ref{condition:shapiro-anticyclotomic-p-extension}} of Lemma \ref{lem:Shapiro application}, with $F=K[1]$. We also define the \emph{basic} class of the system $\bbf^\ac_K := \Cor^{K[1]}_K \bbf(1)^\ac \in \hone(K,\T^{\ac})$.

Let now $\lcond$ be a Selmer structure on $\T^{\ac}$ over $K$, that induces Selmer structures on the quotients $\Tsa^{\ac}$ by propagation. All the constructions of Section \ref{subsec:modified-Kolyvagin-systems} can be performed with $\T^{\ac}$ and $\Tsa^{\ac}$ in place of $\T$ and $\Ts$, respectively. Indeed, if we choose a subset $\admissible'$ of $\admissible_{\boldsymbol{1}}$ and an automorphism $\chi_{n,\ell}\in\Aut(\T^{\ac})$ for every couple $(n,\ell)$ with $n\ell\in\squarefreeadmissible'$, we have the $\calR^{\ac}$-module of \emph{universal $\{\chi_{n,\ell}\}$-Kolyvagin systems} for $(\T^{\ac},\lcond,\admissible')$
\begin{equation*}
	\KSuni(\T^{\ac},\lcond,\admissible',\{\chi_{n,\ell}\}):=\varprojlim_{\sfrak,\ga}\left(\varinjlim_{\tfrak\succeq \sfrak} \KS(\Tsa^{\ac},\lcond,\admissible_{\tfrak}',\{\chi_{n,\ell}\})\right).
\end{equation*}
Again, for a universal Kolyvagin system $\koly^{\ac}\in \KSuni(\T^{\ac},\lcond,\admissible',\{\chi_{n,\ell}\})$ induced by a set of classes $\{\koly(n)^{\ac}_{\sfrak,\ga}\}_{np^\alpha\in{\squarefreeadmissible_\sfrak^{(p)}}'}\in\KS(\Tsa^{\ac},\lcond,\admissible_\sfrak',\{\chi_{n,\ell}\})$, one can define
\begin{equation*}
	\koly(1)^{\ac}:=\varprojlim_{\sfrak,\ga}\koly(1)_{\sfrak,\ga}^{\ac}\in\varprojlim_{\sfrak,\ga}\hone(K,\Tsa^{\ac})=\hone(K,\T^{\ac})\cong \varprojlim_\ga \hone(K_\ga,\T\hspace{1pt}),
\end{equation*}
where the last isomorphism is \eqref{eq:shapiro-isomorphisms}.

From now on, denote by $\admissible_\sfrak'$ and by $\admissible'$ the sets of admissible primes of Definition \ref{def:Kolyvagin-primes}. One defines the \emph{relaxed Selmer structure $\lcond_{\rel}$} on $\T^{\ac}$ by replacing $\T$ with $\T^{\ac}$ in Definition \ref{dfn:relaxed-selmer-structure}. The main result of this section is the following theorem, to be interpreted as an Iwasawa-theoretic parallel to Theorem \ref{th:euler-to-kolyvagin-system}.

\begin{theorem}\label{th:Iwasawa-euler-to-kolyvagin-system}
	If $\{\bbf(np^\ga)\}_{np^\ga\in{\squarefreeadmissible^{(p)}}'}\in\ES^{(p)}(\T,\lcond_{\rel},\admissible', \abf)$, then there are a set of automorphisms $\{\chi_{n,\ell}\}$ of $\T$ and a universal $\{\chi_{n,\ell}\}$-Kolyvagin system \[
	\koly^{\ac}\in\KSuni(\T^{\ac},\lcond_{\rel},\admissible',\{\chi_{n,\ell}\})
	\]
	such that $\koly(1)^{\ac}=\bbf^{\ac}_K$.
\end{theorem}

\begin{remark}\label{rk:stressed-local-condition-iwasawa}
	As observed in Remark \ref{remark:stressed-local-conditions} for Theorem \ref{th:euler-to-kolyvagin-system}, if $\lcond$ is another Selmer structure on $\T^\ac$ over $K$ and $\{\bbf(np^\ga)\}_{np^\ga\in{\squarefreeadmissible^{(p)}}'}\in\ES^{(p)}(\T,\lcond_\rel,\admissible', \abf)$, it follows by the previous theorem that  $\koly^{\ac}\in\KSuni(\T^{\ac},\lcond,\admissible',\{\chi_{n,\ell}\})$, as long as the classes $\kappa(n)^\ac_{\sfrak, \alpha}$ of Definition \ref{def:kolyvagin-derivative-iwasawa} satisfy the local conditions of $\lcond$ at the primes dividing $Np$. 
\end{remark}
We devote the rest of this section to the proof of this theorem. A lot of ideas are very similar to the ones occurred in the proof of Theorem \ref{th:euler-to-kolyvagin-system}, so we will be more sketchy.

\subsection{Derivative classes and local properties}\label{sec:iwasawa-derivative-classes}

Fix $\{\bbf(np^\ga)\}_{np^\ga\in{\squarefreeadmissible^{(p)}}'}\in\ES(\T,\lcond_{\rel},\admissible', \abf)$ and, for this section, $\sfrak=(s_1,s_2)\in\Z_{>0}^2$, $\ga\ge 0$. For any  $n\in\squarefreeadmissible_{\sfrak}'$, we  write $\bbf(np^\ga)_{\sfrak}$ and $\bbf(np^\ga)_{\sfrak}'$ for the image of $\bbf(np^{\ga})$ in $\hone(K_\ga[n],\Ts)$ and $\hone(K_\ga[n],\Ts')$, respectively.

Let $n \in \squarefreeadmissible_\sfrak'$ and  $\ell\in\squarefreeadmissible_\tfrak'$, with $\tfrak=(t_1, t_2)\succeq \sfrak$ such that $\ell\nmid n$. Fix a prime $\lambda_{n\ell}$ of $K_\ga[n\ell]$ above $\lambda = (\ell)\subseteq K$; this in turn induces primes $\lambda_n$ (resp.~$\lambda_{\ell}$, resp.~$\lambda_{\ell}'$) of $K_\ga[n]$ (resp.~$K_\ga[\ell]$, resp.~$K[\ell]$).  As done in Section \ref{sec:dictionary}, we want to adapt the results of Appendix \ref{sec:formula-abstract} to our setting. This time, we use the following dictionary: 
\medskip
\begin{center}
	\def\arraystretch{1.5}
	\begin{tabular}{c|c|c|c|c|c|c|c|c|c|c}
		Appendix \ref{sec:formula-abstract}& $\Gtilde$ & $G$ & $H$  &$\Gtilde_0$ &$G_0$ &$H_0$  &$\gs$ &$c$ &$M$ &$D$ \\
		\hline
		Section \ref{sec:Iwasawa-theory} &$G_{K_\ga[n]^+}$  &$G_{K_\ga[n]}$  &$G_{K_\ga[n\ell]}$  &$G_{\Q_\ell}$ &$G_{K_\gl}$ &$G_{K[\ell]_{\gl_\ell'}}$  &$\gs_\ell$ &$\tau_c$ &$\ell+1$ &$D_\ell$\\
		\hline
	\end{tabular}
\end{center}

\begin{center}
	\def\arraystretch{1.5}
	\begin{tabular}{c|c|c|c|c|c|c|c|c|c|c}
		Appendix \ref{sec:formula-abstract} &$d$ & $R$ & $T$ & $s$  &$T/p^sT$ &$\mathbf{x}$ &$\mathbf{y}$  &$M_1$ &$\phi$ &$\delta$\\
		\hline
		Section \ref{sec:Iwasawa-theory}   &$\ell$ &$\calR_\tfrak'$  &$\T_\tfrak'$  &$t_1$  &$\T_\tfrak$ &$D_{n} \bbf(np^\ga)_{\tfrak}'$ &$D_{n}\bbf(n\ell p^\ga)_{\tfrak}'$  &$\abf_\ell$ &$\Fr_\ell$ &$\ubf_\ell$\\
		\hline
	\end{tabular}
\end{center}
\medskip
where $K_\ga[n]^+$ is the subextension of $K_\ga[n]$ fixed by the complex conjugation $\tau_c$. Notice that there are equalities $K_\ga[n]_{\gl_n}=K_\gl$ and $K_\ga[n\ell]_{\gl_{n\ell}}=K_\ga[\ell]_{\gl_\ell}=K[\ell]_{\gl_\ell'}$. Also, we have canonical isomorphisms $\Gal(K_\ga[n\ell]_{\lambda_{n\ell}}/K_\ga[n]_{\lambda_n}) \cong \Gal(K_\ga[n\ell]/K_\ga[n]) \cong \calG_\ell$.

Exactly as in Section \ref{sec:dictionary}, one can show that conditions \ref{condition:pro-dihedrality}--\ref{condition:semidirect-product}, \ref{condition:structure-T}--\ref{condition:frob-square-identity-mod-p-s}, \ref{condition:corestriction}, \ref{condition:char-poly-frob} and \ref{condition:eichler-shimura} of Appendix \ref{sec:formula-abstract} hold, being a consequence of the assumptions on $\T$ and of the properties defining a $p\hspace{1pt}$-complete anticyclotomic Euler system. The only condition that requires a further argument is \ref{condition:vanishing-h-0}, that is dealt in the following lemma.

\begin{lemma}\label{lem:iwasawa-no-invariants}
	If $T$ is any quotient of $\T$ by an ideal of $\calR$, then $\ho(K_\ga[n],T)=0$.
\end{lemma}
\begin{proof}
	By applying Shapiro's lemma, we have that
	\begin{equation*}
		\ho(K_\ga[n],T)\cong \ho(K[n],\Ind_{K_\ga[n]}^{K[n]}T)=\ho(K[n],\Ind_{K_\ga}^{K}T).
	\end{equation*}
	Since $\Ind_{K_\ga}^{K}T\cong T\otimes_{\Zp}\Zp[\Gal(K_\ga/K)]$ is a quotient of $\T^{\ac}$, we conclude by applying point \emph{\ref{condition:iwasawa-no-galois-invariants}} of Lemma \ref{lem:iwasawa-primes}.
\end{proof}

Similarly to what is done in the beginning of Section \ref{sec:Derivative-classes-and-local-properties}, we can now descend the classes $D_n\bbf(np^\ga)_{\sfrak}$ to $\hone(K,\Tsa^\ac)$ using the following result.

\begin{lemma}
	For every $n\in\squarefreeadmissible_{\sfrak}'$ we have
	\begin{enumerate}[label=\emph{(\alph*)}]
		\item $D_n \bbf(np^\ga)_{\sfrak}\in \hone(K_\ga[n],\Ts)^{\calG_n}$;\label{item:invariance-Iwasawa}
		\item the restriction maps $\hone(K_\ga[1],\Ts)\to \hone(K_\ga[n],\Ts)^{\calG_n}$ is an isomorphism.\label{item:iso-restriction-Iwasawa}
	\end{enumerate}
\end{lemma}
\begin{proof}
	\emph{\ref{item:invariance-Iwasawa}} This is dealt exactly as in point (a) of Lemma \ref{lemma:definition-kolyvagin-derivative}, by using \ref{pE1} in place of \ref{E1}.
	
	\emph{\ref{item:iso-restriction-Iwasawa}} It follows by applying the inflation-restriction exact sequence and Lemma \ref{lem:iwasawa-no-invariants}.
\end{proof}

\begin{definition}\label{def:kolyvagin-derivative-iwasawa}
	For every $n \in \squarefreeadmissible_\sfrak'$, define, denoting by $\Sh_\ga$ the isomorphism defined in \eqref{eq:shapiro-isomorphisms},  $\gk(n)_{\sfrak,\ga}:=\big(\Sh_\ga\circ\Cor^{K_\ga[1]}_{K_\ga}\circ(\res^{K_\ga[n]}_{K_\ga[1]})^{-1}\big)(D_n\bbf(np^\ga)_{\sfrak})\in \hone(K,\T_{\sfrak,\ga}^{\ac})$.
\end{definition}

We will show that a slight modification of the classes $\gk(n)_{\sfrak,\ga}$ form a modified universal Kolyvagin system. The first step in this direction is the following proposition.

\begin{proposition}\label{prop:Iwasawa-derived-classes-in-the-Selmer}
	For any $n\in\squarefreeadmissible_{\sfrak}'$, we have that $\gk(n)_{\sfrak,\ga}\in \hone_{\lcond_{\rel}(n)}(K,\T_{\sfrak,\ga}^{\ac})$.
\end{proposition}
\begin{proof}
	The same argument used in the proof of Proposition \ref{prop:derived-classes-in-the-Selmer} can be used here to show that $\big(\Cor^{K_\ga[1]}_{K_\ga}\circ(\res^{K_\ga[n]}_{K_\ga[1]})^{-1}\big)(D_n\bbf(np^\ga)_{\sfrak})\in H_{\lcond_{\rel}(n)}^1(K_\ga,\T_{\sfrak,\ga}^{\ac})$. Then one concludes using the compatibility of $\Sh_\ga$ in semi-local Galois cohomology, as recorded in Appendix \ref{app:Shapiro}.
\end{proof}

\subsection{From Euler systems to Kolyvagin systems}\label{sec:euler-to-kolyvagin-iwasawa}

As in Section \ref{sec:from-Euler-to-Kolyvagin-systems}, we are now able to apply the abstract work of Appendix \ref{sec:formula-abstract} in order to prove Theorem \ref{th:Iwasawa-euler-to-kolyvagin-system}.

\begin{lemma}\label{lem:formula-capitolo-3}
	Let $\sfrak\in \Z_{>0}^2$ and $\alpha \ge 0$. For any $n\ell\in\squarefreeadmissible_\sfrak'$ with $\ell$ prime, there is $\vartheta_\ell\in\Aut(\Tsa^\ac)$ such that
	\begin{equation*}
		\ga_\ell\big(\loc_{\gl}\gk(n)_{\sfrak,\ga}\big)=\vartheta_{\ell}\big(\beta_\ell([\loc_{\gl}\gk(n\ell)_{\sfrak,\ga}]_{\s} \otimes \sigma_\ell)\big)
	\end{equation*}
	as elements of $\Tsa^\ac$, where $\gl$ is the prime of $K$ above $\ell$.
\end{lemma}
\begin{proof}
	By applying the dictionary of Section \ref{sec:iwasawa-derivative-classes} to the key formula of Proposition \ref{prop:key-formula-appendix} and reasoning as in the proof of Lemma \ref{lem:formula-capitolo-2}, one obtains the relation 
	\begin{equation}\label{eq:formula-iwasawa}
		\ga_\ell\big(\loc_{\gl_\ga}\tilde{\gk}(np^\ga)_\sfrak\big)=\vartheta_{\gl_\ga}\big(\beta_\ell([\loc_{\gl_\ga}\tilde{\kappa}(n\ell p^{\ga})_\sfrak]_{\s}\otimes\sigma_\ell)\big)
	\end{equation}
	in $\Ts$, where $\gl_\ga$ is any prime of $K_\ga$ above $\gl$, the map $\vartheta_{\gl_\ga}$ is an automorphism of $\Ts$ and
	\begin{gather*}
		\tilde{\kappa}(np^{\ga})_{\sfrak} := \Cor_{K_\ga}^{K_\ga[1]}(\res^{K_\ga[n]}_{K_\ga[1]})^{-1}D_n\bbf(np^{\ga})_{\sfrak}, \\
		\tilde{\kappa}(n\ell p^{\ga})_{\sfrak} := \Cor_{K_\ga}^{K_\ga[1]}(\res^{K_\ga[n\ell]}_{K_\ga[1]})^{-1}D_{n\ell}\bbf(n\ell p^{\ga})_{\sfrak}.
	\end{gather*}
	By making the primes $\gl_\ga\mid\gl$ vary, one can define  
	\begin{equation*} \vartheta_\ell=\big(\vartheta_{\gl_\ga}\big)_{\gl_\ga\mid\gl}\in\Aut\Bigg(\bigoplus_{\gl_\ga\mid\gl}\Ts\Bigg).
	\end{equation*}
	By applying $\Sh_\ga$ to the semi-local counterpart of equation \eqref{eq:formula-iwasawa} and using the functoriality of semi-local Shapiro's lemma (see Appendix \ref{app:Shapiro}), we obtain that the automorphism $\vartheta_\ell$ induces the automorphism of $\Tsa^\ac$ that yields the claimed relation.
\end{proof}

\begin{proof}[Proof of Theorem \ref{th:Iwasawa-euler-to-kolyvagin-system}]
	The combination of Proposition \ref{prop:Iwasawa-derived-classes-in-the-Selmer} and Lemma \ref{lem:formula-capitolo-3} implies that the classes
	\begin{equation*}
		\koly(n)_{\sfrak,\ga}^\ac:=\gk(n)_{\sfrak,\ga}\otimes\bigotimes_{\ell'\mid n}\sigma_{\ell'}\in \hone(K,\Tsa^\ac)\otimes\calG(n)
	\end{equation*}
	form a $\{\vartheta_\ell^{-1}\}$-Kolyvagin system for $(\Tsa^\ac,\lcond_{\rel},\admissible_\sfrak')$. 
	
	We now need to interpolate them into a universal $\{\vartheta_\ell^{-1}\}$-Kolyvagin system for $(\T^\ac,\lcond_{\rel},\admissible')$ with the required properties. The interpolation along the $\sfrak$-variable is carried out exactly as in the proof of Theorem \ref{th:euler-to-kolyvagin-system}. The interpolation along the $\alpha$-variable is a direct consequence of the property \ref{pE0} (for the classes $\bbf(np^\alpha)$) together with the fact that $\proj^{\alpha+1}_\alpha\circ\Sh_\alpha=\Sh_{\alpha+1}\circ\Cor^{K_{\alpha+1}}_{K_\alpha}$, where $\proj^{\alpha+1}_\alpha\colon H^1(K,\T_{\sfrak,\alpha+1}^\ac)\to H^1(K,\T_{\sfrak,\alpha}^\ac)$ is the map induced by the natural projection $\T_{\sfrak,\alpha+1}^\ac\to \T_{\sfrak,\alpha}^\ac$ (see e.g.~the proof of \cite[Proposition II.1.1]{Colmez:theorie-dIwasawa} for this last property).
\end{proof}

\subsection{Applications}\label{sec:Iwasawa-applications}

As done in Section \ref{sec:applications}, we show here that modified universal Kolyvagin systems for $\T^\ac$ have the same usage of classical (universal) Kolyvagin systems when it comes to study the anticyclotomic Iwasawa theory for $\T^\ac$.

For this section, let's assume that $\calR$ is a DVR and keep the notation of Section \ref{sec:applications}. If $M$ is a $\Zp$-module, denote by $M^\vee:=\Hom_{\cont}(M,\Qp/\Zp)$ the Pontryagin dual of $M$. We will denote by $\iota\colon \calR^\ac\to\calR^\ac$ the involution induced by the inversion in $\Gamma^\ac$ and by $\Char_{\calR^\ac}(M)$ the characteristic ideal of any torsion $\calR^\ac$-module $M$. Eventually, let $\A^\ac:=\T^\ac\otimes_{\calR^{\ac}} (\calR^\ac)^\vee$.

The next result is a counterpart of \cite[Theorem 3.5]{longo-vigni:generalized}.

\begin{theorem}\label{th:howard-abstract-iwasawa}
	Let $\lcond$ be a Selmer structure on $\T^{\ac}$ satisfying \cite[Hypotheses H.1-H.5]{howard:heegner-points} and \cite[Assumption 3.2]{longo-vigni:generalized}. Suppose that there are $\chi_{n,\ell}\in\Aut(\T^\ac)$ such that there is an element $\koly^\ac\in\KSuni(\T^\ac,\lcond,\admissible',\{\chi_{n,\ell}\})$ and $\admissible'\subseteq\admissible$ with the property that $\koly(1)^\ac\ne 0$. Then
	\begin{enumerate}[label=\emph{(\roman*)}]
		\item $\hone_{\lcond}(K,\T^\ac)$ is a torsion-free $\calR^\ac$-module of rank $1$;
		\item there exists a finitely generated torsion $\calR^\ac$-module $M$ such that $\Char_{\calR^\ac}(M)=\Char_{\calR^\ac}(M)^\iota$ and a pseudo-isomorphism
		\begin{equation*}
			\hone_{\lcond}(K,\A^\ac)^\vee\sim \calR^\ac\oplus M\oplus M;
		\end{equation*}
		\item $\Char_{\calR^\ac}(M)$ divides $\Char_{\calR^\ac}\big(\hone_{\lcond}(K,\T^\ac)/\calR^\ac \koly(1)^\ac\big)$.
	\end{enumerate}
\end{theorem}
\begin{proof}
	The proof proceeds as that of \cite[Theorem 3.5]{longo-vigni:generalized}, by replacing \cite[Proposition 1.6.1]{howard:heegner-points} (used in the proof of \cite[Proposition 3.3]{longo-vigni:generalized}) with Theorem \ref{thm:howard-abstract}. In fact, the condition $\ho(K_\infty, \A)=0$ here is a consequence of Lemma \ref{lem:no-invariants-T}.
\end{proof}

Note that the the last part of the previous theorem is one divisibility of the Perrin-Riou's formulation of the anticyclotomic Iwasawa main conjecture. It was firstly stated, in the elliptic curve case, in \cite{perrin-riou:iwasawa-heegner} and it has been proven in some concrete cases. We will come back on this topic in Section \ref{sec:examples}. 

\section{Examples}\label{sec:examples}

In this section we list the main examples of ($p\hspace{1pt}$-complete) anticyclotomic Euler systems (in the sense of Definitions \ref{def:euler-systems} and \ref{def:p-complete-euler-systems}) that appear in the literature and we show how our theory can be applied to recover some existing results. Throughout, $p$ will be an odd prime, $K$ will be an imaginary quadratic field and $N$ a positive integer not divisible by $p$ satisfying the \emph{Heegner hypothesis} in $K$, i.e., such that the primes dividing $N$ split in $K$.

\begin{remark}
	We assume the Heegner hypothesis for the sake of simplicity, but this is not necessary, as the results of the previous sections hold with the only assumption that the prime factors of $N$ do not ramify. In fact, in the literature one can find alternative versions of the objects that we will introduce that are built under a less restrictive hypothesis, usually called the \emph{generalized Heegner hypothesis}. This consists in assuming $N=N^+N^-$, where the primes dividing $N^+$ split in $K$ and $N^-$ is the squarefree product of an even number of primes that are inert in $K$. These constructions are usually performed replacing the base modular curve with an indefinite quaternionic Shimura curve. Case by case, we will briefly point out the corresponding theory under this generalized Heegner hypothesis .
\end{remark}

\subsection{Elliptic curves}

We first consider the anticyclotomic Euler system of Heegner points on elliptic curves. It dates back to the papers of Kolyvagin  \cite{KolFin}, \cite{KolES}, where he used this tool in order to bound the size of the Shafarevich-Tate group of elliptic curves and to attack the BSD conjecture in the case of analytic rank 1.  
The method used in those papers was later rewritten in the language of Kolyvagin systems by Howard in \cite{howard:heegner-points}.

Let $E$ be an elliptic curve defined over $\Q$ without complex multiplication, of conductor $N$. Its Tate module $T_p E$ is a $G_\Q$-representation of rank $2$ over $\Zp$ and, setting $J_{s_2}=\{0\}\subseteq\Z_p$ for every $s_2\in\Z_{>0}$, the couple $(\calR, \T\hspace{1pt}) = (\Zp, T_pE)$ satisfies Assumptions \ref{ass:assumptions-on-R} and \ref{ass:assumptions-on-T} if we suppose that $E[p]$ is irreducible and that the image of the representation contains the scalars $1 + p\Zp$. This is true if for instance the representation is surjective, which holds for $p$ large enough thanks to Serre's open image theorem (see \cite[Th\' eoreme 3]{serre1972:proprietes-galoisienne}). In particular, as done in Section \ref{sec:applications}, we will work with the finite quotients $T_pE/p^s T_pE \cong E[p^s]$ for $s\in \Z_{>0}$.

Fix a modular parametrization $X_0(N) \surj E$ and a cyclic $N$-ideal $\mathfrak{n}$ in the ring of integers $\mathcal{O}_K$ of $K$ (that exists thanks to the Heegner hypothesis). For every $n\in\N$ coprime with $N$, the Heegner point $P_n \in E(K[n])$ of conductor $n$ is the image via the modular parametrization of the point $x_n\in X_0(N)$ corresponding to the $N$-cyclic isogeny $\C/\mathcal{O}_n \to \C/(\mathcal{O}_n\cap \mathfrak{n})^{-1}$, where $\mathcal{O}_n = \Z + n\mathcal{O}_K$ is the order of conductor $n$ in $K$.
By applying the Kummer map
\[
\delta_{K[n]} \colon E(K[n]) \otimes \Zp \to \hone(K[n], T_p E),
\]
we define the classes $\cbf(n)_\hp := \delta_{K[n]}(P_n)$. Moreover, let, for every $s > 0$,  $\Sel_{p^s}(E/K[n]) \subseteq \hone(K[n], E[p^s])$ be the $p^s$-th Selmer group of $E$ over $K[n]$ (for its definition see for instance \cite[Section X.4]{silverman:arithmetic-elliptic-curves}). By definition, $\cbf(n)_\hp \in S_p(E/K[n]) := \varprojlim_{s} \Sel_{p^s}(E/K[n])$. 

For any number field $F$, define the Selmer structure $\lcond_{\mathrm{Kum}}$ on $T_p E$ over $F$ to be the one attached to the local conditions defined as the inverse image, via the map induced by the inclusion $T_pE \inj V_pE:=T_pE \otimes_{\Zp} \Qp$, of
\[
\begin{cases}
	\hone_{\ur}(F_v, V_pE) := \ker\bigl(\hone(F_v, V_pE) \to \hone(F_v^\ur, V_pE)\bigr) \quad &\text{if $v \mid N$},\\
	\Imm(\delta_{F_v} \colon E(F_v) \otimes \Qp \to \hone(F, V_p E)) &\text{if $v \mid p$},
\end{cases}
\]
where $\delta_{F_v}$ is the local Kummer map. By \cite[Proposition 1.6.7, Remark 1.5.5 (i)]{rubin:euler-systems}, we have that $S_p(E/F) = \hone_{\lcond_{\mathrm{Kum}}}(F, T_pE)$. It follows by \emph{loc.~cit.}~that the Selmer group of the propagation of $\lcond_{\Kum}$ to $T_pE/p^sT_pE \cong E[p^s]$ coincides with $\Sel_{p^s}(E/F)$ for every $s\ge 1$, and the one associated with the propagation of $\lcond_{\Kum}$ to $T_pE\otimes_{\Zp}\Qp/\Zp \cong E[p^\infty]$ coincides with $\Sel_{p^\infty}(E/F) := \varinjlim_s \Sel_{p^s}(E/F)$.

For any prime $\ell \in \admissible$ (as defined in Definition \ref{def:set-P}), let $a_\ell := \ell +1 - \lvert \tilde{E}(\F_\ell) \rvert$, where $\tilde{E}$ is the reduction of $E$ modulo $\ell$. Write $\abf := \{a_\ell\}_{\ell \in \admissible}$ and note that $a_\ell = \Tr(\Fr_\ell \hspace{1pt}\vert\hspace{1pt} T_pE)$ (in other words, using the language of Section \ref{sec:euler-systems-and-kolyvagin-systems}, we set $\ubf_\ell =1$). Let $\admissible'$ be the infinite subset of $\admissible$ defined in Definition \ref{def:Kolyvagin-primes}, choosing $\Omega$ to be the whole set of indices (as Assumption \ref{ass:u-ell} is trivially satisfied) and denote by $\e_L$ the sign of the functional equation of $L(E, s)$.

\begin{theorem}\label{th:heegner-points}
	We have that $\{\cbf(n)_\hp\}_{n \in \squarefreeadmissible} \in \ES(T_p E, \lcond_\mathrm{Kum}, \admissible, \abf)$. Moreover, there is a universal modified Kolyvagin system $\koly_\hp\in\KSuni(T_p E,\lcond_{\mathrm{Kum}},\admissible', \{\chi_{n, \ell}\})$ such that $\koly(1)_\hp=\cbf_{\hp, K}$, where $\chi_{n, \ell}$ is the multiplication by $\e_n = (-1)^{\omega(n)-1}\e_L$.
\end{theorem}

\begin{proof}
	We already noted that $\cbf(n)_\hp \in \hone_{\lcond_{\mathrm{Kum}}}(K[n], T_p E) = S_p(E/K[n])$; \ref{E1} and \ref{E2} follow from the analogous properties for the Heegner points (see \cite[Proposition 3.7]{gross:kolyvagin}), by applying the Kummer map. This shows the first claim. 
	
	By Theorem \ref{th:euler-to-kolyvagin-system}, there is $\koly_\hp\in\KSuni(T_p E,\lcond_{\rel},\admissible', \{\chi_{n, \ell}\})$ such that $\koly(1)_\hp=\cbf_{\hp, K}$ for some $\chi_{n,\ell}\in\Aut(T_p E)$. The fact that $\chi_{n, \ell}=\e_n$ follows from the explicit description of Remark \ref{rk:E3}, since \ref{E3} holds for $w=-\e_L$ by \cite[Proposition 5.3]{gross:kolyvagin}. Therefore, as noted in Remark \ref{remark:stressed-local-conditions}, the last claim follows once we show that 
	$
	\loc_v \kappa(n)_{\hp, s} \in \hone_{\lcond_{\mathrm{Kum}}}(K_v, E[p^{s}])
	$
	for $v \mid Np$, for $ \kappa(n)_{\hp, s}$ as in Definition \ref{def:kolyvagin-derivative} with $\cbf(n)= \cbf(n)_{\hp}$. This is \cite[Proposition 6.2(a)]{gross:kolyvagin}, as his classes $c(n)$ are easily seen to coincide with our $\kappa(n)_{\hp, s}$.
\end{proof}

We can use the existence of $\koly_\hp$ to obtain the following structure theorem for Selmer groups.

\begin{corollary}[Kolyvagin]
	If $P_K =\Tr_{K[1]/K}(P_1) \ne 0$, then $S_p(E/K)$ is free of rank $1$ over $\Zp$ and there is a finite $\Zp$-module $M$ such that 
	\[
	\Sel_{p^\infty}(E/K) \cong (\Qp/\Zp) \oplus M \oplus M.
	\]
	Furthermore, we have that
	$
	\length_{\Zp}(M) \le \length_{\Zp}\bigl(S_p(E/K)/\Zp \cbf_{\hp, K}\bigr).
	$
\end{corollary}

\begin{proof}
	This is Theorem \ref{thm:howard-abstract} applied to $\koly_\hp$. In fact, by hypothesis and Theorem \ref{th:heegner-points} we have that $\koly_\hp(1)=\cbf_{\hp, K} = \delta_K(P_K)\ne 0$, and (H.2)-(H.5) of \cite[Section 1.3]{howard:heegner-points} hold as explained in \emph{loc.~cit.}, proof of Theorem 1.6.5.
\end{proof}

\begin{remark}\label{rk:elliptic-curves-id-kolyvagin-system}
	In this case, it is possible to modify $\koly_\hp$ into an  $\{\id\}$-universal Kolyvagin system (i.e., a universal Kolyvagin system in the classical sense) by multiplying the classes $\kappa(n)_{\hp, s}$ by $v_n:=\prod_{i=1}^{\omega(n)} \e_{\ell_1 \cdots \ell_{i}}$ for any chosen prime decomposition $n = \prod_{i=1}^{\omega(n)} \ell_i$. 
	
	Note that this is possible since, for every $m\mid n$, $\e_{m}$ depends only on $\omega(m)$, and therefore $v_n$ does not depend on the chosen prime decomposition. In general, when the modifying automorphisms $\chi_{n,\ell}$ are not elements of $\calR$ or depend also on the prime $\ell$, we haven't been able, up to our best effort, to obtain a genuine $\{\id\}$-universal Kolyvagin system out of a $\chi_{n,\ell}$-universal Kolyvagin system.
\end{remark}

Suppose now that $E$ has good ordinary reduction at $p$. For any place $v \mid p$ of $K$, denote by $\tilde{E}_v$ the reduction of $E$ at $v$ and by $F^+_v (T_p E)$ the kernel of the reduction map $T_p E \to T_p \tilde{E}_v$, which is a $G_{K_v}$-stable free sub\hspace{1pt}-$\Zp$-module of rank $1$. 

Consider the $G_K$-representations $\T^\ac = T_p E \otimes_{\Zp} \Lambda^\ac$ and $\A^\ac = \T^\ac \otimes_{\Lambda^\ac} (\Lambda^\ac)^\vee$.
We define the \emph{Greenberg Selmer structure} $\lcond_\Gr$ over $K$ on $M= \T^\ac, \A^\ac$ in the following way:  
we take the unramified condition at primes $v \nmid p$ and 
\[
\hone_\ord(K_v, M) = \Imm\bigl(\hone\bigl(K_v, F^+_v (M)\bigr) \to \hone(K_v, M)\bigr)
\]
for any $v\mid p$, where $F_v^+(\T^\ac) = F_v^+(T_p E) \otimes_{\Zp} \Lambda^\ac$ and $F_v^+(\A^\ac) = F_v^+(\T^{\ac}) \otimes_{\Lambda^\ac} (\Lambda^\ac)^\vee$.
It turns out that there are pseudoisomorphisms
\[
\hone_{\lcond_{\Gr}}(K, \T^\ac) \sim \varprojlim_n S_p(E/K_n) \quad \text{and} \quad \hone_{\lcond_{\Gr}}(K, \A^\ac)
\sim \varinjlim_n \Sel_{p^{\infty}}(E/K_n).
\]
We now use Heegner points in order to construct a $p\hspace{1pt}$-complete anticyclotomic Euler system.
For any $np^\alpha \in \squarefreeadmissible^{(p)}$, define $P_{\alpha, n} = \Tr_{K[np^{\alpha+1}]/K_\alpha[n]}(P_{np^{\alpha+1}}) \in E(K_\alpha[n])$. These point satisfy the norm relations 
\[
\Tr_{K_{\alpha +1}[n]/K_\alpha[n]}(P_{\alpha+1, n}) = a_p P_{\alpha, n} - P_{\alpha-1, n}
\]
for $\alpha \ge 1$, which Howard used in order to construct another class of points $Q_{\alpha, n} \in E(K_\alpha[n])$ that are compatible with respect to the maps $\Tr_{K_{\alpha +1}[n]/K_\alpha[n]}$ (see \cite[Lemma 2.3.3]{howard:heegner-points}). Let $\bbf(np^\alpha)_\hp = \delta_{K_\alpha[n]} (Q_{\alpha, n})\in \hone(K_\alpha[n], T_p E)$. It follows in particular that $\bbf^\ac_{K, \hp} \in \hone_{\lcond_{\Gr}}(K, \T^\ac)$.

\begin{theorem}
	We have that $\{\bbf(np^\ga)_\hp\}_{np^\ga\in\squarefreeadmissible^{(p)}}\in\ES^{(p)}(T_p E,\lcond_{\Kum},\admissible, \abf)$. Moreover, there is $\koly^{\ac}_\hp\in\KSuni(\T^{\ac},\lcond_{\Gr},\admissible', \{\chi_{n, \ell}\})$ such that $\koly(1)^\ac_{\hp, s}=\bbf^\ac_{\hp, K}$, where $\chi_{n, \ell}$ is the multiplication by $\e_n = (-1)^{\omega(n)-1}\e_L$.
\end{theorem}

\begin{proof}
	By definition, $\bbf(np^\alpha) \in S_p(E/K_{\alpha}[n])$, and \ref{pE0}-\ref{pE2} follow from the properties of the points $Q_{\alpha, n}$ by applying the Kummer map (see \cite[Lemma 2.3.3, proof of 2.3.5]{howard:heegner-points}). As observed in  Remark \ref{rk:stressed-local-condition-iwasawa}, in order to prove the second claim, we only need to apply Theorem \ref{th:Iwasawa-euler-to-kolyvagin-system} and show that $\loc_{v} \kappa(n)^\ac_\hp \in \hone_{\lcond_\Gr}(K, \T^\ac_{s, \alpha})$ for $v \mid Np$,  for $\kappa(n)_{\hp, s, \alpha}$ as in Definition \ref{def:kolyvagin-derivative-iwasawa} for $\bbf(np^\alpha) = \bbf(np^\alpha)_\hp$. This is \cite[Lemma 2.3.4]{howard:heegner-points}.
\end{proof}

The existence of a $p\hspace{1pt}$-complete anticyclotomic Euler system has one divisibility of the Perrin-Riou  formulation of the Iwasawa main conjecture as a corollary (see \cite{perrin-riou:iwasawa-heegner}). 
\begin{corollary}[Howard]
	$\hone_{\lcond_\Gr}(K, \T^\ac)$ is a torsion-free $\Lambda^\ac$-module of rank $1$. Moreover, there is a finitely generated torsion $\Lambda^\ac$-module $M$ such that
	\begin{enumerate}[label=\emph{(\roman*)}]
		\item $\Char(M)=\Char(M)^\iota$;
		\item $\hone_{\lcond_\Gr}(K, \A^\ac)^\vee \sim \Lambda^\ac \oplus M \oplus M$;
		\item $\Char(M) \supseteq \Char(\hone_{\lcond_\Gr}(K, \T^\ac)/\Lambda^\ac\bbf^\ac_{\hp, K})$.
	\end{enumerate}
\end{corollary}

\begin{proof}
	This is Theorem \ref{th:howard-abstract-iwasawa} for $\koly^\ac = \koly^\ac_\hp$, as $\koly^\ac_\hp(1) \ne 0$ by the main result of \cite{cornut:mazur-conjecture}.
\end{proof}

In \cite{perrin-riou:iwasawa-heegner}, Perrin-Riou conjectured that an equality should hold in the last point of the previous theorem. This fact is a reformulation of the anticyclotomic Iwasawa main conjecture in this special case. Using more sophisticated techniques one can prove also the other divisibility (and hence the full conjecture) under some mild  assumptions: see \cite{castella:p-adic-heights-heegner-points}, \cite{wan-heegner-point-imc}, \cite{burungale-castella-kim:perrin-riou-imc} and \cite{castella-wan:perrin-riou-imc-supersingular} in the supersingular setting.

\begin{remark}
	An Euler system of Heegner points is built under the (indefinite) generalized Heegner hypothesis in \cite{BD}. See also \cite{bertolini-darmon:imc-elliptic-curves-anticyclomic} and \cite{bertolini-longo-venerucci:imc-elliptic-curves} for the anticyclotomic Iwasawa main conjecture under a generalized Heegner hypothesis (also for the definite case).
\end{remark}

\subsection{Modular forms} 

Let $f = \sum_{i=1}^\infty a_i q^i$ be a cuspidal new-form  of even weight $k \ge 2$ and level $\Gamma_0(N)$. Let $F = \Q(a_i)_{i>0}$ be the Hecke field of $f$, let $\p \mid p$ be an ideal of $\calO_F$ and denote by $F_\p$ the completion of $F$ at $\p$, by $\calO_\p$ its valuation ring and choose a uniformizer $\pi$ of $\calO_\p$. Suppose that $p \nmid 6N\phi(N)(k-2)!$, where $\phi$ is the Euler's totient function. Inside the representation $V = V_\p(k/2)$, the $k/2$-th twist of the Deligne's representation attached to $f$, one can find (see \cite[Proposition 3.1]{nekovar:chow-groups}) an $\calO_\p$-lattice $T$ admitting a $G_{\Q}$-equivariant alternating perfect pairing $T \times T \to \calO_\p(1)$. 

Since $\calO_\p$ is a DVR, setting again $J_{s_2}=\{0\}$, for every $s_2\in\Z_{>0}$, the couple $(\calO_\p, T)$ satisfies Assumptions \ref{ass:assumptions-on-R} and \ref{ass:assumptions-on-T}, if we assume that $T/\pi T$ is irreducible and that the image of the the representation contains the scalars $1+p\Z_p$: this is true for instance if $p$ is non-excluded in the sense of \cite[Definition 6.1]{besser:finiteness-sha}; these primes are infinitely many (see \cite[Lemma 3.8]{longo-vigni:beilinson}). Note also that the characteristic polynomial of $\Fr_\ell$ over $T$ is $X^2 - a_\ell \ell^{1-k/2} X + \ell$ for any $\ell \nmid Np$. Again, since $J_{s_2}=0$ for every $s_2\in\Z_{>0}$, we are interested in the representations $T/p^sT$, for  $s\in \Z_{>0}$.

We assume also the simplifying assumption that the local Tamagawa numbers are trivial at the primes $v \mid N$. These are defined as the order of the finite module $\hone(\inertia_v, T)^{\Fr_v=1}_{\calO_\p-\tors}$ and are involved in a formulation of the Bloch-Kato conjecture (see \cite[Section I.4.2.2]{fontaine-perrin-riou:bloch-kato}).

For this representation and $K$ imaginary quadratic satisfying the Heegner hypothesis we may construct two anticyclotomic Euler systems: one starting from the (classical) Heegner cycles of \cite{nekovar:chow-groups}, the other starting with the generalized Heegner cycles of \cite{bdp:generalized}. The two constructions are analogous, therefore we will focus only on the latter, since generalized Heegner cycles may be used also in order to construct a $p\,$-complete anticyclotomic Euler system, as we will see.

Let $E$ be a canonical elliptic curve (in the sense of Gross, see \cite[Section 4.1]{castella-hsieh:heegner-cycles}) defined over $K[1]$ with CM by $\calO_K$ and fix $\xi \colon \C/\calO_K \iso E(\C)$ a complex uniformization. The generalized Kuga--Sato variety is the variety $X_{k-2} := \kugasato^{k-2}_{1, N} \times_{K[1]} E^{k-2}$ defined over $K[1]$, where by $\kugasato^{k-2}_{1, N}$ we denote the $(k-2)$-th Kuga--Sato variety over the modular curve $X_1(N)$ of level $\Gamma_1(N)$. Let $n$ be a positive integer, consider the elliptic curve $E_n = E/\mathcal{C}_n$, where $\mathcal{C}_n = \xi(n^{-1}\calO_n/\calO_K)$ is a cyclic subgroup of $E$ of order $n$ and let $\phi_n \colon E \to E_n$ be the quotient isogeny. Therefore, $E_n$ is endowed with the complex uniformization $\C/\calO_n \cong E_n(\C)$ and $\phi_n$ corresponds to the map $\C/\calO_K \to \C/\calO_n$ given by $z \mapsto nz$. Attached to $\phi_n$, one defines the cycle
\[
\Delta_n \in \chow^{k-1}(X_{k-2}/K[n])_0 \otimes \Zp 
\]
in the following way: let $\mathfrak{n}$ be the cyclic $N$-ideal, whose existence is ensured by the Heegner hypothesis relative to $N$, and let $\mu_n$ be a particular $\mathfrak{n}$-torsion point of $E_n$ (for its definition see \cite[Section 2.3]{castella-hsieh:heegner-cycles}). The couple $(E_n, \mu_n)$ determines a point $P_n$ of $X_1(N)$ via its moduli interpretation. Denote by $\iota_n \colon E_{n}^{k-2} \inj X_1(N)$ the embedding of $E_n^{k-2}$ as the fibre of $P_n$ with respect to the structural morphism of $\kugasato^{k-2}_{1, N}$. We define
\[
\Delta_n := (\varepsilon_X)_\ast \Upsilon_n,
\]
where $\varepsilon_X$ is the projector of \cite[(2.2.1)]{bdp:generalized} and $\Upsilon_n$ is the $(k-2)$-th self-product of the graph of $\phi_n$, seen as a cycle on $X_{k-2}$ via
\[
\Upsilon_n = \mathrm{Graph}(\phi_n)^{k-2} \subseteq (E_n \times E)^{k-2} = E_n^{k-2} \times E^{k-2} \xhookrightarrow{\iota_n \times \id_{E^{k-2}}} X_{k-2}.
\]
The image of $\Delta_n$ via a suitable Abel-Jacobi map (see \cite[Sections 4.2, 4.4]{castella-hsieh:heegner-cycles}) determines a cohomology class $\cbf(n)_{\gen} \in \hone(K[n], T)$ (denoted in \cite{castella-hsieh:heegner-cycles} by $z_{f, c, \chi}$, for $c=n$ and $\chi$ the trivial character). If follows from \cite[Theorem 3.1]{nekovar:p-adic-abel-jacobi} that $\cbf(n)_\gen \in \honef(K[n], T)$, where, for any number field $F$, $\honef(F, T)$ is the Bloch--Kato Selmer group, i.e., the Selmer group with respect to the local conditions $\lcond_\bk$ defined as the inverse image via the natural map $\hone(F, T) \to \hone(F, V)$ of
\[
\begin{cases}
	\hone_{\ur}(F_v, V) := \ker\bigl(\hone(F_v, V) \to \hone(F_v^\ur, V)\bigr) \quad &\text{if $v \mid N$},\\
	\ker\bigl(\hone(F_v, V) \to \hone(F_v, V \otimes B_{\mathrm{cris}})\bigr) &\text{if $v \mid p$}.
\end{cases}
\]
Write also $A := V/T$ and define the Bloch--Kato Selmer group $\honef(F, A)$ taking as Selmer structure the image of the above local conditions on $V$ via the natural map $\hone(F, V) \to \hone(F, A)$. Since $A[p^s] \cong T/p^sT$ for any $s \in \Z_{>0}$ we have that $\honef(F, A) = \varinjlim_s \honef(F, T/p^sT)$, where we induce the Selmer structure $\lcond_\bk$ on $T/p^sT$ by propagation.

Set $\abf = \{a_\ell\}_{\ell \in \admissible}$ and note that $\Tr(\Fr_\ell \hspace{1pt}\vert\hspace{1pt} T) = a_\ell \ell^{1-k/2}$ (i.e., $\ubf_\ell = \ell^{1-k/2}$). Let $\admissible'$ be the set of admissible primes of Definition \ref{def:Kolyvagin-primes}, for $\Omega$ the whole set of indices (as Assumption \ref{ass:u-ell} holds with $\epsilon = 1$ and  $a=1-k/2$) and denote by  $w_f \in \{\pm 1\}$ the Atkin-Lehner eigenvalue of $f$.
\begin{theorem}\label{thm:gen-heegner-cycles}
	We have that $\{\cbf(n)_\gen\}_{n \in \squarefreeadmissible} \in \ES(T, \lcond_\bk, \admissible, \abf)$. Moreover, there is a modified universal Kolyvagin system $\koly_\gen \in\KSuni(\T,\lcond_\bk,\admissible',\{\chi_{n,\ell}\})$ such that $\koly(1)_\gen=\cbf_{\gen, K}$, where $\chi_{n,\ell}$ is the multiplication by $v_{n,\ell} = \e_n\frac{\e_n(\ell+1)-\ubf_\ell\,\abf_\ell}{\e_n(\ell+1)-\abf_\ell}\in \calO_\p^\times$.
\end{theorem} 

\begin{proof}
	We already noted that $\cbf(n)_\gen \in \hone_{\lcond_\bk}(K[n], T) = \honef(K[n], T)$; \ref{E1} and \ref{E2} follow from \cite[Proposition 7.4]{castella-hsieh:heegner-cycles}. Let $\kappa(n)_{\gen,s}$ as in Definition \ref{def:kolyvagin-derivative}, for $\cbf(n) = \cbf(n)_\gen$. To prove the second claim is enough, by Remark \ref{remark:stressed-local-conditions}, to show that $\loc_v \kappa(n)_{\gen, s} \in \honef(K_v, T/p^sT)$ for $v \mid Np$. For $v\mid p$, see the proof of  \cite[Proposition 7.6]{castella-hsieh:heegner-cycles}; for $v \mid N$ the same proof of the case $v \nmid nNp$ of  Proposition \ref{prop:derived-classes-in-the-Selmer} works, since our assumption on the triviality of the Tamagawa numbers implies that $\honef(F, M) = \hone_\ur(F, M)$ for $M = T, A, T/p^s T$ and $F = K_v, K[n]_{v_n}$, for any $v_n \mid v$. In fact, since $\hone(\inertia_v, T)_{\calO_\p-\tors}$ is the kernel of the map induced in cohomology by the inclusion $T \subseteq V$, by the definitions of $\honef(K_v, T)$ and $\hone_\ur(K_v, A)$, we have the short exact sequence 
	\[
	\begin{tikzcd}
		0 \ar[r] & \hone_\ur(K_v, T) \ar[r] & \honef(K_v, T) \ar[r] & \hone(\inertia_v, T)^{\Fr_v=1}_{\calO_\p-\tors} \ar[r] & 0,
	\end{tikzcd}
	\]
	hence $[\honef(K_v, T) : \hone_\ur(K_v, T)] = \lvert  \hone(I_v, T)^{\Fr_v=1}_{\calO_\p-\tors} \rvert = 1$. Moreover, by \cite[Lemma I.3.5]{rubin:euler-systems}, there is a short exact sequence of finite $\calO_\p$-modules
	\[
	\begin{tikzcd}
		0 \ar[r] & \calW_{K_v}^{\Fr_v = 1} \ar[r] & \calW_{K_v} \ar[r, "\Fr_v-1"] & \calW_{K_v} \ar[r] & \calW_{K_v}/(\Fr_v-1)\calW_{K_v} \ar[r] & 0,
	\end{tikzcd}
	\]
	where $\calW_{K_v} := A^{\inertia_v}/(A^{\inertia_v})_\div$, such that $\lvert \calW_{K_v}^{\Fr_v = 1} \rvert= [\honef(K_v, T) : \hone_\ur(K_v, T)]$ and moreover $\lvert \calW_{K_v}/(\Fr_v-1)\calW_{K_v} \rvert = [\hone_\ur(K_v, A) : \honef(K_v, A)]$. The exactness of the sequence  implies in particular that these two indices coincide and therefore are both $1$. The claimed equality follows when $F=K_v$. For $F=K[n]_{v_n}$ we apply again \emph{loc.~cit.}; in fact $K[n]$ is unramified at $v_n\mid N$ as $(n, N)=1$ and hence $\calW_{K_v} = \calW_{K[n]_{v_n}}$.
	
	Finally the description of $\chi_{n, \ell}$ follows from Remark \ref{rk:E3}, since \ref{E3} holds for $w = w_f$ by \cite[Proposition 7.4]{nekovar:chow-groups}.
\end{proof}

\begin{corollary}
	If $\cbf_{\gen, K} \ne 0$, then $\honef(K, T)$ is free of rank $1$ over $\calO_\p$ and there is a finite $\calO_\p$-module $M$ such that 
	\[
	\honef(K, A) \cong (F_\p/\calO_p) \oplus M \oplus M.
	\]
	Furthermore, we have that
	$
	\length_{\calO_\p}(M) \le \length_{\calO_\p}\bigl(\honef(K, T)/\calO_\p \cbf_{\gen, K}\bigr).
	$
\end{corollary}

\begin{proof}
	This is Theorem \ref{thm:howard-abstract} applied to $\koly_\gen$. In fact, notice that the conditions (H.2)-(H.5) of  \cite[Section 1.3]{howard:heegner-points} are satisfied for $(T, \lcond_\bk, \admissible)$: this is proved in \cite[Section 5.1]{longo-vigni:generalized}.
\end{proof}

\begin{remark}\label{rk:improved-results-mod-forms}
	With some ad hoc versions of Kolyvagin's method, one can improve the latter result, as done for instance in \cite[Theorem 0.2]{mastella:vanishing-sha-mod-forms}. See also \cite[Theorem 1.2]{besser:finiteness-sha} and \cite[Theorem 1.3]{masoero:sha} for analogous results using classical Heegner cycles.  
\end{remark}

\begin{remark}
	If the Tamagawa numbers are non-trivial, one needs to change the definition of $\kappa(n)_{\gen,s}$ by multiplying by a suitable power of $p$ (see \cite[Lemma 10.1, Proposition 10.2(3)]{nekovar:chow-groups}).
\end{remark}

\begin{remark}
	By definition, every $\ell\in\admissible_s'$ satisfied $\ell\equiv-1\bmod p^s$. Therefore, if $k \equiv 2 \mod 4$ then  $u(\ell) = \ell^{1-k/2} \equiv 1 \bmod p^s$ and hence $\chi_{n, \ell} = \Fr_\ell$ (or, in other terms, $v_{n, \ell} = \e_n$) is independent of $\ell$. In this case we can modify $\koly_\gen$ into an $\{\id\}$-universal Kolyvagin system  
	exactly as done in Remark \ref{rk:elliptic-curves-id-kolyvagin-system} in the elliptic curve case (i.e., when $k=2$).
	Instead, when $k\equiv 0\pmod 4$ the units $v_{n,\ell}$ depend both on $n$ and $\ell$ and the argument used above doesn't work anymore.
\end{remark}

Suppose now that $f$ is a \emph{$p$-ordinary} modular form, i.e., that $a_p \in \calO_\p^\times$. By \cite[Theorem 2.2.2]{wiles:lambda-adic}, for any prime $v\mid p$ of $K$ there is a $G_{K_v}$-\hspace{1pt}stable free sub\hspace{1pt}-$\hspace{1pt}\calO_\p$-module $F_v^+(T) \subseteq T$ of rank $1$.  
Let moreover $\Lambda^\ac_{\calO_\p} = \calO_\p\llbracket\Gamma^\ac \rrbracket$ be the anticyclotomic Iwasawa algebra over $\calO_\p$ and consider the $G_K$-representations $\T^\ac = T \otimes_{\calO_\p} \Lambda_{\calO_\p}^\ac$ and $\A^\ac = \T^\ac \otimes_{\Lambda^\ac_{\calO_\p}} (\Lambda^\ac_{\calO_\p})^\vee$, on which the choice of $F_v^+(T)$ induces a Greenberg Selmer structure $\lcond_\Gr$ over $K$, as in the case of elliptic curves. It turns out again that we have pseudoisomorphisms of $\Lambda_{\calO_\p}^\ac$-modules
\[
\hone_{\lcond_{\Gr}}(K, \T^\ac) \sim \varprojlim_n \honef(K_n, T) \quad \text{and} \quad \hone_{\lcond_{\Gr}}(K, \A^\ac)
\sim \varinjlim_n \honef(K_n, A).
\]
For any $np^\alpha \in \squarefreeadmissible^{(p)}$, define $\cbf_{\gen, \alpha, n} = \Cor^{K[np^{\alpha+1}]}_{K_\alpha}\bigl(\cbf(np^{\alpha+1})_\gen \bigr) \in \honef(K_\alpha[n], T)$. These classes satisfy the norm relations (see \cite[Lemma 4.1]{longo-vigni:generalized})
\[
\Cor^{K_{\alpha +1}[n]}_{K_\alpha[n]}(\cbf_{\gen, \alpha+1, n}) = a_p \cbf_{\gen, \alpha, n} - p^{k-2} \res^{K_\alpha[n]}_{K_{\alpha-1}[n]}\cbf_{\gen, \alpha-1, n}
\]
for $\alpha \ge 1$, which Longo and Vigni, inspired by the work of Howard \cite{howard:heegner-points} and Perrin-Riou \cite{perrin-riou:iwasawa-heegner}, used to build another set of cohomology classes $\bbf(np^{\alpha})_\gen\in \honef(K_\alpha[n], T)$ (denoted in \emph{loc.~cit.}~by $\beta_m[n]$, where $m=\alpha$) that are compatible with respect to $\Cor^{K_{\alpha+1}[n]}_{K_\alpha[n]}$.

\begin{theorem}
	We have that $\{\bbf(np^\ga)_\gen\}_{np^\ga\in\squarefreeadmissible^{(p)}}\in\ES^{(p)}(T,\lcond_{\bk},\admissible, \abf)$. Moreover, there is $\koly^{\ac}_\gen \in\KSuni(\T^{\ac},\lcond_{\Gr},\admissible', \{\chi_{n, \ell}\})$ such that $\koly(1)^\ac_{\gen, s}=\bbf^\ac_{\gen, K}$, where the automorphisms $\chi_{n, \ell}$ are those defined in Theorem \ref{thm:gen-heegner-cycles}.
\end{theorem}

\begin{proof}
	We already observed that $\bbf(np^\alpha) \in \honef(K_{\alpha}[n], T)$; note moreover the conditions \ref{pE0}-\ref{pE2} follow by \cite[Lemma 4.1, Proposition 4.9]{longo-vigni:generalized}), whence the first claim. Let $\kappa(n)_{\gen, s, \alpha}$ be the classes of Definition \ref{def:kolyvagin-derivative-iwasawa} for $\bbf(np^\alpha) = \bbf(np^\alpha)_\gen$. The second claim follows from Theorem \ref{th:Iwasawa-euler-to-kolyvagin-system} and Remark \ref{rk:stressed-local-condition-iwasawa}, since  $\loc_{v} \kappa(n)^\ac_\hp \in \hone_{\lcond_\Gr}(K, \T^\ac_{s, \alpha})$ for $v \mid Np$ by  \cite[Lemma 4.12]{longo-vigni:generalized}.
\end{proof}

As in the elliptic curve case, the existence of a $p\hspace{1pt}$-complete anticyclotomic Euler system has as corollary one divisibility of Perrin Riou's formulation of the Iwasawa main conjecture.
\begin{corollary}[Longo--Vigni]
	$\hone_{\lcond_\Gr}(K, \T^\ac)$ is a torsion-free $\Lambda^\ac_{\calO_\p}$-module of rank $1$. Moreover, there is a finitely generated torsion $\Lambda^\ac_{\calO_\p}$-module $M$ such that
	\begin{enumerate}[label=\emph{(\roman*)}]
		\item $\Char(M)=\Char(M)^\iota$;
		\item $\hone_{\lcond_\Lambda}(K, \A^\ac)^\vee \sim \Lambda^\ac_{\calO_\p} \oplus M \oplus M$;
		\item $\Char(M) \supseteq \Char(\hone_{\lcond_\Gr}(K, \T^\ac)/\Lambda^\ac_{\calO_\p}\bbf^\ac_{\gen, K})$.
	\end{enumerate}
\end{corollary}

\begin{proof}
	This is Theorem \ref{th:howard-abstract-iwasawa} for  $\koly^\ac = \koly_\gen^\ac$, as $\koly_\gen^\ac(1) \ne 0$ by \cite[proof of Theorem 4.18]{longo-vigni:generalized}.
\end{proof}

Recently Longo, Pati and Vigni announced a proof of the other divisibility in the last point of the previous theorem, under mild assumptions \cite{longo-pati-vigni:perrin-riou-imc-modular-forms}.

\begin{remark}
	Under the generalized Heegner hypothesis, classical and generalized Heegner cycles are built respectively in \cite{elias-devera:cm:cycles} and \cite{brooks:shimura-curves-l-functions}. Analogues of the theorems listed in this sections using these more general objects might be found also in \cite{magrone:generalized-quaternionic}, \cite{pati:IMC}, \cite{lei-mastella-zhao:bloch-kato}. 
\end{remark}

\subsection{Hida families}

Let now $f$ be a cuspidal eigenform of weight $k$ (possibly odd) and level $\Gamma_0(N)\cap\Gamma_1(p)$. Keep the same notation as in the previous section for the Fourier coefficients and the Hecke field of $f$. We assume that $f$ is an \emph{ordinary $p$-stabilized newform}, in the sense that $a_p\in\calO_\p^\times$ and the conductor of $f$ is divisible by $N$. We also suppose that the residual representation attached to $f$ is irreducible and that $p>3$.

In this setting, Hida \cite{Hid86a},\cite{Hid86b} built a $G_\Q$-representation $\T$ that is associated with a $p$-adic family of modular forms passing through $f$. The representation $\T$ is a free module of rank 2 over a local normal domain $\calR$, equipped with a finite and flat map 
\begin{equation*}
	i\colon \gL:=\calO_\p\llbracket 1+p\Zp\rrbracket\longrightarrow\calR,
\end{equation*}
whose \emph{arithmetic primes} (see \cite[Definition 2.1.1]{H06}) parametrize the modular forms that lie in the family. As done in \cite[Definition 2.13]{H06}, one defines a \emph{critical character} $\Theta\colon G_\Q\to\gL^\times$, so that the twist $\T^\dag$ of $\T$ via $\Theta^{-1}$ has a perfect alternating $G_\Q$-invariant pairing $\T^\dag\times\T^\dag\to\calR(1)$.

There is a set of elements $\{\omega_s\}_{s>0}$ of $\calR$ such that $\omega_s\mid \omega_{s+1}$ and $i([1+p^s\Zp])\subseteq (\omega_s)$ (see \cite[Proposition 1.4.3]{NP}), denoting by square brackets the group elements of $\gL$. Let $J_s=(\omega_s)\subseteq \calR$, so that
\begin{equation*}
	\Rs:=\calR/(p^{s_1},\omega_{s_2})\quad\text{and}\quad \Ts:=\T^\dag\otimes_\calR\Rs
\end{equation*}
for every $\sfrak=(s_1,s_2)\in\Z_{>0}^2$. One can show that $\varprojlim_\sfrak \Rs=\calR$ and that $\calR/(\omega_s)$ is finite and free over $\calO_\p$ (see \cite[Lemma 4.1.2]{zerman:phd-thesis}). If we assume that the image of $G_\Q$ in $\Aut_\calR(\T\hspace{1pt})$ contains the scalars $1+p\Zp$ (see \cite{conti-lang-medved:big-image} for such big image assumptions), we then obtain that the couple $(\calR,\T^\dag)$ satisfies Assumptions \ref{ass:assumptions-on-R} and \ref{ass:assumptions-on-T}. Indeed, $\T^\dag/\m_\calR\T^\dag$ is irreducible since it is equivalent (up to finite base change) to the residual representation attached to $f$ (see \cite[Theorem 5.4]{longo-vigni:quaternion-2011}); an explicit computation using \cite[Proposition 2.1.2]{H06} yields that the determinant of the action of $\Fr_\ell$ on $\T^\dag$ is $\ell$; using the perfect pairing defined above one can show that the determinant of the action of the complex conjugation $\tau_c$ on $\T^\dag$ is $-1$, yielding also Assumption \ref{ass:assumptions-on-T} \ref{condition:assumption-eigenvalues-complex-conjugation}.

For every place $v\mid p$ of $K$, the representation $\T^\dag$ comes equipped with an exact sequence of $\calR\llbracket G_{K_v}\rrbracket$-modules
\begin{equation}\label{eq:Greenberg-selmer-Hida}
	0\to F_v^+(\T^\dag)\to\T^\dag\to F_v^-(\T^\dag)\to 0,
\end{equation}
where $F_v^\pm(\T^\dag)$ is free of rank 1 over $\calR$. Then, we can define the \emph{Greenberg Selmer structure} $\lcond_{\Gr}$ on $\T^\dag$ over $K$ via the local conditions
\begin{equation}\label{eq:def-greenberg-conditions-hida}
	\hone_{\lcond_\Gr}(K_v,\T^\dag):=\begin{cases}
		\honeur(K_v,\T^\dag) &\text{if $v\nmid p$;}\\
		\Imm\big(\hone(K_v,F_p^+(\T^\dag))\to \hone(K_v,\T^\dag)\big) &\text{if $v\mid p$}.
	\end{cases}
\end{equation}
The ring $\calR$ is built as a quotient of a big Hecke algebra acting on a tower of modular curves (see \cite[Section 2.1]{H06}), therefore we can define $\abf_\ell\in\calR$ to be the image of the Hecke operator $T_\ell$ and $\abf_p\in\calR$ to be the image of $U_p$. Set $\abf:=\{\abf_\ell\}_{\ell\in\admissible}$.

In this setting, Howard \cite{H06} built the ($p\hspace{1pt}$-complete) anticyclotomic Euler system of big Heegner points. Interpolating classical Heegner points in towers of modular curves, he constructed certain cohomology classes
\begin{equation*}
	\mathfrak{X}_c\in \hone_{\lcond_\Gr}(K[c],\T^\dag)
\end{equation*}
for every $c\in\N$ coprime with $N$.

With the aim of pursuing Kolyvagin's descent, we assume for simplicity Hypotheses (H.tam) and (H.stz) of \cite{buyukboduk:big-heegner}, which control Tamagawa factors and exceptional zeroes at finite order characters. By \cite[Proposition 2.10]{zerman:kolyvagin-hida-families}, for every $\ell\nmid Np$ one has $\Tr(\Fr_\ell \, | \, \T^\dag)=\Theta^{-1}(\Fr_\ell)\abf_\ell$, so that $\ubf_\ell=\Theta^{-1}(\Fr_\ell)$. As shown in \cite[Lemma 2.21]{zerman:kolyvagin-hida-families}, if $s_1\ge s_2$ we have that $\ubf_\ell$ reduces to $(-1)^\delta$ on $\Ts'$, for some $\delta\in\N$ that only depends on the representation. Therefore, choosing $\Omega$ to be the set of all $(s_1,s_2)\in\Z_{>0}^2$ such that $s_1\ge s_2$, we obtain that the set $\admissible'=\admissible'(\Omega)$ of Definition \ref{def:Kolyvagin-primes} satisfies Assumption \ref{ass:u-ell}. Thus, we can apply the formalism of Section \ref{sec:euler-systems-and-kolyvagin-systems} to obtain the following result.

\begin{theorem}\label{thm:hida-kolyvagin-system}
	We have that $\{\mathfrak{X}_n\}_{n\in\squarefreeadmissible}\in \ES(\T^\dag,\lcond_{\Gr},\admissible,\abf)$. Moreover, there are automorphisms $\chi_{n,\ell}\in\Aut(\T^\dag)$ and $\koly_{\Hid}\in\KSuni(\T^\dag,\lcond_\Gr,\admissible',\{\chi_{n,\ell}\})$ such that $\koly_{\Hid}(1)=\Cor^{K[1]}_K\mathfrak{X}_1$.
\end{theorem}
\begin{proof}
	The first claim is a consequence of \cite[Propositions 2.3.1 and 2.3.2]{H06}. By Theorem \ref{th:euler-to-kolyvagin-system}, we find $\koly_{\Hid}\in\KSuni(\T^\dag,\lcond_\rel,\admissible',\{\chi_{n,\ell}\})$ such that $\koly_{\Hid}(1)=\Cor^{K[1]}_K\mathfrak{X}_1$. As explained in Remark \ref{remark:stressed-local-conditions}, we're just left to show that $\gk_{\Hid}(n)_\sfrak\in \hone_{\lcond_\Gr}(K,\Ts)$ for every $n\in\squarefreeadmissible'$. This can be proved exactly as in \cite[Section 4.3]{buyukboduk:big-heegner}.
\end{proof}

If, instead, we take $c=np^\ga$ with $n\in\squarefreeadmissible$ and $\ga\ge 0$, the set of elements
\begin{equation*}
	\bbf(np^\ga):=\Cor_{K_\ga[n]}^{K[np^\ga]} \abf_p^{-\ga} \mathfrak{X}_{np^\ga}\in \hone(K_\ga[n],\T^\dag)
\end{equation*}
lies in $\ES^{(p)}(\T^\dag,\lcond,\admissible,\abf)$, thanks to \cite[Propositions 2.3.1 and 2.3.2]{H06} and \cite[Proposition 4.14]{buyukboduk:big-heegner}. Moreover, one can define the Greenberg Selmer structure on the anticyclotomic twist $\T^{\dag,\ac}:=\T^\dag\otimes_\calR\calR\llbracket\Gamma^\ac\rrbracket$ by tensoring the exact sequence \eqref{eq:Greenberg-selmer-Hida} by $\calR^\ac$ and replacing $\T^\dag$ with $\T^{\dag,\ac}$ in \eqref{eq:def-greenberg-conditions-hida}. Keeping the same notations as above, we obtain the following result, which is a reinterpretation of \cite[Theorem A.1]{buyukboduk:big-heegner}.

\begin{theorem}\label{thm:hida-iwasawa-kolyvagin-system}
	There are $\chi_{n,\ell}\in\Aut(\T^{\dag,\ac})$ and $\koly_{\Hid}^\ac\in\KSuni(\T^{\dag,\ac},\lcond_\Gr,\admissible',\{\chi_{n,\ell}\})$ such that $\koly_{\Hid}^{\ac}(1)=\Cor^{K[1]}_K\varprojlim_\ga\bbf(p^\ga)$.
\end{theorem}
\begin{proof}
	By Theorem \ref{th:Iwasawa-euler-to-kolyvagin-system}, we obtain an element $\koly_{\Hid}^\ac\in\KSuni(\T^{\dag,\ac},\lcond_\rel,\admissible',\{\chi_{n,\ell}\})$ that satisfies the required equality. By Remark \ref{remark:stressed-local-conditions}, we are just left to show that $\gk_{\Hid}^\ac(n)_\sfrak\in \hone_{\lcond_\Gr}(K,\Ts)$ for every $n\in\squarefreeadmissible'$. This is proved in \cite[Section 4.3]{buyukboduk:big-heegner}.
\end{proof}

\begin{remark}
	Big Heegner classes have been built under the generalized Heegner hypothesis in \cite{longo-vigni:quaternion-2011}. In \cite{zerman:kolyvagin-hida-families}, they are used to build a modified universal Kolyvagin system, expanding the arguments of this section.
\end{remark}

Both Theorems \ref{thm:hida-kolyvagin-system} and \ref{thm:hida-iwasawa-kolyvagin-system} can be used to study the structure of Greenberg Selmer groups. However, since $\calR$ is not a discrete valuation ring, these type of application are less staightforward and out of the scope of this article. The interested reader may want to consult \cite{buyukboduk:big-heegner}, \cite{Fou} or \cite{zerman:kolyvagin-hida-families}.

\appendix

\section{Nekov\' a\v r's key formula}\label{sec:formula-abstract}

In this appendix we review and improve the abstract formula of \cite[Section 9]{nekovar:chow-groups}, that we used in the proof of Lemmas \ref{lem:formula-capitolo-2} and \ref{lem:formula-capitolo-3} to find an explicit relation among the derivative classes. Let $p$ and $\ell$ be two distinct odd primes and let $R$ be a complete Noetherian local ring that is finite and free over $\Zp$. Let $\Zhat$ be the profinite completion of $\Z$ and fix a topological generator $\phi$ of $\Zhat$. 

\subsection{Conditions and generalities}\label{sec:setting}
Consider the following objects and conditions, that will be assumed for the rest of this appendix.

\begin{enumerate}[label=(G\arabic*)]
	\item \label{condition:pro-dihedrality} 
	Let $\Gtilde$ be a profinite group and $H\unlhd G\unlhd\Gtilde$ be a chain of normal closed subgroups with $\Gtilde/H=\langle\gs\rangle\rtimes\langle c\rangle$ dihedral, where $\langle\gs\rangle$ is cyclic of order $M$ and $\langle c\rangle$ is cyclic of order $2$ acting on $\langle\gs\rangle$ by $c\gs c=\gs^{-1}$. Moreover, $G/H=\langle\gs\rangle$ and $\Gtilde/G=\langle c\rangle$. 
	\item \label{condition:localization}
	Let $\Gtilde_0$ be a closed subgroup of $\Gtilde$ and call $G_0:=\Gtilde_0\cap G$ and $H_0:=G_0\cap H$. Suppose that the inclusion $\tilde{G}_0\subseteq \tilde{G}$ induces an isomorphism $\tilde{G}_0/H_0\cong \tilde{G}/H$. Call $\sigma_0$ the preimage of $\gs$ and $c_0$ the preimage of $c$, so that $G_0/H_0=\langle\sigma_0\rangle$ and $\tilde{G}_0/G_0=\langle c_0\rangle$. 
	\item \label{condition:semidirect-product}
	The group $\Gtilde_0$ is equipped with a surjective homomorphism $   \pi\colon \Gtilde_0\surj \Zhat$,
	that induces surjections $G_0\surj 2\Zhat$ and $H_0\surj 2\Zhat$.
\end{enumerate}
\begin{enumerate}[label=(T\arabic*)]
	\item \label{condition:structure-T}
	Let $T$ be a finitely generated torsion-free module over $R$ with a continuous action of $\Gtilde$.
	\item \label{condition:unramifiedness}
	$\Gtilde_0$ acts on $T$ through its quotient $\Zhat$. 
	\item \label{condition:frob-square-identity-mod-p-s}
	$\phi^2$ acts as the identity on $T/p^s T$ for some $s\ge 1$ such that $p^s\mid M$.
	\item \label{condition:vanishing-h-0}
	$\ho(H, T/p^s T) = \{0\}$.
\end{enumerate}
\begin{enumerate}[label=(UR)]
	\item \label{condition:unramified-classes} Let $\cohoclass{x}\in \hone(G,T)$ and $\cohoclass{y}\in \hone(H,T)$ be classes whose restriction to $G_0$ and $H_0$ lie in the image of the inflation maps $\hone(2\Zhat,T)\to \hone(G_0,T)$ and $\hone(2\Zhat,T)\to \hone(H_0,T)$, respectively.
\end{enumerate}
\begin{enumerate}[label=(COR)]
	\item The classes $\cohoclass{x}$ and $\cohoclass{y}$ satisfy $\Cor_{G}^H(\cohoclass{y})=M_1 \cohoclass{x}$, for some $M_1\in R$ divisible by $p^s$.\label{condition:corestriction}
\end{enumerate}

\begin{enumerate}[label=(E--S)]
	\item Assume that $\res_{H}^{H_0}(\mathbf{\ybar})=\res_G^{H_0}(\phi (\mathbf{\xbar}))$, denoting by $\mathbf{\xbar}$ and $\mathbf{\ybar}$ the image of $\mathbf{x}$ and $\mathbf{y}$ in $\hone(G, T/p^sT)$ and $\hone(H, T/p^sT)$, respectively. \label{condition:eichler-shimura}
\end{enumerate}
\begin{enumerate}[label=(FR)]
	\item $\phi^2-\delta M_1\phi+d=0$ on $T$ for some $\gd\in R^\times$ and $d\in R$.\label{condition:char-poly-frob}
\end{enumerate}

Let now $x\in Z^1(G_0,T)$ be a $1$-cocycle representing $\cohoclass{x}$.  Condition (UR) implies that $x$ is inflated by a 1-cocycle $\tilde{x}\in Z^1(2\Zhat,T)$, up to summing a coboundary. Hence, for every $g_0 \in G_0$ such that $\pi(g_0) = \phi^{2v}$ (for some $u\in\Z_{\ge 0}$), there is $b_x\in T$ such that
\begin{equation}\label{eq:definition-ax}
	x(g_0) = \tilde{x}(\phi^{2v})+(g_0-1)b_x=(1+\phi^2+\dots+\phi^{2(v-1)})a_x+(\phi^{2v}-1)b_x 
\end{equation}
where $a_x=\tilde{x}(\phi^2)\in T$, since $g_0$ acts as $\phi^{2v}$ on $T$ by assumption \ref{condition:unramifiedness}. Note that the cohomology class $[\tilde{x}\hspace{1pt}]$ corresponds to the class of $a_x$ in $T/(\phi^2-1)T$ via the isomorphism
\begin{equation*}
	\hone(2\Zhat,T)\xlongrightarrow{\cong} T/(\phi^2-1)T
\end{equation*}
induced by evaluating cocycles at $\phi^2$ (see \cite[Lemma B.2.8]{rubin:euler-systems}). The same argument works also for $y\in Z^1(H_0,T)$ representing $\cohoclass{y}$ and $g_0 \in H_0$ such that $\pi(g_0) = \phi^{2v}$, for some $v\in\Z_{\ge 0}$, yielding
\begin{equation}\label{eq:definition-ay}
	y(g_0) = \tilde{y}(\phi^{2v})+(g_0-1)b_y=(1+\phi^2+\dots+\phi^{2(v-1)})a_y+(\phi^{2v}-1)b_y. 
\end{equation}

\subsection{Kolyvagin's derivative}\label{sec:kolyvagin-derivative-abstract}

In this section, we introduce the notion of Kolyvagin's derivative as a purely group-theoretic construction, following \cite[Section~7]{nekovar:chow-groups}. As a piece of notation, we will always use the upper bar to denote the image in $T/p^sT$ of cocycles, cohomology classes and elements of $T$.

First, notice that $\Cor^H_G \cohoclass{\ybar}=0$ in $\hone(G,T/p^sT)$, by \ref{condition:corestriction}. Define the operators
\[
\Tr := \sum_{i=0}^{M-1}\sigma^i\quad\text{and}\quad D:=\sum_{i=1}^{M-1}i\gs^i
\]
in $\Z[G/H]$, satisfying the identity 
\begin{equation}\label{eq:telescopic-identity-appendix}
	(\sigma-1)D = M - \Tr.
\end{equation}
For future reference, we now make some computations at the level of cocycles. Fix a lifting $\tilde{\gs}$ of $\gs$ to $G$. From the explicit definition of corestriction (see e.g.~\cite[Section~I.5.4]{neukirch:cnf}), it follows that $\Cor^H_G\cohoclass{\ybar}$ can be represented by the cocycle $\Cor \ybar\in Z^1(G,T/p^sT)$ defined as
\begin{equation}\label{equ:corestriction-cocycle}
	(\Cor \ybar)(\tilde{\gs}) := \ybar(\tilde{\gs}^M)\quad \text{and}\quad (\Cor \ybar)(h) := \sum_{i=0}^{M-1}\tilde{\gs}^{i}\ybar(\tilde{\gs}^{-i}h\tilde{\gs}^i), 
\end{equation}
for every $h\in H$. Moreover, since $\Cor_{G}^H\cohoclass{\ybar}=0$, then $\Cor \ybar$ is a coboundary, i.e.,
\begin{equation}
	\label{equ:definition-of-a}
	(\Cor \ybar)(g)=(g-1)\abar
\end{equation}
for every $g\in G$ and for a unique (by the assumption \ref{condition:vanishing-h-0}) $\abar\in T/p^sT$. Consider also the cocycle $D\ybar\in Z^1(H,T/p^sT)$ defined as
\begin{equation}\label{eq:definition-cocycle-Dy}
	(D\ybar)(h):= \sum_{i=1}^{M-1}i\tilde{\gs}^{i}\ybar(\tilde{\gs}^{-i}h\tilde{\gs}^i),
\end{equation}
for every $h\in H$, that represents the cohomology class $D\cohoclass{\ybar}$. Note that, since $M\cohoclass{\ybar}=0$ by \ref{condition:frob-square-identity-mod-p-s} and moreover $\Tr (\cohoclass{\ybar}) = (\res_G^H \circ \Cor^H_G) (\cohoclass{\ybar})= 0$ by \ref{condition:corestriction}, 
it follows from \eqref{eq:telescopic-identity-appendix} that $D\cohoclass{\ybar} \in \hone(H, T/p^sT)^{G/H}$. In particular, by \ref{condition:vanishing-h-0}, applying the inflation-restriction sequence we obtain that $D\cohoclass{\ybar}$ lies in the image of $\res_G^H$. At level of cocycles, we can be more precise.

\begin{proposition}\label{prop:def-f-y}
	There is a cocycle $f_{\ybar}\in Z^1(G,T/p^s T)$ uniquely determined by the conditions 
	\begin{equation*}
		\res_G^H f_{\ybar}=D\ybar \quad \text{and} \quad  f_{\ybar}(\sigmati) = -\sigmati \abar,
	\end{equation*}
	where $\abar$ is the element defined in \eqref{equ:definition-of-a}. Moreover, its class $[f_{\ybar}]\in \hone(G,T/p^s T)$ is independent of the choice of the representative $\ybar$ of $\cohoclass{y}$.
\end{proposition}
\begin{proof}
	Follow the same steps as in \cite[Section 7]{nekovar:chow-groups}.
\end{proof}

\subsection{The key formula}\label{sec:proof} 

We are now able to prove the key formula that links $\cohoclass{x}=[x]$ and $\cohoclass{y}=[y]$, following \cite[Section 9]{nekovar:chow-groups}. Recall the elements $a_x$ and $a_y$ defined in \eqref{eq:definition-ax} and \eqref{eq:definition-ay}.

\begin{proposition}\label{prop:key-formula-appendix}
	The relation
	\begin{equation*}
		\biggl(\frac{ M}{p^s}\phi-\frac{M_1}{p^s}\biggr)\abar_x=\biggl(\frac{\delta M_1}{p^s}\phi-\frac{d+1}{p^s}\biggr)\abar.
	\end{equation*}
	holds in $T/p^s T$ and $\abar=-f_{\ybar}(\sigmati_0)$ for any lift $\sigmati_0$ of $\sigma_0$ to $G_0$.
\end{proposition}
\begin{proof}
	Choose a lift $\sigmati_0 \in G_0$ of $\sigma_0$ such that $\pi(\sigmati_0)$ is trivial in $2\Zhat$ (see \ref{condition:semidirect-product}). 
	By \ref{condition:localization}, $\sigmati_0$ is also a lift of $\sigma \in G/H$.
	In particular, the cocycle $f_{\ybar}\in Z^1(G,T/p^s T)$ defined in Proposition~\ref{prop:def-f-y} has the following properties: 
	\begin{equation*}
		\res^H_G(f_{\ybar})=D\ybar\in Z^1(H,T/p^s T), \qquad f_{\ybar}(\sigmati_0) = - \sigmati_0 \abar = - \abar,
	\end{equation*}
	as $\sigmati_0$ acts trivially on $T$ and $T/p^s T$ by \ref{condition:unramifiedness}, where $\abar \in T/p^sT$ is the element defined in \eqref{equ:definition-of-a}. We now show that this element is independent on the choice of the lift $\sigmati_0$ of $\sigma_0$.
	
	Consider then the following commutative diagram
	\begin{equation}\label{equ:inf-res-H0-G0}
		\begin{tikzcd}
			&& {\hone(G,T/p^sT)} \ar[r, "\res"] \ar[d, "\res"] & {\hone(H,T/p^sT)} \ar[d, "\res"] \\
			0 \ar[r] & {\hone(G_0/H_0,T/p^sT)} \ar[r, "\infl"] & {\hone(G_0,T/p^sT)} \ar[r, "\res"] & {\hone(H_0,T/p^sT)},
		\end{tikzcd}
	\end{equation}
	where the bottom row is exact (here we used the fact that $\ho(H_0, T/p^sT) = T/p^s T$ since the action of $H_0$ is trivial on $T/p^s T$ by \ref{condition:unramifiedness} and \ref{condition:frob-square-identity-mod-p-s}). It follows from the following lemma that $\res^{H_0}_G([f_{\ybar}])=\res_H^{H_0} D\bar{\cohoclass{y}}$ is trivial and hence we find an element 
	\[
	[z_0]\in \hone(G_0/H_0,T/p^s T)=\Hom(G_0/H_0,T/p^s T)
	\]
	such that $\infl_{G_0}^{H_0}z_0=\res^{G_0}_G f_{\ybar}$.
	In particular
	\begin{equation}\label{eq:proprieta-elemento-a}
		z_0(\sigma_0)=(\res^{G_0}_G f_{\ybar})(\sigmati_0)=f_{\ybar}(\sigmati_0)=-\abar,
	\end{equation}
	therefore the element $\abar\in T/p^s T$ only depends on $\sigma_0$.
	
	\begin{lemma}\label{lemma:D=0}
		$\res_H^{H_0} \bigl(D \bar{\cohoclass{y}}\bigr) = 0$.
	\end{lemma}
	
	\begin{proof}
		Recall that the class $D \bar{\cohoclass{y}}$ is represented by the cocycle $D\ybar$ defined in \eqref{eq:definition-cocycle-Dy}. For any $h_0\in H_0$ with $\pi(h_0) = \phi^{2v}$ for some $v\in\Zhat$, by \eqref{eq:definition-ay} it follows that 
		\begin{equation*}
			\ybar(h_0)=\bar{\tilde{y}}(\phi^{2v})+(\phi^{2v}-1)b_x=\ybar(\sigmati_0^{-i}h_0\sigmati_0^i),
		\end{equation*}
		for any $i=1,\dots, M-1$, as $\pi(h_0)=\pi(\sigmati_0^{-i}h_0\sigmati_0^i)$. Hence,
		\[
		D\ybar(h_0) = \sum_{i=1}^{M-1} i \sigmati_0^i \ybar(h_0) = \biggl( \sum_{i=1}^{M-1} i \biggr)\ybar(h_0) = \biggl(\frac{M(M-1)}{2}\biggr) \ybar(h_0) = 0,
		\]
		since $\sigmati_0$ acts trivially on $T/p^sT$ by \ref{condition:unramifiedness} and $M \equiv 0  \bmod p^s$ by \ref{condition:frob-square-identity-mod-p-s}.
	\end{proof}
	
	According to assumption \ref{condition:corestriction}, there is an element $a\in T$ such that
	\begin{equation*}
		(\Cor(y))(g)-M_1 x(g)=(g-1)a,
	\end{equation*}
	for every $g\in G$, where $\Cor(y)$ is a cocycle defined as in \eqref{equ:corestriction-cocycle}, with $y$ in place of $\ybar$. When quotienting by $p^s T$, the term $M_1 x(g)$ vanishes and hence  $a + p^sT = \abar$ in $T/p^s T$ by \eqref{equ:definition-of-a}. Restricting to $g = h_0\in H_0$, we get the relation
	\begin{equation*}
		\sum_{i=0}^{M-1} y(\sigmati_0^{-i} h_0\sigmati_0^i)-M_1 x(h_0)=(h_0-1)a,
	\end{equation*}
	as elements of $T$. Now note that $y(\sigmati_0^{-i} h_0\sigmati_0^i) = y(h_0)$ for any $i\ge0$ by the same argument we applied to $\ybar$ in the proof of Lemma \ref{lemma:D=0}. Hence the previous relation becomes
	\begin{equation*}
		My(h_0)-M_1 x(h_0)=(h_0-1)a.
	\end{equation*}
	Putting this equation together with \eqref{eq:definition-ax} and \eqref{eq:definition-ay} for $g_0 = h_0$ and $v = 1$ we obtain
	\begin{equation*}
		Ma_y-M_1a_x=(\phi^2-1)(a-Mb_y+M_1b_x)
	\end{equation*}
	on $T$. 
	Using assumption \ref{condition:char-poly-frob} and recalling that $p^s$ divides $M$ and $M_1$, we obtain the relation
	\begin{equation*}
		M a_y - M_1 a_x=(\delta M_1\phi-(d+1))(a-p^s t),
	\end{equation*}
	for some $t\in T$. Since $T$ is torsion free over $R$ and all terms are divisible by $p^s$ in $T$, we can divide both sides by $p^s$, yielding the formula
	\begin{equation*}
		\biggl(\frac{M}{p^s}a_y-\frac{M_1}{p^s}a_x\biggr)=\biggl(\frac{\delta M_1}{p^s}\phi-\frac{d+1}{p^s}\biggr)(a-p^s t).
	\end{equation*}
	By quotienting modulo $p^s$, assumption \ref{condition:eichler-shimura} implies that $\abar_y = \phi (\abar_x)$, and therefore we obtain the claimed relation.
\end{proof}

\section{Semi-local Galois cohomology}\label{sec:semi-local-cohomology}

The work of Sections \ref{sec:euler-systems-and-kolyvagin-systems}--\ref{sec:Iwasawa-theory} relies on a nice interplay between \emph{global} and \emph{local} operators, i.e., operators on global cohomology groups and on their localizations. However, it frequently happens that an operator between global cohomology groups does not restrict to an operator between the corresponding local cohomology groups. In this appendix, inspired by \cite[Appendix B.4--B.5]{rubin:euler-systems}, we study how one can solve this problem using semi-local Galois cohomology.

Let $L/F$ be a Galois extension of number fields and let $v$ be a prime of $F$. For every prime $w$ of $L$ above $v$, fix a prime $\mathfrak{w}$ of $\bar{F}$ above $w$ and let $\inertia_{w}\subseteq\decomp_w\subseteq G_F$ denote the inertia and the decomposition of $\mathfrak{w}/v$. Fix a prime $w_0$ of $L$ above $v$ and write $\decomp=\decomp_{w_0}$ and $\inertia=\inertia_{w_0}$. For every $w\mid v$, choose an element $\sigma_w\in G_F$ such that $w=w_0\circ\gs_w$, so that $\decomp_w=\sigma_w^{-1}\decomp\sigma_w$.

Let $T$ be a discrete $G_F$-module. In the following, we need to describe local cohomology groups as cohomology groups of some explicit decomposition groups. In order to do that, we set $\hone(F_v,T):=\hone(\decomp,T)$, $\hone(L_{w},T):=\hone(\decomp_{w}\cap G_L,T)$, $\hone(F_v^{\ur},T):=\hone(\inertia,T)$ and $\hone(L_w^{\ur},T):=\hone(\inertia_{w}\cap G_L,T)$.  In this way, localization maps from global to local cohomology groups are just restrictions to a decomposition group. Define semi-local cohomology groups 
\begin{equation*}
	\hone(L_v,T):=\bigoplus_{w\mid v} \hone(L_w,T)\quad\text{and}\quad  \hone(L_v^{\ur},T):=\bigoplus_{w\mid v}\hone(\inertia_w\cap G_L,T),
\end{equation*}
and denote semi-localization maps as $\loc_v:=\bigoplus_{w\mid v}\loc_w\colon \hone(L,T)\longrightarrow \hone(L_v,T)$. 

\subsection{Galois action} \label{subsec:Galois-action}

By using the canonical identifications
\begin{equation*}
	\hone(L_v,T)\cong \hone(L,\Ind_{\decomp}^{G_F}(T))\quad\text{and}\quad  \hone(L_v^{\ur},T)\cong \hone(L,\Ind_{\inertia}^{G_F}(T))
\end{equation*}
coming from \cite[Propositions 4.2 and 5.1]{rubin:euler-systems}, one has a well defined action of $\Gal(L/F)$ on both modules and a commutative diagram
\begin{equation*}
	\begin{tikzcd}
		{\hone(L,T)} & {\hone(L_v,T)} & {\hone(L_v^{\ur},T)} \\
		{\hone(L,T)} & {\hone(L_v,T)} & {\hone(L_v^{\ur},T)},
		\arrow["{\loc_v}", from=1-1, to=1-2]
		\arrow["\gs", from=1-1, to=2-1]
		\arrow["\res", from=1-2, to=1-3]
		\arrow["\gs", from=1-2, to=2-2]
		\arrow["\gs", from=1-3, to=2-3]
		\arrow["{\loc_v}", from=2-1, to=2-2]
		\arrow["\res", from=2-2, to=2-3]
	\end{tikzcd}
\end{equation*}
for every $\sigma\in\Gal(L/F)$. 

\subsection{Corestriction} \label{app:corestriction}

As an immediate consequence of the esplicit description of \cite[Proposition 1.5.6]{neukirch:cnf}, we obtain the two unnamed vertical arrows that make the diagram
\begin{equation*}
	\begin{tikzcd}
		{\hone(L,T)} & {\hone(L_v,T)} & {\hone(L_v^{\ur},T)} \\
		{\hone(F,T)} & {\hone(F_v,T)} & {\hone(F_v^{\ur},T)}
		\arrow["{\Cor^L_F}"', from=1-1, to=2-1]
		\arrow["{\loc_v}", from=1-1, to=1-2]
		\arrow["{\loc_v}", from=2-1, to=2-2]
		\arrow[from=1-2, to=2-2]
		\arrow["\res", from=1-2, to=1-3]
		\arrow[from=1-3, to=2-3]
		\arrow["\res", from=2-2, to=2-3]
	\end{tikzcd}
\end{equation*}
commute. In particular, when $v$ is split in $L$, we have that
\begin{equation*}
	\loc_v(\Cor_F^L(\xi))=\hspace{-10pt}\sum_{\gs\in\Gal(L/F)}\hspace{-10pt}\gs \circ \loc_{w_0\circ\gs}(\xi),
\end{equation*}
for every $\xi\in \hone(L,T)$.

\subsection{Shapiro's lemma}\label{app:Shapiro}

Assume now that $T$ is also a $\Zp$-module and let $G:=\Gal(L/F)$. As explained in Section \ref{subsec:Shapiro's-lemma}, there is an  isomorphism $\Sh\colon \hone(L,T)\to \hone(F,T\otimes_{\Zp}\Zp[G])$ induced by Shapiro's lemma. As a consequence of \cite[Proposition B.4.2]{rubin:euler-systems}, Shapiro's lemma extends also to semi-local cohomology with a commutative diagram of the following type
\begin{equation*}
	\begin{tikzcd}
		{\hone(L,T)} & {\hone(L_v,T)} & {\hone(L_v^{\ur},T)} \\
		{\hone(F,T\otimes_{\Zp}\Zp[G])} & {\hone(F_v,T\otimes_{\Zp}\Zp[G])} & {\hone(F_v^{\ur},T\otimes_{\Zp}\Zp[G]),}
		\arrow["{\loc_v}", from=1-1, to=1-2]
		\arrow["{\Sh}"', from=1-1, to=2-1]
		\arrow["\res", from=1-2, to=1-3]
		\arrow["\Sh", from=1-2, to=2-2]
		\arrow["\Sh", from=1-3, to=2-3]
		\arrow["{\loc_v}", from=2-1, to=2-2]
		\arrow["\res", from=2-2, to=2-3]
	\end{tikzcd}
\end{equation*}
where all vertical maps are isomorphisms. Therefore, if $T$ is unramified at $v$, Shapiro's map induces isomorphisms between (semi-local) finite conditions and between (semi-local) singular quotients. For more details, see also \cite[Section 3.1.2]{skinner-urban:Iwasawa-main-conj-for-GL2}.

\printbibliography
 
\end{document}